\documentclass[a4paper,11pt,reqno]{amsart}
\usepackage{amsmath,amsfonts,amsthm,amssymb,color}
\usepackage[T1]{fontenc}
\usepackage{pdfsync}
\usepackage{csquotes}
\usepackage{graphicx}
\usepackage{pstricks}
\usepackage{lmodern}
\usepackage{enumerate}
\usepackage{tikz}

\usepackage{hyperref}
\definecolor{gr}{rgb}   {0.,   0.69,   0.23 }
\definecolor{bl}{rgb}   {0.,   0.5,   1. }
\definecolor{mg}{rgb}   {0.85,  0.,    0.85}
\definecolor{yl}{rgb}   {0.8,  0.7,   0.}
\definecolor{or}{rgb}  {0.7,0.2,0.2}

\newcommand{\mathbbm}[1]{\text{\usefont{U}{bbm}{m}{n}#1}}
\newcommand{\be}{\beta}

\newcommand{\id}{\mbox{Id}}

\newcommand{\1}{{\bf 1}}

\newcommand{\xiti}{\tilde{\xi}}
\newcommand{\etati}{\tilde{\eta}}
\newcommand{\lati}{\tilde{\la}}
\newcommand{\betati}{\tilde{\beta}}

\newcommand{\R}{\mathbb R}

\newcommand{\ca}{\mathcal A}

\newcommand{\cac}{\mathcal C}
\newcommand{\cd}{\mathcal D}
\newcommand{\ce}{\mathcal E}
\newcommand{\cf}{\mathcal F}
\newcommand{\cg}{\mathcal G}

\newcommand{\cj}{\mathcal J}

\newcommand{\cp}{\mathcal P}

\newcommand{\cw}{\mathcal W}

\newcommand{\al}{\alpha}

\newcommand{\ga}{\gamma}
\newcommand{\gga}{\Gamma}
\newcommand{\ka}{\kappa}
\newcommand{\la}{\lambda}

\newcommand{\si}{\sigma}

\newcommand{\vp}{\varphi}

\newtheorem{theorem}{Theorem}[section]

\newtheorem{corollary}[theorem]{Corollary}

\newtheorem{definition}[theorem]{Definition}

\newtheorem{lemma}[theorem]{Lemma}

\newtheorem{proposition}[theorem]{Proposition}

\theoremstyle{remark}
\newtheorem{remark}[theorem]{Remark}

\pgfdeclareshape{crosscircle}
{
  \inheritsavedanchors[from=circle] % this is nearly a circle
  \inheritanchorborder[from=circle]
  \inheritanchor[from=circle]{north}
  \inheritanchor[from=circle]{north west}
  \inheritanchor[from=circle]{north east}
  \inheritanchor[from=circle]{center}
  \inheritanchor[from=circle]{west}
  \inheritanchor[from=circle]{east}
  \inheritanchor[from=circle]{mid}
  \inheritanchor[from=circle]{mid west}
  \inheritanchor[from=circle]{mid east}
  \inheritanchor[from=circle]{base}
  \inheritanchor[from=circle]{base west}
  \inheritanchor[from=circle]{base east}
  \inheritanchor[from=circle]{south}
  \inheritanchor[from=circle]{south west}
  \inheritanchor[from=circle]{south east}
  \inheritbackgroundpath[from=circle]
  \foregroundpath{
    \centerpoint%
    \pgf@xc=\pgf@x%
    \pgf@yc=\pgf@y%
    \pgfutil@tempdima=\radius%
    \pgfmathsetlength{\pgf@xb}{\pgfkeysvalueof{/pgf/outer xsep}}%  
    \pgfmathsetlength{\pgf@yb}{\pgfkeysvalueof{/pgf/outer ysep}}%  
    \ifdim\pgf@xb<\pgf@yb%
      \advance\pgfutil@tempdima by-\pgf@yb%
    \else%
      \advance\pgfutil@tempdima by-\pgf@xb%
    \fi%
    \pgfpathmoveto{\pgfpointadd{\pgfqpoint{\pgf@xc}{\pgf@yc}}{\pgfqpoint{-0.707107\pgfutil@tempdima}{-0.707107\pgfutil@tempdima}}}
    \pgfpathlineto{\pgfpointadd{\pgfqpoint{\pgf@xc}{\pgf@yc}}{\pgfqpoint{0.707107\pgfutil@tempdima}{0.707107\pgfutil@tempdima}}}
    \pgfpathmoveto{\pgfpointadd{\pgfqpoint{\pgf@xc}{\pgf@yc}}{\pgfqpoint{-0.707107\pgfutil@tempdima}{0.707107\pgfutil@tempdima}}}
    \pgfpathlineto{\pgfpointadd{\pgfqpoint{\pgf@xc}{\pgf@yc}}{\pgfqpoint{0.707107\pgfutil@tempdima}{-0.707107\pgfutil@tempdima}}}
  }
}
\makeatother

\colorlet{symbols}{blue!90!black}
\colorlet{testcolor}{green!60!black}

\def\symbol#1{\textcolor{symbols}{#1}}
\def\1{\mathbf{\symbol{1}}}

\usetikzlibrary{shapes.misc}
\usetikzlibrary{shapes.symbols}
\usetikzlibrary{snakes}
\usetikzlibrary{decorations}
\usetikzlibrary{decorations.markings}

\def\drawx{\draw[-,solid] (-3pt,-3pt) -- (3pt,3pt);\draw[-,solid] (-3pt,3pt) -- (3pt,-3pt);}
\tikzset{
	root/.style={circle,fill=testcolor,inner sep=0pt, minimum size=2mm},
	dot/.style={circle,fill=black,inner sep=0pt, minimum size=1mm},
	var/.style={circle,fill=black!10,draw=black,inner sep=0pt, minimum size=2mm},
	dotred/.style={circle,fill=black!50,inner sep=0pt, minimum size=2mm},
	generic/.style={semithick,shorten >=1pt,shorten <=1pt},
	dist/.style={ultra thick,draw=testcolor,shorten >=1pt,shorten <=1pt},
	testfcn/.style={ultra thick,testcolor,shorten >=1pt,shorten <=1pt,<-},
	testfcnx/.style={ultra thick,testcolor,shorten >=1pt,shorten <=1pt,<-,
		postaction={decorate,decoration={markings,mark=at position 0.6 with {\drawx}}}},
	kprime/.style={semithick,shorten >=1pt,shorten <=1pt,densely dashed,->},
	kprimex/.style={semithick,shorten >=1pt,shorten <=1pt,densely dashed,->,
		postaction={decorate,decoration={markings,mark=at position 0.4 with {\drawx}}}},
	kernel/.style={semithick,shorten >=1pt,shorten <=1pt,->},
	multx/.style={shorten >=1pt,shorten <=1pt,
		postaction={decorate,decoration={markings,mark=at position 0.5 with {\drawx}}}},
	kernelx/.style={semithick,shorten >=1pt,shorten <=1pt,->,
		postaction={decorate,decoration={markings,mark=at position 0.4 with {\drawx}}}},
	kernel1/.style={->,semithick,shorten >=1pt,shorten <=1pt,postaction={decorate,decoration={markings,mark=at position 0.45 with {\draw[-] (0,-0.1) -- (0,0.1);}}}},
	kernel2/.style={->,semithick,shorten >=1pt,shorten <=1pt,postaction={decorate,decoration={markings,mark=at position 0.45 with {\draw[-] (0.05,-0.1) -- (0.05,0.1);\draw[-] (-0.05,-0.1) -- (-0.05,0.1);}}}},
	kernelBig/.style={semithick,shorten >=1pt,shorten <=1pt,decorate, decoration={zigzag,amplitude=1.5pt,segment length = 3pt,pre length=2pt,post length=2pt}},
	rho/.style={dotted,semithick,shorten >=1pt,shorten <=1pt},
	renorm/.style={shape=circle,fill=white,inner sep=1pt},
	labl/.style={shape=rectangle,fill=white,inner sep=1pt},
	xi/.style={circle,fill=symbols!10,draw=symbols,inner sep=0pt,minimum size=1.2mm},
	xix/.style={crosscircle,fill=symbols!10,draw=symbols,inner sep=0pt,minimum size=1.2mm},
	xib/.style={circle,fill=symbols!10,draw=symbols,inner sep=0pt,minimum size=1.6mm},
	xibx/.style={crosscircle,fill=symbols!10,draw=symbols,inner sep=0pt,minimum size=1.6mm},
	not/.style={circle,fill=symbols,draw=symbols,inner sep=0pt,minimum size=0.5mm},
	>=stealth,
	}
\makeatletter
\def\DeclareSymbol#1#2#3{\expandafter\gdef\csname MH@symb@#1\endcsname{\tikz[baseline=#2,scale=0.15,draw=symbols]{#3}}\expandafter\gdef\csname MH@symb@#1s\endcsname{\scalebox{0.7}{\tikz[baseline=#2,scale=0.15,draw=symbols]{#3}}}}
\def\<#1>{\csname MH@symb@#1\endcsname}
\makeatother

\DeclareSymbol{circle}{0.5}{\draw (0,0.7) node[xi] {};}
\DeclareSymbol{line}{0.5}{\draw (0,0.2) node[not] {} -- (0,1.4) node[not] {};}

\DeclareSymbol{Psi}{0.5}{\draw (0,0) node[not] {} -- (0,1.5) node[xi] {};}
\DeclareSymbol{Psi2}{0.5}{\draw (-0.8,1) node[xi] {} -- (0,0) node[not] {} -- (0.8,1) node[xi] {};}
\DeclareSymbol{IPsi2}{0}{\draw (0,1) -- (0.8,2.2) node[xi] {};\draw (0,-0.25) node[not] {} -- (0,1) node[not] {} -- (-0.8,2.2) node[xi] {};}
\DeclareSymbol{PsiIPsi2}{0}{\draw (0,1) -- (0.8,2.2) node[xi] {};\draw (0,-0.25) node[not] {} -- (0,1) node[not] {} -- (-0.8,2.2) node[xi] {};\draw (0,-0.25) node[not] {} -- (0.8,1) node[xi] {};}

\newcommand{\luxor}{\mathbf{\Psi}}
\newcommand{\luxo}{\mathbf{\Psi}^{\mathbf{1}}}
\newcommand{\cherry}{\mathbf{\Psi}^{\mathbf{2}}}

\title{A nonlinear  Schr\"{o}dinger equation with fractional noise}

\numberwithin{equation}{section}
\begin{document}

\subjclass[2000]{60H15; 35Q55; 60G22}    

\keywords{Stochastic Schr\"odinger equation, space-time  fractional  noise, Wick renormalization.}

%\begin{center}
%{\large\textbf{
%A non-linear fractional Schr\"{o}dinger equation}}\\~\\
%Aur{\'e}lien Deya\footnote{Institut \'Elie Cartan, Universit\' e de Lorraine, BP 70239, 54506 Vandoeuvre-l\`es-Nancy, France. }, Nicolas Schaeffer\footnotemark[\value{footnote}] and Laurent Thomann\footnotemark[\value{footnote}]
%\end{center}

\author{Aur{\'e}lien Deya}
\author{Nicolas Schaeffer}
\author{Laurent Thomann}
\address{Universit{\'e} de Lorraine, CNRS, IECL, F-54000 Nancy, France}
\email{aurelien.deya@univ-lorraine.fr}
\email{nicolas.schaeffer@univ-lorraine.fr}
\email{laurent.thomann@univ-lorraine.fr}

\begin{abstract}
We study a stochastic  Schr\"{o}dinger equation with a quadratic nonlinearity and a space-time fractional perturbation, in space dimension less than 3.  When the Hurst index is large enough, we prove local well-posedness of the problem using classical arguments. However, for a small Hurst index, even the interpretation of the equation needs some care. In this case, a renormalization procedure must come into the picture, leading to a Wick-type interpretation of the model. Our fixed-point argument then involves some specific regularization properties of the Schr\"{o}dinger group, which allows us to cope with the strong irregularity of the solution.
\end{abstract}

\maketitle

\

\section{Introduction and main results}\label{sec:intro}

\subsection{General introduction}
In this paper we study  the following $d$-dimensional stochastic     Schr\"{o}dinger equation with a quadratic nonlinearity and a space-time fractional perturbation:
\begin{equation}\label{equa-abstract}
\left\{
\begin{array}{l}
\imath \partial_t u-\Delta u= \rho^2 |u|^2 + \dot{B} \, , \quad  t\in [0,T] \, , \, x\in \R^d \, ,\\
u_0=\phi\, ,
\end{array}
\right.
\end{equation}
where  $\rho:\R^d\to \R$ is a smooth cut-off function in space and $\dot{B}$ stands for the derivative of a space-time fractional Brownian motion of Hurst index $H=(H_0,\ldots,H_d)\in (0,1)^{d+1}$.

\smallskip

We first show that, when $2H_0+\sum_{i=1}^{d}H_i >d+1$ (the so-called \emph{regular} case), the interpretation and local well-posedness of~\eqref{equa-abstract} can be derived from quite direct arguments, based on a first-order expansion and the use of Strichartz inequalities. 

\smallskip

The equation behaves less favorably when $2H_0+\sum_{i=1}^{d}H_i \leq d+1$  (the irregular or \emph{rough} case). In this situation, we first need a Wick-type  renormalization procedure in order to interpret the model. The fixed-point argument then relies on the smoothing properties of the Schr\"{o}dinger equation, and in particular on its local regularization effect.  

\medskip

We can (loosely) sum up our results as follows. 
 
\begin{theorem} 
Assume that $1\leq d \leq 3$ and set $\alpha_H:=\big(2H_0+\sum_{i=1}^{d}H_i\big)-(d+1)$. The following (partial) picture holds true:

\smallskip

\noindent
$(i)$ \textbf{Case $\alpha_H >0$.} The equation~\eqref{equa-abstract} is almost surely locally well-posed in $H^{\beta}(\R^d)$ for some $\beta>0$. 

\smallskip

\noindent
$(ii)$ \textbf{Case $\alpha_H \leq 0$.}  There exists $\alpha_d<0$ such that if $\alpha_H >\alpha_d$   then the equation~\eqref{equa-abstract}   can be interpreted in the Wick (renormalized) sense and  it is     almost surely locally well-posed in $H^{-\beta}(\R^d)$ for some $\beta>0$. 
\end{theorem}

We refer to  Definition~\ref{defi:sol-regu} and Theorem~\ref{resu} for precise statements in the regular case $(i)$, and to  Definition~\ref{defi:sol} and Theorem~\ref{resu1} in the rough case $(ii)$. To our knowledge, it is the first result in the context of  nonlinear Schr{\"o}dinger equations where   both renormalization arguments and local regularization properties   are used  to control  such an irregular noise (in Sobolev spaces of negative order).  \medskip

\

The stochastic Schr{\"o}dinger equation is a widely studied model in the SPDE literature. Just as for stochastic heat or wave equations, the stochastic Schr{\"o}dinger model admits numerous possible variants and is known to be the source of many challenging questions, whose treatment can only be achieved through the sophisticated combination of PDE tools with probabilistic considerations. Let us here mention some major contributions of  de Bouard, Debussche and their coauthors, see e.g.~\cite{debou-debu-4,debou-debu-6,debou-debu-8,debu-weber,debu-martin} and references therein, as well as  the recent works by Oh and his coauthors on this topic, see e.g.~\cite{oh-forlano-wang,oh-okamoto,oh-popovnicu-wang,oh-tzvetkov}.

\

In this paper, we propose to make another step in this direction, by considering the model of a nonlinear Schr\"{o}dinger equation with quadratic perturbation and forcing term given by a space-time fractional noise. To be more specific, the dynamics we will focus on can be described as follows:
\begin{equation}\label{dim-d}
\left\{
\begin{array}{l}
\imath \partial_t u-\Delta u= \rho^2 |u|^2 + \dot{B}\, , \quad  t\in [0,T] \, , \, x\in \R^d \, ,\\
u_0=\phi\, ,
\end{array}
\right.
\end{equation}
where 

\smallskip

\noindent
$\bullet$ $\rho:\R^d \to \R$ is a smooth compactly-supported function of the space variable, allowing us to bring the analysis back to compact domains, 

\smallskip

\noindent
$\bullet$ $\phi$ is a deterministic initial condition, the regularity of which will be specified later on,

\smallskip

\noindent
$\bullet$ one has $\dot{B}:=\partial_t \partial_{x_1}\cdots\partial_{x_d}B$, that is $\dot{B}$ is the space-time derivative (in the sense of distributions) of $B$, with $B$ a \emph{fractional sheet} of  Hurst index $H=(H_0,H_1,\dots,H_d) \in (0,1)^{d+1}$ (see Definition~\ref{def:fbm} below for details).

\

The consideration of such a fractional noise $\dot{B}$ will be the main specificity of our analysis. Observe that when $H_0=\frac12$, $\dot{B}$ is a white noise in time, and the situation in this case can be compared with the models studied e.g. in~\cite{debou-debu-4,oh-okamoto,oh-popovnicu-wang}. Our aim in the sequel will be to offer as much flexibility as possible regarding the choice of the Hurst index $H\in (0,1)^{d+1}$. Thus, for a (hopefully) large range of such indexes, we intend to at least prove local well-posedness of equation~\eqref{dim-d}. Note that this objective was already at the core of the investigations of the first author in~\cite{deya-wave,deya-wave-2} for a quadratic fractional wave equation (extending the white-noise model studied in~\cite{gubi-koch-oh}).  

\

Before we go further, let us recall that over the last decade, tremendous developments have been observed in the field of singular stochastic PDEs. This progress has been especially prominent in the parabolic (SPDE) setting, with the introduction of the theory of regularity structures~\cite{hai-14} or the paracontrolled approach~\cite{gubi-imke-perk}. Among other novelties, those theories provide a convenient framework towards renormalization procedures, thus paving the way to a rigorous treatment of many long-standing problems. If one focuses on additive-noise models only (in the vein of~\eqref{dim-d}), let us quote for instance, among a flourishing literature,  the work of Catellier and Chouk~\cite{cate-chouk} about the stochastic quantization equation on the three-dimensional torus
\begin{equation}\label{stoch-quant-equa}
\partial_t u-\Delta u=- u^3 +\xi \, , \quad  t\in [0,T] \, , \, x\in \mathbb{T}^3 \, ,
\end{equation}
or the study by Hairer and Shen~\cite{hairer-shen} about the parabolic sine-Gordon model
\begin{equation}\label{sine-gordon}
\partial_t u +\frac{1}{2}(1-\Delta) u +\sin(\be u)=\xi \, , \quad  t\in [0,T] \, , \, x\in \mathbb{T}^2 \, ,
\end{equation}
with $\xi$ a space-time white noise. 

\smallskip

Unfortunately, the application of those new groundbreaking approaches beyond the parabolic setting has proved to be very limited so far (this holds true for both regularity structures and paracontrolled theories). To our knowledge, the only attempt to extend such a strategy to a non-parabolic setting is due to Gubinelli, Koch and Oh in their recent work~\cite{gubi-koch-oh-2}, dealing with a stochastic wave model. In particular, we are not aware of any similar extension to a singular stochastic Schr{\"o}dinger equation. 

\

As regards the deterministic Schr\"odinger equation with polynomial nonlinearities, its well-posedness in positive Sobolev spaces was established long ago using Strichartz estimates (see \cite{GinibreVelo, CazWei} and also the monography \cite{caz}). More recent developments, applying in particular to NLS with quadratic nonlinearities, also led to well-posedness results for the model in negative Sobolev spaces, thanks to subtle bilinear estimates in the so-called Bourgain spaces (see \cite{Co-De-Ke-Sta, Tao, Beje-DeSilva,Iwa-Oga}, and also Remark \ref{rk:bourgain-spaces} below). 

\

With this background in mind, let us now go back to the analysis of equation~\eqref{dim-d}. The starting point of the study will be the mild formulation of~\eqref{dim-d}, that is the equation
\begin{equation}\label{mild-equation-u}
u_t=S_t \phi -\imath\int_0^t S_{t-\tau}(\rho^2 |u_{\tau}|^2)\, d\tau +\<Psi>_t \, ,
\end{equation}
where $S$ stands for the Schr\"{o}dinger group, and where we have set 
\begin{equation}\label{def-luxo-formal}
\<Psi>_t:=-\imath\int_0^t S_{t-\tau}(\dot B_{\tau})\, d\tau \, .
\end{equation}
Note that $\<Psi>$ can also be seen as the solution of the following \enquote{linear} counterpart of~\eqref{dim-d}:
\begin{equation}\label{linear-equation}
 \left\{
    \begin{array}{ll}
  \imath\partial_t\<Psi>-\Delta \<Psi> = \dot B, \hskip 0.3 cm t \in [0,T], \hspace{0,2cm} x \in \mathbb{R}^d,   \\
 \  \<Psi>(0,.)=0\, .
    \end{array}
\right.
\end{equation}
For this reason, we will henceforth refer to $\<Psi>$ as the linear solution of the problem. 

\smallskip

A first essential part of our work will be to give a precise meaning to both definition~\eqref{def-luxo-formal} and equation~\eqref{mild-equation-u}. In order to initiate this analysis, let us proceed with a standard transformation of the problem, by considering the equation satisfied by the process $v:=u-\<Psi>$, that is the equation (which is still formal at this point)
\begin{multline}\label{equa-v-intro}
v_t=S_t(\phi)-\imath\int_0^t S_{t-\tau}(\rho^2 |v_{\tau}|^2)\, d\tau-\imath\int_0^t S_{t-\tau}((\rho \overline{v}_{\tau})\cdot(\rho \<Psi>_{\tau}))\, d\tau\\
-\imath\int_0^t S_{t-\tau}((\rho v_{\tau})\cdot(\overline{\rho \<Psi>_{\tau}}))\, d\tau-\imath\int_0^t S_{t-\tau}( |\rho\<Psi>_{\tau}|^2)\, d\tau \, .
\end{multline}
The main idea behind this transition from $u$ to $v$ is that $\<Psi>$ is expected to behave as some \enquote{first-order expansion} of $u$. In other words, due to the specific properties of the group $S$ (which we will detail later on), we expect the process $v$ to be more regular than $u$ and $\<Psi>$, making equation~\eqref{equa-v-intro} more tractable than~\eqref{mild-equation-u}. Following this idea, we will focus our subsequent investigations on equation~\eqref{equa-v-intro}.

\

Let us now elaborate on the successive steps of our reasoning, and introduce our main results. Note that these steps are overall similar to those developed in~\cite{deya-wave,deya-wave-2,gubi-koch-oh} for the corresponding quadratic wave model. Nevertheless, when going into the technical details, some new fundamental difficulties arise in the analysis of the Schr{\"o}dinger model, as we will try to highlight it in the  presentation below.

\subsection{Step 1: Study of the linear equation}\label{subsec:intro-linear}

\

\smallskip

Recall that the noise $\dot{B}$ involved in~\eqref{dim-d} is defined as the derivative of a fractional sheet $B$, which is a non-differentiable process (in the standard sense). Consequently, just as the white noise $\xi$ in~\eqref{stoch-quant-equa}-\eqref{sine-gordon}, $\dot{B}$ can only be understood as some random negative-order distribution, and thus the interpretation of the convolution in~\eqref{def-luxo-formal} requires some clarification.

\smallskip

To do so, let us start from a smooth approximation $(B_n)_{n\geq 0}$ of $B$, that is, for each fixed $n\geq 0$, $(t,x)\mapsto B_n(t,x)$ is a.s. smooth, and $B_n(t,x) \stackrel{n\to\infty}{\longrightarrow} B(t,x)$ for every $(t,x)\in [0,T]\times \R^d$ (the choice of such an approximation process will be specified in Section~\ref{stocha} below). Then consider the corresponding sequence of linear solutions, that is the sequence $\<Psi>_n$ of solutions to the equation
\begin{equation}\label{linear-equation-n}
 \left\{
    \begin{array}{ll}
  \imath\partial_t\<Psi>_n-\Delta \<Psi>_n = \dot B_n\, , \quad t \in [0,T]\, , \ x \in \mathbb{R}^d\, ,   \\
 \  \<Psi>_n(0,.)=0\, ,
    \end{array}
\right.
\end{equation}
where $\dot{B}_n:=\partial_t\partial_{x_1}\cdots \partial_{x_d}B_n$. For every fixed $n\geq 0$, the smoothness of $B_n$ (and accordingly the smoothness of $\dot{B}_n$) makes the analysis of~\eqref{linear-equation-n} considerably easier than the one of~\eqref{linear-equation}, and readily allows us to define a unique Gaussian solution process $\big\{\<Psi>_n(t,x), \, t\in [0,T],x\in \R^d\big\}$ (see Definition~\ref{defi:luxo-n}). 

\smallskip

The solution $\<Psi>$ of the rough equation~\eqref{linear-equation} is then interpreted through the following convergence result, which can be seen as our first main contribution:

\begin{proposition}\label{sto}
Let $d\geq 1$ and fix $(H_0,\ldots,H_d)\in (0,1)^{d+1}$. Let $(B_n)_{n\geq 0}$ be the sequence of smooth processes defined by formula~\eqref{defi:approx-b-n}, and let $\<Psi>_n$ be the solution of~\eqref{linear-equation-n} associated with $B_n$.

\smallskip

Then, for every test function $\chi:\mathbb{R}^d$ $\rightarrow$ $\mathbb{R}$ (i.e., smooth and compactly-supported), the sequence $(\chi \<Psi>_n)_{n  \geq 0}$ converges almost surely in the space $\mathcal{C}([0,T]; \mathcal{W}^{-\al,p}(\mathbb{R}^d))$, for all $2\leq p\leq \infty$ and
\begin{equation}\label{assump-gene-al}
\alpha > d+1-\bigg(2H_0+\sum_{i=1}^{d}H_i\bigg)\, .
\end{equation}
We denote the limit of this sequence by  $\chi \<Psi>$.
\end{proposition}

\smallskip

The proof of this convergence result relies on the stochastic properties of $\dot{B}$, and will be developed in Section~\ref{subsec:proof-sto-1}. As the reader might expect it, the resulting regularity property (i.e., the fact that $\chi\<Psi>\in \mathcal{C}([0,T]; \mathcal{W}^{-\al,p}(\mathbb{R}^d))$, for every $\al$ satisfying~\eqref{assump-gene-al}) will be of crucial importance in the analysis of~\eqref{equa-v-intro}.

\smallskip

Using a standard patching argument, the limit elements $\{\chi\<Psi>, \ \chi\in \cac_c^\infty(\R^d)\}$ provided by Proposition~\ref{sto} can then be merged into a single locally-controlled distribution $\<Psi>$ (see Proposition~\ref{prop:defi-psi-glo}), which we will refer to in the sequel.

\smallskip

\begin{remark}\label{rk:compari-lin}
We can compare the above regularity restriction~\eqref{assump-gene-al} for $\<Psi>$ with the corresponding result of~\cite{deya-wave} in the wave setting, that is when replacing $\imath\partial_t\<Psi>-\Delta \<Psi>$ with $\partial^2_t\<Psi>-\Delta \<Psi>$ in~\eqref{linear-equation}. In the latter situation, and according to~\cite[Proposition 1.2]{deya-wave}, one must have 
$$
\alpha_{\text{wave}} > d-\frac12-\sum_{i=0}^{d}H_i\, .
$$
In light of this result, it is interesting to see how the Schr{\"o}dinger scaling (where the time variable somehow \enquote{counts twice}) echoes on condition~\eqref{assump-gene-al}, through the combination $2H_0+\sum_{i=1}^{d}H_i$.\\
\indent Besides, although such a property cannot be found in the existing literature, we could show along the same pattern that in the heat situation (that is with $\partial_t\<Psi>-\Delta \<Psi>$ instead of $\imath\partial_t\<Psi>-\Delta \<Psi>$ in~\eqref{linear-equation}), the restriction for $\al$ becomes
$$
\alpha_{\text{heat}} > d-\bigg(2H_0+\sum_{i=1}^{d}H_i\bigg)\, ,
$$
which, compared to~\eqref{assump-gene-al}, reflects the stronger regularization properties of the heat kernel.
\end{remark}

\smallskip

\subsection{Step 2: Interpretation of the main equation}

\

\smallskip

Now equipped with a proper definition of $\<Psi>$ (as well as a sharp control on its regularity), we can go back to our interpretation issue for the main equation~\eqref{mild-equation-u} (or equivalently~\eqref{equa-v-intro}). In fact, according to the result of Proposition~\ref{sto}, two cases need to be distinguished.

\subsubsection{The regular case}

\

\smallskip

In the sequel, we will call \emph{regular case} the situation where 
\begin{equation}\label{cond-regu-case}\tag{\textbf{H1}}
2H_0+\sum_{i=1}^{d}H_i >d+1 \, .
\end{equation}
In this case, we can pick $\al <0$ such that condition~\eqref{assump-gene-al} is satisfied, and therefore, thanks to Proposition~\ref{sto}, we are allowed to consider every element $\chi \<Psi>$ ($\chi\in \cac^\infty_c(\R^d)$) as a function \emph{of both time and space variables} (in an almost sure way). As a result, the square element $|\rho \<Psi>|^2$ involved in~\eqref{equa-v-intro} simply makes sense as a standard pointwise product of functions. This immediately leads us to the following direct interpretation of the equation: 

\begin{definition}\label{defi:sol-regu}
Fix $d\geq 1$ and assume that condition~\eqref{cond-regu-case} is satisfied. Recall that for every test function $\chi:\R^d \to \R$, $\chi\<Psi>$ is the process defined in Proposition~\ref{sto}.

\smallskip

Then we call a solution (on $[0,T]$) of equation~\eqref{dim-d} any stochastic process $(u(t,x))_{t \in [0,T], x \in \mathbb{R}^d}$ such that, almost surely, the process $v:=u-\<Psi>$ is a solution of the mild equation
\begin{multline*}
v_t=S_t(\phi)-\imath\int_0^t S_{t-\tau}(\rho^2 |v_{\tau}|^2)\, d\tau-\imath\int_0^t S_{t-\tau}((\rho \overline{v}_{\tau})\cdot(\rho \<Psi>_{\tau}))\, d\tau\\
-\imath\int_0^t S_{t-\tau}((\rho v_{\tau})\cdot(\overline{\rho \<Psi>_{\tau}}))\, d\tau-\imath\int_0^t S_{t-\tau}(|\rho \<Psi>_{\tau}|^2)\, d\tau \, , \quad t\in [0,T]\, .
\end{multline*}
%$$v_t=S_t(\phi)-\imath\int_0^t S_{t-\tau}(\rho^2v_{\tau}^2)\, d\tau-2\imath\int_0^t S_{t-\tau}((\rho v_{\tau})\cdot (\rho \<Psi>_{\tau}))\, d\tau-\imath\int_0^t S_{t-\tau}((\rho \<Psi>_{\tau})^2)\, d\tau \, , \quad t\in [0,T]\, .$$
\end{definition}

\smallskip

\subsubsection{The rough case}

\

\smallskip

Let us now turn to the second situation, where 
\begin{equation}
2H_0+\sum_{i=1}^{d}H_i \leq d+1 \, .
\end{equation}
In this case, we can no longer pick $\al <0$ such that condition~\eqref{assump-gene-al} is satisfied, and so, referring to Proposition~\ref{sto}, the element $\rho\<Psi>$ involved in~\eqref{equa-v-intro} must be considered as a function of time \emph{with values in a space of negative-order distribution}. We will call this situation the \emph{rough case}, to signify that we are here working with very irregular processes.

\smallskip

Naturally, the fact that $\rho \<Psi>_\tau$ must be seen as a distribution (for every fixed $\tau$) makes the interpretation of $|\rho\<Psi>_{\tau}|^2$ in~\eqref{equa-v-intro} unclear. 

\smallskip

This problem can be emphasized through a regularization approach. Namely, let us go back to the sequence of approximated linear solutions $(\<Psi>_n)_{n\geq 0}$ satisfying~\eqref{linear-equation-n}. For every fixed $n\geq 0$, $\<Psi>_n$ is (almost surely) a function of time and space, and so, for every $(t,x)\in (0,T]\times \R^d$, we can compute the moment $\mathbb{E}\big[|\<Psi>_n(t,x)|^2\big]$. The following asymptotic result (which will be proved in Section~\ref{subsec:renorm-cstt}) rules out any possibility to extend the pointwise interpretation to the limit element~$|\<Psi>|^2$:

\begin{proposition}\label{prop:renorm-cstt}
Fix $d\geq 1$ and assume that $2H_0+\sum_{i=1}^{d}H_i\leq d+1$. Then the quantity $\mathbb{E}\big[|\<Psi>_n(t,x)|^2\big] $ does not depend on $x$. Denoting $\sigma_n(t):=\mathbb{E}\big[|\<Psi>_n(t,x)|^2\big]$,  the  following asymptotic equivalence property holds true: for every $(t,x)\in (0,T]\times \R^d$,
\begin{equation}\label{estim-sigma-n}
    \sigma_n(t) \underset{n \rightarrow \infty}\sim  \left\{
    \begin{array}{ll}
		c_H^1\, t\, n & \mbox{if}\quad 2H_0+\sum_{i=1}^{d}H_i=d+1 ,\\
        c_H^2\, t\, 2^{2n(d+1-[2H_0+\sum_{i=1}^{d}H_i])} & \mbox{if} \quad 2H_0+\sum_{i=1}^{d}H_i<d+1 \, ,  
    \end{array}
\right.
\end{equation}
for some constants $c_H^1,c_H^2>0$.
\end{proposition}

A natural way to circumvent this divergence issue and to offer a possible interpretation of~$|\<Psi>|^2$ is to turn to a \emph{renormalization procedure}. In fact, it will here be sufficient to consider the most elementary of those methods, namely the Wick renormalization. Thus, for all fixed $n\geq 0$, $(t,x)\in [0,T]\times \R^d$, we set
\begin{equation}\label{defi:cerise}
\<Psi2>_n(t,x) := |\<Psi>_n(t,x)|^2-\sigma_n(t) \quad \text{where} \ \sigma_n(t):=\mathbb{E}\big[|\<Psi>_n(t,x)|^2\big] \, .
\end{equation}
Our second main contribution now reads as follows:

\begin{proposition}\label{sto1}
Let $d\geq 1$ and consider $(H_0,H_1,\dots,H_d)\in (0,1)^{d+1}$ such that
\begin{equation}\label{cond-hurst-psi-2}\tag{\textbf{H2}}
d+\frac{3}{4}< 2H_0+\sum_{i=1}^{d}H_i \leq d+1\, .
\end{equation}
Then, for every test function $\chi:\mathbb{R}^d\rightarrow\mathbb{R}$, the sequence $(\chi^2\<Psi2>_n)_{n \geq 1}$ (defined by~\eqref{defi:cerise}) converges almost surely in the space $\mathcal{C}([0,T]; \mathcal{W}^{-2\al,p}(\mathbb{R}^d))$, for all $2\leq p\leq \infty$ and $\al >0$ satisfying~\eqref{assump-gene-al}.

\smallskip

We denote the limit of this sequence by  $\chi^2 \<Psi2>$.
\end{proposition}

Just as the proof of Proposition~\ref{sto}, the proof of Proposition~\ref{sto1} will strongly lean on the stochastic properties of $\dot{B}$. It will be the topic of Section~\ref{subsec:cerise} below.

\smallskip

\begin{remark}
The restriction $2H_0+\sum_{i=1}^{d}H_i>d+\frac{3}{4}$ in~\eqref{cond-hurst-psi-2} (which will stem from our technical computations) can be compared with the restriction $\sum_{i=0}^d H_i > d-\frac34 $ in~\cite[Proposition 1.4]{deya-wave} for the quadratic wave model. We suspect that this condition might not be optimal, that is, we can certainly extend the construction of $\chi^2\<Psi2>$ below this threshold. However, condition~\eqref{cond-hurst-psi-2} will prove to be sufficient for our purpose, as it can be seen from the comparison with the more restrictive assumption~\eqref{cond-irreg-case} in our main theorem (see also Remark~\ref{rk:limitations}).
\end{remark}

\smallskip

With the above constructions in hand, the following \emph{Wick interpretation} of the equation naturally arises:

\begin{definition}\label{defi:sol}
Fix $d\geq 1$ and assume that condition~\eqref{cond-hurst-psi-2} is satisfied. Recall that for every test function $\chi:\R^d \to \R$, $\chi\<Psi>$ and $\chi^2\<Psi2>$ are the processes defined in Proposition~\ref{sto} and Proposition~\ref{sto1}, respectively.

\smallskip

In this setting, we call a \emph{Wick} solution (on $[0,T]$) of equation~\eqref{dim-d} any stochastic process $(u(t,x))_{t \in [0,T], x \in \mathbb{R}^d}$ such that, almost surely, the process $v:=u-\<Psi>$ is a solution of the mild equation
\begin{multline}\label{equa-v-rough}
v_t=S_t(\phi)-\imath\int_0^t S_{t-\tau}(\rho^2 |v_{\tau}|^2)\, d\tau-\imath\int_0^t S_{t-\tau}((\rho \overline{v}_{\tau})\cdot(\rho \<Psi>_{\tau}))\, d\tau\\
-\imath\int_0^t S_{t-\tau}((\rho v_{\tau})\cdot(\overline{\rho \<Psi>_{\tau}}))\, d\tau-\imath\int_0^t S_{t-\tau}(\rho^2 \<Psi2>_{\tau})\, d\tau\, , \quad t\in [0,T] \, .
\end{multline}
%\begin{equation}
%v_t=S_t(\phi)-\imath\int_0^t S_{t-\tau}(\rho^2v_{\tau}^2)\, d\tau-2\imath\int_0^t S_{t-\tau}((\rho v_{\tau})\cdot(\rho \<Psi>_{\tau}))\, d\tau-\imath\int_0^t S_{t-\tau}(\rho^2 \<Psi2>_{\tau})\, d\tau\, ,
%\end{equation}
\end{definition}

\smallskip

\begin{remark}
As the reader may have noticed, our tree notation $\<Psi>$ and $\<Psi2>$ for the main stochastic processes follows the convention used in~\cite{hai-14}: the circle $\<circle>$ therein stands for the random noise $\dot{B}$, while the line $\<line>$ represents the Duhamel integral operator $\mathcal{I}=(i\partial_t -\Delta)^{-1}$.
\end{remark}

\smallskip

\subsection{Step 3: Local wellposedness results}

\

\smallskip

At this stage of the procedure, the stochastic part of our analysis can be considered as complete, and our aim now is to solve equation~\eqref{dim-d} (understood along Definition~\ref{defi:sol-regu} or Definition~\ref{defi:sol}) by means of a \emph{deterministic} fixed-point argument.

\smallskip

As in the previous section, and for the sake of clarity, let us separate the regular and rough situations in the presentation of our results.

\subsubsection{The regular case}

\

\smallskip

Let us first handle the regular case, where condition~\eqref{cond-regu-case} is satisfied and Definition~\ref{defi:sol-regu} of a solution prevails. By relying on the most standard estimates associated with the Schr{\"o}dinger group $S$ (the so-called \emph{Strichartz inequalities}, summed up in Lemma~\ref{lem:strichartz} below), we are here able to establish the following result: 

\begin{theorem}[\textbf{Local well-posedness under~\eqref{cond-regu-case}}]\label{resu}
Assume that $1\leq d \leq 4$ and that condition~\eqref{cond-regu-case} is satisfied.
Let $\beta$ be such that $0< \beta < 2H_0+\sum_{i=1}^{d}H_i- (d+1)$, and consider the pair $(p,q)$ given by the formulas
$$p=\frac{12}{d-\beta} \ , \ q=\frac{6d}{2d+\beta} \, .$$
Assume finally that $\phi \in H^\beta(\mathbb{R}^d)$. In  this setting, the following assertions hold true:

\smallskip

\noindent
$(i)$ Almost surely, there exists a time $T_0 >0$ such that equation~\eqref{dim-d} admits a unique solution~$u$ (in the sense of Definition~\ref{defi:sol-regu}) in the set 
$$\mathcal{S}_{T_0}:= \<Psi> + X^{\beta}(T_0),\hspace{0,5cm} \text{where} \ X^{\beta}(T_0):=\mathcal{C}([0,T_0]; H^\beta(\mathbb{R}^d))\cap L^p([0,T_0];\mathcal{W}^{\beta,q}(\mathbb{R}^d)).$$

\smallskip

\noindent
$(ii)$ For every $n\geq 1$, let $u_n$ denote the \emph{smooth} solution of~\eqref{dim-d}, that is $u_n$ is the solution (in the sense of Definition~\ref{defi:sol-regu}) associated with $\rho \<Psi>_n$. Then, for every 
$$0< \beta < 2H_0+\sum_{i=1}^{d}H_i- (d+1)$$
and for every test function $\chi: \R^d \to \R$, the sequence $(\chi u_{n})_{n\geq 1}$ converges almost surely in $\mathcal{C}([0,T_0];H^{\beta}(\mathbb{R}^d))$ to $\chi u$, where $u$ is the solution  exhibited in item $(i)$.
\end{theorem}

Let us be slightly more specific about the convergence property in item $(ii)$. In fact, what we will establish in the sequel (see Theorem~\ref{thm:regular} below) is that for every element $\luxo$ in a suitable space of functions, the equation obtained by replacing~$\rho\<Psi>$ with $\luxo$ in~\eqref{equa-v-intro} admits a unique solution~$v$, on some time interval depending only on~$\luxo$. Besides, this solution is a continuous function of~$\luxo$ (see Proposition~\ref{control-regular}). Item $(i)$ in the above statement is then an application of these general results to $\luxo=\rho\<Psi>$, while item $(ii)$ corresponds to taking $\luxo=\rho\<Psi>_n$, which provides us with the desired time $T_0>0$.

\smallskip

\subsubsection{The rough case}\label{subsec:wellposed-rough-intro}

\

\smallskip

Let us conclude this presentation of our main results by considering the wellposedness issue for the equation in the rough case. To be more specific, we assume from now on that the assumptions in~\eqref{cond-hurst-psi-2} are satisfied, so that the two processes $\rho\<Psi>$ and $\rho^2\<Psi2>$ are well defined and the equation can be understood in the sense of Definition~\ref{defi:sol}. In other words, we now focus on the analysis of equation~\eqref{equa-v-rough}. %and try to to settle a fixed-point argument. 

\smallskip

In order to describe the major technical difficulty  arising in this case, recall first that under assumption~\eqref{cond-hurst-psi-2} and following Proposition~\ref{sto}, the element $\rho\<Psi>_\tau$ must here be treated as a distribution of negative Sobolev regularity (for every fixed time $\tau$). Consequently, the term $(\rho \overline{v}_\tau) \cdot (\rho \<Psi>_\tau)$ (or $(\rho v_\tau) \cdot (\overline{\rho \<Psi>_\tau})$) in~\eqref{equa-v-rough} can only be understood as the product of a distribution, namely $\rho \<Psi>_\tau$, with a function, that is $\rho \overline{v}_\tau$. Such a product is known to obey the following simple rule (see Lemma~\ref{lem:product} for a precise statement): if $\rho \<Psi>_\tau$ is of Sobolev regularity $-\alpha$ (with $\al >0$), then $\rho \overline{v}_\tau$ must be a function of Sobolev regularity $\beta$ with $\beta>\alpha$, and in this case $(\rho \overline{v}_\tau) \cdot (\rho \<Psi>_\tau)$ is indeed well defined as a distribution of negative order $-\alpha$.

\smallskip

Going back to equation~\eqref{equa-v-rough}, our only hope to settle a fixed-point argument thus lies in the possibility to control the terms
$$\int_0^t S_{t-\tau}((\rho \overline{v}_{\tau})\cdot(\rho \<Psi>_{\tau}))\, d\tau \ , \ \int_0^t S_{t-\tau}((\rho v_{\tau})\cdot(\rho \overline{\<Psi>}_{\tau}))\, d\tau$$
as functions of Sobolev regularity $\be >\al >0$. Otherwise stated, we (morally) need convolution with $S$ to produce a regularization effect and allow the transition from $H^{-\al}(\R^d)$ to $H^\be(\R^d)$. 

\smallskip

Unfortunately, such a regularization property, which corresponds to a well-controlled phenomenon in the heat or in the wave situation, is not as standard in the Schr{\"o}dinger setting. In particular, the most classical estimates on $S$ (the so-called Strichartz inequalities, which we have already mentioned in the regular situation) cannot provide us with the desired smoothing effect (see Lemma~\ref{lem:strichartz} for more details). 

\smallskip

To overcome this obstacle, we propose to turn to more specific \emph{local} regularization properties, similar to those exhibited by Constantin and Saut in~\cite{constantin-saut}. It indeed appears that if we only focus on the local regularity of the distributions at stake (meaning that a cut-off function is inserted within the usual Sobolev norm, see~\eqref{defi:h-rho}), then a small gain can be expected from the convolution with $S$. This will be the topic of our intermediate Lemma~\ref{lem:loc-regu}, which can be seen as an extension of the main result in~\cite{constantin-saut}. 

\smallskip

For this technical property to be implemented here, an additional condition must be imposed on the function $\rho:\R^d \to \R$ in~\eqref{dim-d} (or in~\eqref{equa-v-rough}): namely, we need $\rho$ to be of the form
\begin{equation}\label{form-rho}\tag{$\mathbf{F_\rho}$}
\rho(x_1,\dots,x_d)=\rho_1(x_1)\cdots \rho_d(x_d)
\end{equation}
for smooth compactly-supported functions $\rho_1,\ldots,\rho_d$ on $\R$.

\smallskip

Note also that for stability reasons (toward a fixed-point argument), the consideration of a local Sobolev norm is of course not without consequences: it will have to be counterbalanced by a commutator-type estimate, that is a suitable control on switching $\rho$ with the fractional Laplacian in Sobolev norms, which will be the purpose of Lemma~\ref{lem:commut}.

\smallskip

With the above elements in mind, we are finally in a position to state our main result in the rough situation (the spaces involved in this statement will all be introduced into details in the subsequent Section~\ref{subsec:notations}):
 
\begin{theorem}[\textbf{Local well-posedness under~\eqref{cond-irreg-case}}]\label{resu1}
Assume that $1\leq d\leq 3$ and that the cut-off function $\rho: \R^d \to \R$ in~\eqref{dim-d} is of the form~\eqref{form-rho}. Besides, assume that 
\begin{equation}\label{cond-irreg-case}\tag{\textbf{H2'}}
-\al_d<2H_0+\sum_{i=1}^{d}H_i- (d+1)\leq 0 \, , \quad \text{where} \ \ \al_d:=
\left\{
\begin{array}{l}
\frac{3}{20} \quad \text{if} \ d=1\\
\frac{1}{10} \quad \text{if} \ d=2\\
\frac{1}{24} \quad \text{if} \ d=3 
\end{array}
\right. \, .
\end{equation}
Fix $\al>0$ such that $d+1-\big(2H_0+\sum_{i=1}^{d}H_i\big)<\al<\al_d$. Then the following assertions hold true:

\smallskip

\noindent
$(i)$ One can find parameters $0\leq \ka\leq \frac12$ and $p,q\geq 2$ such that almost surely, and for every $\phi \in H^{-2\alpha}(\mathbb{R}^d)$, there exists a time $T_0>0$ for which equation~\eqref{dim-d} admits a unique Wick solution $u$ (in the sense of Definition~\ref{defi:sol}) in the set
$$\mathcal{S}_{T_0}:= \<Psi> + X^{\alpha,\ka,(p,q)}_\rho(T_0)\, ,$$
where
$$X^{\alpha,\ka,(p,q)}_\rho(T):=\mathcal{C}([0,T]; H^{-2\alpha}(\mathbb{R}^d))\cap L^p([0,T]; \cw^{-2\alpha,q}(\mathbb{R}^d))\cap L^{\frac{1}{\ka}}_T H^{-2\al+\ka}_\rho \, .$$

\smallskip

\noindent
$(ii)$ For every $n\geq 1$, let $\widetilde{u}_n$ denote the \emph{smooth} Wick solution of~\eqref{dim-d}, that is $\widetilde{u}_n$ is the solution (in the sense of Definition~\ref{defi:sol}) associated with the pair $(\rho \<Psi>_n,\rho^2\<Psi2>_n)$. Then, for every $\al$ satisfying~\eqref{assump-gene-al} and every test function $\chi:\R^d \to \R$, the sequence $(\chi\widetilde{u}_{n})_{n\geq 1}$ converges almost surely in $\mathcal{C}([0,T_0];H^{-2\al}(\mathbb{R}^d))$ to $\chi u$, where $u$ is the Wick solution exhibited in item $(i)$.
\end{theorem}

\

The above wellposedness result can be considered as the main novelty of our work. We are indeed not aware of any previous study of a nonlinear Schr{\"o}dinger equation involving such an irregular noise, and thus forcing us to rely on both renormalization arguments and local regularization properties.

\

\begin{remark}\label{rk:limitations}
It is interesting to note that the conditions in~\eqref{cond-irreg-case}, which stem from the deterministic part of the study (as emphasized by Theorem~\ref{thm:noregular} below), are more restrictive than those in~\eqref{cond-hurst-psi-2} ensuring the existence of the stochastic element $\<Psi2>$. This observation contrasts with the situation described for instance in~\cite{deya-wave,deya-wave-2,gubi-koch-oh,oh-okamoto-2}, where the application of a similar strategy is, on the contrary, limited by the scope of validity of the stochastic objects.
\end{remark}

\smallskip

\begin{remark}
Due to the limitations of the local regularization effect of the Schr{\"o}dinger group (as reported in Lemma~\ref{lem:loc-regu}), further expansions of the strategy, in the spirit of the \enquote{second-order analysis} developed e.g. in~\cite{deya-wave-2,oh-okamoto-2}, seem difficult to set up in this Schr{\"o}dinger setting.
\end{remark}

\smallskip

\begin{remark}
As can be seen from the above description of our methodology, the cut-off function~$\rho$ in~\eqref{dim-d} is to play two fundamental roles in the study: \\
\indent $\bullet$ First, it will allow us to bring the analysis of the equation back to compact domains, where the regularity of the driving processes is well controlled (by Propositions~\ref{sto} and ~\ref{sto1}). The situation, in this regard, is somehow equivalent to studying equation~\eqref{dim-d} on a torus (although the definition of the space-time fractional noise on a torus is not as clear as in the current Euclidean setting).\\
\indent $\bullet$ Secondly, thanks to the involvement of $\rho$, we can appeal to the specific local regularization properties of $S$, which, as we have explained it above, will be our key ingredient toward stability and fixed-point arguments. Observe that no similar regularization result would be available for a study of the equation on a torus.\\
\indent This being said, in spite of the restriction~\eqref{form-rho}, the function $\rho$ can still be taken equal to $1$ on any arbitrary compact domain, and so, at least locally (in time and in space), our solution $u$ of~\eqref{dim-d} can be regarded as a viable model for the (formal) dynamics
\begin{equation}\label{eq-without-rho}
\imath \partial_t u-\Delta u=|u|^2 + \dot{B} \, .
\end{equation}
A direct analysis of equation~\eqref{eq-without-rho} may be possible through an extension of the subsequent methods to weighted Sobolev spaces (allowing to control the asymptotic behaviour of the processes), but such adaptations are clearly beyond our reach for the moment.
\end{remark}

\smallskip

\begin{remark}
Our arguments and results could certainly be extended to the nonlinearity $\rho^2 u^2$ or~$\rho^2 \overline{u}^2$ (instead of $\rho^2 |u|^2$) through minor modifications of the stochastic constructions of Section~\ref{stocha} (the deterministic well-posedness procedure would then clearly follow the same lines).
\end{remark}

\

\subsection{Notations}\label{subsec:notations}

 Fix a (space) dimension $d\geq 1$. Throughout the paper, we will call a \emph{test function} (on $\R^d$) any function $\rho:\R^d \to \R$ that is smooth and compactly-supported. Besides, we denote by $\mathcal{S}(\R^d)$ the space of Schwartz functions on $\R^d$.

\smallskip

We will also refer to the scale of Sobolev spaces defined for all $s\in \mathbb{R}$ and $1\leq p\leq \infty$ as
$$\mathcal{W}^{s,p}(\mathbb{R}^d):=\left\{f\in \mathcal{S}'(\mathbb{R}^d):\ \|f \|_{\mathcal{W}^{s,p}}=\|\mathcal{F}^{-1}(\{1+|.|^2\}^{\frac{s}{2}}\mathcal{F}f) |L^p(\mathbb{R}^d)\| 
  <\infty \right\} \, ,$$
where the Fourier transform $\cf$, resp. the inverse Fourier transform $\cf^{-1}$, is defined along the convention: for all $f \in \mathcal{S}(\mathbb{R}^d)$ and $x \in \mathbb{R}^d,$
$$\mathcal{F}(f)(x)=\hat{f}(x):=\int_{\mathbb{R}^d} f(y)e^{-\imath \langle x, y \rangle}dy\, , \quad \text{resp.} \ \ \mathcal{F}^{-1}(f)(x):=\frac{1}{(2\pi)^d}\int_{\mathbb{R}^d} f(y)e^{\imath \langle x, y \rangle}dy\, .$$
We then set $H^s(\mathbb{R}^d):=\mathcal{W}^{s,2}(\mathbb{R}^d)$.

\smallskip

Now, as far as space-time functions (or distributions) are concerned, and for the sake of clarity, we will occasionally use the following shortcut notation: for all $T\geq 0$, $1\leq p,q\leq \infty$ and $s\in \R$, 
\begin{equation}\label{shortcut-spaces}
L^p_T \cw^{s,q}:=L^p([0,T];\mathcal{W}^{s,q}(\mathbb{R}^d))\, , \quad  \quad \|.\|_{L^p_T \cw^{s,q}}:=\|.\|_{L^p([0,T];\mathcal{W}^{s,q}(\mathbb{R}^d))} \, .
\end{equation}
The notation $\mathcal{C}([0,T];\mathcal{W}^{s,q}(\mathbb{R}^d))$ will refer to the set of continuous functions on $[0,T]$ with values in $\mathcal{W}^{s,q}(\mathbb{R}^d)$.

\smallskip

Let us finally introduce the aforementioned \emph{local} Sobolev spaces, that will play a prominent role in the analysis of the rough situation. Namely, for all test function $\rho:\R^d \to \R$ and $s\in \R$, we set
\begin{equation}\label{defi:h-rho}
H^{s}_\rho(\R^d):=\{ v \in \mathcal{S}^{'}(\mathbb{R}^d);\ \|\rho \cdot (\id-\Delta)^{\frac{s}{2}}(v)\|_{L^2(\mathbb{R}^d)}<\infty \} \, ,
\end{equation}
where the operator $(\id-\Delta)^{\frac{s}{2}}$ is understood (as usual) through the Fourier transform formula 
$$\cf\big((\id-\Delta)^{\frac{s}{2}}(v)\big)(\xi):=\{1+|\xi|^2\}^{\frac{s}{2}} \cf v (\xi) \, .$$
We endow $H^{s}_\rho(\R^d)$ with the natural seminorm $\|v\|_{H^s_\rho}:=\|\rho \cdot (\id-\Delta)^{\frac{s}{2}}(v)\|_{L^2(\mathbb{R}^d)}$,
and finally set, along the convention in~\eqref{shortcut-spaces},
\begin{equation*} 
L^p_T H^{s}_\rho:=L^p([0,T];H^{s}_\rho(\mathbb{R}^d))\, , \quad  \quad \|.\|_{L^p_T H^{s}_\rho}:=\|.\|_{L^p([0,T];H^{s}_\rho(\mathbb{R}^d))} \, .
\end{equation*}

\smallskip

\subsection{Outline of the study}

\

\smallskip

The rest of the paper is organized as follows. In Section~\ref{stocha} we perform the stochastic constructions which allow to give a sense to the equation. In Section~\ref{sec:regular} we prove the local well-posedness result in the regular case, while Section~\ref{sec:irreg-case-det} is devoted to the analysis in the irregular case. Finally, Section~\ref{appendix} is an appendix in which we gather the proofs of some technical results. 

\medskip

\textit{Throughout the paper, and for the sake of clarity, we will use the notation $A\lesssim B$ in order to signify that there exists an irrelevant constant $c$ such that $A\leq c B$.}

\

\section{Stochastic constructions}\label{stocha}

As emphasized in Section~\ref{sec:intro}, our analysis of equation~\eqref{dim-d} will be clearly splitted into a stochastic part (essentially corresponding to the construction of $\<Psi>$ and $\<Psi2>$) and a deterministic part (devoted to the fixed-point procedure). 

\smallskip

In this section, we propose to deal with the stochastic step of the study. In other words, we focus here on the proofs of Proposition~\ref{sto}, Proposition~\ref{prop:renorm-cstt} and Proposition~\ref{sto1}.

\smallskip

Before we go into the details, and for the sake of clarity, let us recall the specific definition of the process at the core of the model, that is the space-time fractional Brownian motion:

\begin{definition}\label{def:fbm}
Let $d \geq 1$ be a space dimension, $T\geq 0$ a positive time and $(\Omega,\mathcal{F},\mathbb{P})$ a complete filtered probability space. On this space, and for every fixed $H=(H_0,H_1,\dots,H_d) \in (0,1)^{d+1}$, we call a space-time fractional Brownian motion of Hurst index $H$ any centered Gaussian process $B:\Omega \times ([0,T]\times \R^d) \to \R$ with covariance function given by the formula: 
$$\mathbb{E} \big[ B(s,x_1,\dots,x_d)B(t,y_1,\dots, y_d)\big] = R_{H_0}(s,t)\prod_{i=1}^{d} R_{H_i}(x_i,y_i) \, , \quad  s,t\in [0,T] \, , \ x,y\in \R^d \, ,$$
where
$$R_{H_i}(x,y):=\frac12 (|x|^{2H_i}+|y|^{2H_i}-|x-y|^{2H_i}) \ .$$
\end{definition}

\

Now remember that our strategy to initiate the construction of both $\<Psi>$ and $\<Psi2>$ consists in the introduction of a smooth approximation $(B_n)_{n\geq 0}$ of $B$ (leading immediately to a smooth approximation $(\dot{B}_n)_{n\geq 0}$ of the noise $\dot{B}$). We will rely here on a sequence derived from the so-called harmonizable representation of $B$, and which happens to be especially suited for Fourier analysis and computations in Sobolev spaces (a similar choice is made in~\cite{deya-wave}, for the same reasons).

\smallskip

Along this idea, let us first introduce, on some complete probability space $(\Omega,\mathcal{F},\mathbb{P})$, a space-time white noise $W$ on $\R^{d+1}$. Then fix $H=(H_0,H_1,\dots,H_d)\in (0,1)^{d+1}$ and set, for all $t\in [0,T]$ and $x\in \R^d$,
\begin{equation}\label{harmo-repres}
B(t,x_1,\dots,x_d):= c\int_{\xi \in \mathbb{R}}\int_{\eta \in \mathbb{R}^d} \frac{e^{\imath t\xi}-1}{|\xi|^{H_0+\frac{1}{2}}}\prod_{i=1}^{d}\frac{e^{\imath x_i\eta_i}-1}{|\eta_i|^{H_i+\frac{1}{2}}}\, \widehat{W}(d\xi,d\eta)\, ,
\end{equation} 
for some constant $c$, and where $\widehat{W}$ stands for the Fourier transform of $W$. 

\smallskip

It is a well-established fact (see e.g.~\cite{harmo}) that for some appropriate value $c=c_H$ of the constant, the so-defined process $B$ is a space-time fractional Brownian motion of index $H$ (in the sense of Definition~\ref{def:fbm}). Our approximation of $B$ (for every fixed $H=(H_0,H_1,\dots,H_d)\in (0,1)^{d+1}$) is now obtained through a basic truncation of the integration domain in~\eqref{harmo-repres}: namely, we set $B_0\equiv 0$, and for $n\geq 1$,
\begin{equation}\label{defi:approx-b-n}
B_n(t,x_1,\dots,x_d)(\omega):= c_H\int_{|\xi|\leq 2^{2n}}^{}\int_{|\eta|\leq 2^n} \frac{e^{\imath t\xi}-1}{|\xi|^{H_0+\frac{1}{2}}}\prod_{i=1}^{d}\frac{e^{\imath x_i\eta_i}-1}{|\eta_i|^{H_i+\frac{1}{2}}}\, \widehat{W}(d\xi,d\eta)\, .
\end{equation}
By a quick examination of the possible singularities in $(\xi,\eta)$, it is not hard to see that, owing to the restricted integration domain, $B_n$ indeed defines a smooth process, for every fixed $n\geq 0$.

\begin{remark}
The choice of the scaling $\{|\xi|\leq 2^{2n}, |\eta|\leq 2^n\}$ in~\eqref{defi:approx-b-n} is directly related to the structure of the Schr{\"o}dinger operator, and will prove to be essential in the estimation of the renormalization constant (see Section~\ref{subsec:renorm-cstt}). This choice naturally contrasts with the \enquote{hyperbolic} scaling used in~\cite{deya-wave} for the corresponding wave model (see also Remark~\ref{rk:compari-lin}). Note also that the approximation \eqref{defi:approx-b-n} is the same as the one used in \cite{deya} for the study of a (rough) parabolic model.
\end{remark}

\smallskip

With approximation $(B_n)_{n\geq 0}$ in hand, we now would like to consider the sequence $(\<Psi>_n)_{n\geq 0}$ of approximated linear solutions, that is the sequence of solutions to 
\begin{equation}\label{regu}
 \left\{
    \begin{array}{ll}
  \imath\partial_t\<Psi>_n-\Delta \<Psi>_n = \dot B_n\, , \hskip 0.3 cm t \in [0,T]\, , \hspace{0,2cm} x \in \mathbb{R}^d\, ,   \\
 \  \<Psi>_n(0,.)=0\, ,
    \end{array}
\right.    
\end{equation}
where, for every $n\geq 0$, $\dot{B}_n$ is defined as the standard derivative $\dot{B}_n:=\partial_t\partial_{x_1}\cdots \partial_{x_d}B_n$. 

\smallskip

Note however that, without further integrability assumptions, the smoothness of $B_n$ (for each fixed $n\geq 0$) is not a sufficient condition to apply classical deterministic results immediately ensuring existence and uniqueness of $\<Psi>_n$. A possible way to circumvent the problem in this case is to rely on some stochastic interpretation of $\<Psi>_n$, based on the Gaussianity of the processes under consideration.

\smallskip

In order to justify this interpretation (that is, Definition~\ref{defi:luxo-n} below), let us go back to formula~\eqref{defi:approx-b-n} for $B_n$. Denoting by $S$ the $d$-dimensional Schr{\"o}dinger group and applying a Fubini-type theorem, the solution $\<Psi>_n$ can (at least formally) be written as
\small
\begin{align*}
&\<Psi>_n(t,x)=-\imath\int_{0}^{t}ds\int_{\mathbb{R}^d}dy\, S_{t-s}(x-y)\dot B_n(s,y)\\ 
&=c_H\imath^{d}\int_{0}^{t}ds\int_{\mathbb{R}^d}dy\, \int_{|\xi|\leq 2^{2n}}\int_{|\eta|\leq 2^n}S_{t-s}(x-y)\frac{\xi}{|\xi|^{H_0+\frac{1}{2}}}\prod_{i=1}^{d}\frac{\eta_i}{|\eta_i|^{H_i+\frac{1}{2}}}e^{\imath\xi s}e^{\imath\langle \eta,y \rangle}\widehat{W}(d\xi,d\eta)\\
&= c_H\imath^{d}\int_{|\xi|\leq 2^{2n}}^{}\int_{|\eta|\leq 2^n}\frac{\xi}{|\xi|^{H_0+\frac{1}{2}}}\prod_{i=1}^{d}\frac{\eta_i}{|\eta_i|^{H_i+\frac{1}{2}}}e^{\imath\langle \eta,x \rangle}\left[\int_{0}^{t}ds \, e^{\imath\xi s}\left(\int_{\mathbb{R}^d}dy\, S_{t-s}(x-y)e^{-\imath\langle \eta,x-y \rangle}\right)\right]\widehat{W}(d\xi,d\eta)\\
&= c_H\imath^{d}\int_{|\xi|\leq 2^{2n}}^{}\int_{|\eta|\leq 2^n}\frac{\xi}{|\xi|^{H_0+\frac{1}{2}}}\prod_{i=1}^{d}\frac{\eta_i}{|\eta_i|^{H_i+\frac{1}{2}}}e^{\imath\langle \eta,x \rangle}\left[\int_{0}^{t}ds \, e^{\imath\xi(t-s)}\left(\int_{\mathbb{R}^d}dy\, S_{s}(y)e^{-\imath\langle \eta,y \rangle}\right)\right]\widehat{W}(d\xi,d\eta)\, .
\end{align*}
\normalsize
At this point, remember that the spatial Fourier transform of $S$ is explicitly given by
$$\int_{\mathbb{R}^d}dx \, e^{-\imath\langle \xi,x \rangle}S_{t}(x)=e^{\imath t|\xi|^2}\, ,$$
and so we end up with the (a priori formal) representation
\begin{equation}\label{formula-heuris}
\<Psi>_n(t,x)=c_H\imath^{d}\int_{|\xi|\leq 2^{2n}}^{}\int_{|\eta|\leq 2^n}\frac{\xi}{|\xi|^{H_0+\frac{1}{2}}}\prod_{i=1}^{d}\frac{\eta_i}{|\eta_i|^{H_i+\frac{1}{2}}}e^{\imath\langle \eta,x \rangle}\gamma_t(\xi,|\eta|)\, \widehat{W}(d\xi,d\eta)\, ,
\end{equation}
where for all $t\geq 0$, $\xi \in \mathbb{R}$ and $r>0$, we define the quantity $\gamma_t(\xi,r)$ as 
\begin{equation}\label{defi-ga-t}
\gamma_t(\xi,r):=e^{\imath\xi t}\int_{0}^{t}e^{\imath s \{r^2-\xi\} }ds \, .
\end{equation}

\

Thanks to formula~\eqref{formula-heuris}, we are in a position to offer the following natural (and rigorous) stochastic definition for the solution of~\eqref{regu}:
\begin{definition}\label{defi:luxo-n}
We call a solution of equation~\eqref{regu} (or \emph{linear solution associated with~\eqref{dim-d}}) any centered complex Gaussian process 
$$\big\{\<Psi>_n(s,x), \, n\geq 1, \, s\geq 0, \, x\in \R^d \big\}$$
whose covariance is given by the relations: for all $n,m\geq 1$, $s, t\geq 0$ and $x, y \in \mathbb{R}^d$,
\begin{equation}\label{cova-luxo}
\mathbb{E}\Big[\<Psi>_n(s,x)\overline{\<Psi>_m(t,y)}\Big]=c_H^2\int_{(\xi,\eta)\in D_n \cap D_m}\frac{1}{|\xi|^{2H_0-1}}\prod\limits_{i=1}^{d}\frac{1}{|\eta_i|^{2H_i-1}}\gamma_{s}(\xi,|\eta|)\overline{\gamma_{t}(\xi,|\eta|)}e^{\imath \langle \eta, x-y \rangle}\, d\xi d\eta \, ,
\end{equation}
\begin{equation}\label{cova-luxo-2}
\mathbb{E}\Big[\<Psi>_n(s,x)\<Psi>_m(t,y)\Big]=-c_H^2\int_{(\xi,\eta)\in D_n \cap D_m}\frac{1}{|\xi|^{2H_0-1}}\prod\limits_{i=1}^{d}\frac{1}{|\eta_i|^{2H_i-1}}\gamma_{s}(\xi,|\eta|)\gamma_{t}(-\xi,|\eta|)e^{\imath \langle \eta, x-y \rangle}\, d\xi d\eta \, ,
\end{equation}
where $D_n:=B_{2n}^1 \times B_n^d$ with $B_\ell^k:=\left\{ \lambda \in \mathbb{R}^k: |\lambda| \leq 2^\ell \right\}$.
\end{definition}

\

\subsection{Preliminary estimates}

As a first step toward Proposition~\ref{sto} and Proposition~\ref{sto1}, let us collect some essential estimates on the quantity $\gamma_t(\xi,r)$ at the core of the covariance formula~\eqref{cova-luxo} (and explicitly defined by~\eqref{defi-ga-t}).

\smallskip

To this end, we set, for all $0\leq s \leq t$, $\xi \in \mathbb{R}$ and $r>0$,
\begin{equation*}
\gamma_{s,t}(\xi,r):=\gamma_t(\xi,r)-\gamma_s(\xi,r) \, .
\end{equation*}

\smallskip

\begin{lemma}\label{lem}
For all $T\geq 0$, $0\leq s \leq t\leq T$, $\xi \in \mathbb{R}$, $r>0$ and $\kappa, \lambda \in [0,1]$, it holds that 
$$
|\gamma_{s,t}(\xi,r)|\lesssim \min\bigg( |\xi|^{\kappa}|t-s|^{\kappa} + |t-s|,\frac{|t-s| r^2 }{|\xi|}+  \frac{|t-s|^{\kappa}\{1+r^2\}}{|\xi|^{1-\kappa}},\frac{|t-s|^{\kappa}\{r^{2\kappa}+|\xi|^{\kappa}\}}{||\xi|-r^2|^{1-\lambda(1-\kappa)}}\bigg)\, ,
$$
where the proportional constant in $\lesssim$ only depends on $T$.
\end{lemma}

\begin{proof} To begin with, let us write
$$\gamma_{s,t}(\xi,r)=\{e^{\imath\xi t}-e^{\imath\xi s}\}\int_{0}^{s}e^{\imath u \{r^2-\xi\} }du + e^{\imath \xi t}\int_{s}^{t}e^{\imath u\{r^2-\xi\}}du \, ,$$ 
and so
$$
|\gamma_{s,t}(\xi,r)|\lesssim |e^{\imath\xi t}-e^{\imath\xi s}|\bigg|\int_{0}^{s}e^{\imath u\{r^2-\xi\}}du\bigg| +\bigg|\int_{s}^{t}e^{\imath u\{r^2-\xi\}}du\bigg|\lesssim |\xi|^{\kappa}|t-s|^{\kappa} + |t-s|\, .
$$
Then observe that 
$$\gamma_{t}(\xi,r)=e^{\imath\xi t}\int_{0}^{t}e^{\imath s \{r^2-\xi\} }ds=-\frac{e^{\imath r^2 t}-e^{\imath \xi t}}{\imath\xi}+\frac{e^{\imath \xi t}r^2}{\xi}\int_{0}^{t}e^{\imath s \{r^2-\xi\} }ds \, , $$
which readily entails 
\begin{align*}
&\gamma_{s,t}(\xi,r)\\
&=-\frac{\{e^{\imath r^2 t}-e^{\imath r^2 s}\}-\{e^{\imath \xi t}-e^{\imath \xi s}\}}{\imath\xi}+\frac{r^2}{\xi}\{e^{\imath\xi t}-e^{\imath\xi s}\}\int_{0}^{s}e^{\imath u \{r^2-\xi\} }du+\frac{e^{\imath \xi t}r^2}{\xi}\int_{s}^{t}e^{\imath u \{r^2-\xi\} }du \, .
\end{align*}
Thus,
\begin{eqnarray*}
|\gamma_{s,t}(\xi,r)|&\lesssim& r^2 \frac{|t-s|}{|\xi|}+\frac{|t-s|^{\kappa}}{|\xi|^{1-\kappa}}+r^2\frac{|t-s|^{\kappa}}{|\xi|^{1-\kappa}}+\frac{r^2}{|\xi|}|t-s|\\
&\lesssim& r^2 \frac{|t-s|}{|\xi|}+ \{1+r^2\} \frac{|t-s|^{\kappa}}{|\xi|^{1-\kappa}}\, .
\end{eqnarray*}
Finally, it can be checked that 
$$\gamma_{t}(\xi,r)=\frac{\imath}{r^2-\xi}\{e^{\imath \xi t}-e^{\imath r^2 t}\}\, ,$$
which yields
\begin{eqnarray*}
|\gamma_{s,t}(\xi,r)|&=& \frac{1}{|\xi-r^2|}\Big|\{e^{\imath r^2 t}-e^{\imath r^2 s}\}-\{e^{\imath \xi t}-e^{\imath \xi s}\}\Big|^{\kappa}\Big|\{e^{\imath r^2 t}-e^{\imath \xi t}\}-\{e^{\imath r^2 s}-e^{\imath \xi s}\}\Big|^{1-\kappa}\\
%&\lesssim&\frac{1}{|\xi-r^2|}\{|e^{i r^2 t}-e^{i r^2 s}|+|e^{i \xi t}-e^{i \xi s}|\}^{\kappa}\{|e^{i r^2 t}-e^{i \xi t}|+|e^{i r^2 s}-e^{i \xi s}|\}^{1-\kappa}\\
&\lesssim&\frac{|t-s|^{\kappa}}{|\xi-r^2|}\{r^2+|\xi|\}^{\kappa}\{|e^{\imath r^2 t}-e^{\imath \xi t}|^{\lambda}+|e^{\imath r^2 s}-e^{\imath \xi s}|^{\lambda}\}^{1-\kappa}\\
%&\lesssim&\frac{|t-s|^{\kappa}|\xi-r^2|^{\lambda(1-\kappa)}}{|\xi-r^2|}\{r^2+|\xi|\}^{\kappa}\{t^{\lambda}+s^{\lambda}\}^{1-\kappa}\\
&\lesssim&\frac{|t-s|^{\kappa}}{|\xi-r^2|^{1-\lambda(1-\kappa)}}\{r^{2\kappa}+|\xi|^{\kappa}\}\ \lesssim \ \frac{|t-s|^{\kappa}}{||\xi|-r^2|^{1-\lambda(1-\kappa)}}\{r^{2\kappa}+|\xi|^{\kappa}\} \, .
\end{eqnarray*}
\end{proof}

\begin{corollary}\label{tec}
For all $T\geq 0$, $0\leq s \leq t \leq T$, $H\in (0,1)$, $r>0$, $\varepsilon\in (0,1)$ and $\kappa\in [0, \min(H, \frac{1-\varepsilon}{2}))$, it holds that
$$ \int_{\mathbb{R}}\frac{|\gamma_{s,t}(\xi,r)|^2}{|\xi|^{2H-1}}\, d\xi \lesssim \frac{|t-s|^{2\kappa}}{1+r^{4(H-\kappa)-2-2\varepsilon}}\, ,$$
where the proportional constant in $\lesssim$ only depends on $T$.
\end{corollary}

\begin{proof}
We will naturally lean on the estimates exhibited in Lemma~\ref{lem}.

\smallskip

For $0<r<1$, we have
$$\int_{\mathbb{R}}\frac{|\gamma_{s,t}(\xi,r)|^2}{|\xi|^{2H-1}}\, d\xi\lesssim |t-s|^{2\kappa}\bigg[\int_{|\xi|\leq 1}\frac{d\xi}{|\xi|^{2H-1}}+\int_{|\xi|\geq 1}\frac{d\xi}{|\xi|^{2(H-\kappa)+1}}\bigg]\lesssim |t-s|^{2\kappa}\, .$$
Then, for $r>1$, let us consider the decomposition 
$$\int_{\mathbb{R}}\frac{|\gamma_{s,t}(\xi,r)|^2}{|\xi|^{2H-1}}\, d\xi=  \int_{||\xi|-r^2|\geq \frac{|\xi|}{2}}\frac{|\gamma_{s,t}(\xi,r)|^2}{|\xi|^{2H-1}}\, d\xi+ \int_{||\xi|-r^2|\leq \frac{|\xi|}{2}}\frac{|\gamma_{s,t}(\xi,r)|^2}{|\xi|^{2H-1}}\, d\xi\, .$$
On the one hand, it holds that
\begin{eqnarray*}
\int_{||\xi|-r^2|\geq \frac{|\xi|}{2}}\frac{|\gamma_{s,t}(\xi,r)|^2}{|\xi|^{2H-1}}\, d\xi &\lesssim& |t-s|^{2\kappa}\int_{\left\{|\xi|\leq \frac{2}{3}r^2\right\}\cup \left\{ |\xi|\geq 2r^2\right\}} \frac{r^{4\kappa}+|\xi|^{2\kappa}}{|\xi|^{2H-1}||\xi|-r^2|^2}\, d\xi \\
&\lesssim& \frac{|t-s|^{2\kappa}}{r^{4(H-\kappa)}}\int_{\left\{|\xi|\leq \frac{2}{3}\right\}\cup \left\{ |\xi|\geq 2\right\}} \frac{1+|\xi|^{2\kappa}}{|\xi|^{2H-1}||\xi|-1|^2}\, d\xi \ \lesssim \ \frac{|t-s|^{2\kappa}}{r^{4(H-\kappa)}}\, .
\end{eqnarray*}
On the other hand, setting $\lambda=\frac{1+\varepsilon}{2(1-\kappa)}$, one has 
\begin{eqnarray*}
\int_{||\xi|-r^2|\leq \frac{|\xi|}{2}}\frac{|\gamma_{s,t}(\xi,r)|^2}{|\xi|^{2H-1}}d\xi &\lesssim& |t-s|^{2\kappa}\int_{\left\{\frac{2}{3} r^2\leq |\xi| \leq 2r^2\right\}} \frac{r^{4\kappa}+|\xi|^{2\kappa}}{|\xi|^{2H-1}||\xi|-r^2|^{2-2\lambda(1-\kappa)}}\, d\xi\\
&\lesssim& \frac{|t-s|^{2\kappa}}{r^{4(H-\kappa)-4\lambda(1-\kappa)}}\int_{\left\{\frac{2}{3} \leq |\xi| \leq 2\right\}} \frac{1}{|\xi|^{2H-1}||\xi|-1|^{2-2\lambda(1-\kappa)}}\, d\xi \\
&\lesssim& \frac{|t-s|^{2\kappa}}{r^{4(H-\kappa)-2-2\varepsilon}} \, .
\end{eqnarray*}
\end{proof}

\smallskip

Let us also take advantage of this preliminary section to introduce the following lemma, which accounts for the technical simplifications offered by the test function $\chi$ in the forthcoming computations.

\begin{lemma}\label{lem:handle-chi-correc}
Let $\chi:\R^d \to \R$ be a test function and fix $\si\in \R$. Then, for every $p\geq 1$ and for all $\eta_1,\ldots,\eta_p\in \R^d$, it holds that
$$\bigg| \int_{\R^d} dx \, \prod_{i=1}^p \int_{(\R^d)^2} \frac{d\la_i d\lati_i}{\{1+|\la_i|^2\}^{\frac{\si}{2}}\{1+|\lati_i|^2\}^{\frac{\si}{2}}} e^{\imath \langle x,\la_i-\lati_i\rangle} \widehat{\chi}(\la_i-\eta_i)\overline{\widehat{\chi}(\lati_i-\eta_i)}  \bigg| \lesssim \prod_{i=1}^p \frac{1}{\{1+|\eta_i|^2\}^\si} \, ,$$
where the proportional constant only depends on $\chi$ and $\si$. 
\end{lemma}

\begin{proof}
Let us set, for all $\eta,\la\in \R^d$,
$$\gga_\eta(\la):=\widehat{\chi}(\la-\eta) \frac{1}{\{1+|\la|^2\}^{\frac{\si}{2}}} \, $$
and observe that the integral under consideration can then be written as
\begin{align}
&\int_{\R^d} dx \, \prod_{i=1}^p \int_{(\R^d)^2} \frac{d\la_i d\lati_i}{\{1+|\la_i|^2\}^{\frac{\si}{2}}\{1+|\lati_i|^2\}^{\frac{\si}{2}}} e^{\imath \langle x,\la_i-\lati_i\rangle} \widehat{\chi}(\la_i-\eta_i)\overline{\widehat{\chi}(\lati_i-\eta_i)} \nonumber\\
&=c\int_{\R^d} dx \, \prod_{i=1}^p  \big| \cf^{-1}\big( \gga_{\eta_i}\big)(x) \big|^2\nonumber\\
&=c\int_{\R^d} dx \,  \big| \cf^{-1}\big( \gga_{\eta_1}\ast \cdots \ast \gga_{\eta_p}\big)(x) \big|^2\nonumber\\
&\lesssim \big\| \gga_{\eta_1}\ast \cdots \ast \gga_{\eta_p} \big\|_{L^2(\R^d)}^2 \lesssim \big\| \gga_{\eta_1}\big\|_{L^1(\R^d)}^2 \cdots \big\|  \gga_{\eta_{p-1}} \big\|_{L^1(\R^d)}^2\big\| \gga_{\eta_p} \big\|_{L^2(\R^d)}^2 \, ,\label{rie-tho}
\end{align}
where the last inequality is derived from Plancherel theorem.

\smallskip

The conclusion now comes from the fact that for all $\eta,\la\in \R^d$ and for every $\ka >0$,
\begin{equation}\label{boun-chi-lem-correc}
\big|\widehat{\chi}(\la) \{1+|\la+\eta|^2\}^{-\frac{\si}{2}}\big| \leq c_{\si,\chi,\ka} \{1+|\la|^2\}^{-\ka} \{1+|\eta|^2\}^{-\frac{\si}{2}} \, .
\end{equation}
Let us briefly verify~\eqref{boun-chi-lem-correc} for $\si \geq 0$ (the proof for $\si<0$ being immediate). In fact, since $\chi$ is smooth and compactly-supported, one has
\begin{align}
&\big| \hat{\chi}(\la)\{1+|\la+\eta|^2\}^{-\frac{\si}{2}} \big|\nonumber\\
&=\big| \hat{\chi}(\la)\{1+|\la+\eta|^2\}^{-\frac{\si}{2}} \big| \mathbf{1}_{\{|\la|\geq\frac12 |\eta|\}}+\big| \hat{\chi}(\la)\{1+|\la+\eta|^2\}^{-\frac{\si}{2}} \big| \mathbf{1}_{\{|\la|<\frac12 |\eta|\}}\nonumber\\
&\leq c_\si \Big[ |\hat{\chi}(\la)| \mathbf{1}_{\{|\la|\geq\frac12 |\eta|\}}+|\hat{\chi}(\la)|\{1+|\eta|^2\}^{-\frac{\si}{2}}  \mathbf{1}_{\{|\la|<\frac12 |\eta|\}}\Big]\nonumber\\
&\leq c_{\si,\chi,\ka}\Big[ \{1+|\la|^2\}^{-\ka} \{1+|\la|^2\}^{-\frac{\si}{2}} \mathbf{1}_{\{|\la|\geq\frac12 |\eta|\}}+\{1+|\la|^2\}^{-\ka}\{1+|\eta|^2\}^{-\frac{\si}{2}}  \mathbf{1}_{\{|\la|<\frac12 |\eta|\}}\Big]\nonumber\\
&\leq c_{\si,\chi,\ka} \{1+|\la|^2\}^{-\ka} \{1+|\eta|^2\}^{-\frac{\si}{2}} \, .\nonumber%\label{correc-3}
\end{align}
Going back to~\eqref{rie-tho}, and with estimate~\eqref{boun-chi-lem-correc} in hand, we can check for instance that
$$\big\| \gga_{\eta_1}\big\|_{L^1(\R^d)} =\int_{\R^d} d\la \, \big|\widehat{\chi}(\la)\big| \frac{1}{\{1+|\la+\eta_1|^2\}^{\frac{\si}{2}}}\leq \frac{c_{\si,\chi}}{\{1+|\eta_1|^2\}^{\frac{\si}{2}}} \, ,$$
which yields the desired bound.
\end{proof}

\

\subsection{Proof of Proposition~\ref{sto}}\label{subsec:proof-sto-1}
For simplicity, we will assume for the whole proof that $T= 1$ and we set, for all $m, n\geq 0$, $\<Psi>_{n,m}:=\<Psi>_m-\<Psi>_n$. Besides, along the statement of the proposition, we fix $\al$ satisfying~\eqref{assump-gene-al}.

\

\noindent
\textbf{Step 1:}
Let us show that for all $p\geq 1$, $1\leq n \leq m$, $0\leq s \leq t \leq 1$ and $\varepsilon,\ka >0$ small enough, one has
\begin{equation}\label{bou-psi-step-1}
 \int_{\mathbb{R}^d}\mathbb{E}\bigg[\Big|\mathcal{F}^{-1}\Big(\{1+|.|^2\}^{-\frac{\alpha}{2}}\mathcal{F}\big(\chi \big[\<Psi>_{n,m}(t,.)-\<Psi>_{n,m}(s,.)\big]\big)\Big)(x)\Big|^{2p}\bigg]\, dx \lesssim 2^{-4n \varepsilon p}|t-s|^{2 \ka p}\, ,
\end{equation}
where the proportional constant only depends on $p, \alpha$ and $\chi$.

\

First, we can observe that the random variable under consideration is Gaussian, and so, for every $p\geq 1$, one has
\begin{align}
&\mathbb{E}\bigg[\Big|\mathcal{F}^{-1}\Big(\{1+|.|^2\}^{-\frac{\alpha}{2}}\mathcal{F}\big(\chi \big[\<Psi>_{n,m}(t,.)-\<Psi>_{n,m}(s,.)\big]\big)\Big)(x)\Big|^{2p}\bigg]\nonumber\\
&\leq c_p\,  \mathbb{E}\bigg[\Big|\mathcal{F}^{-1}\Big(\{1+|.|^2\}^{-\frac{\alpha}{2}}\mathcal{F}\big(\chi \big[\<Psi>_{n,m}(t,.)-\<Psi>_{n,m}(s,.)\big]\big)\Big)(x)\Big|^{2}\bigg]^p\, ,\label{appl-hyper}
\end{align}
where the constant $c_p$ only depends on $p$.

\smallskip

Let us then expand this variable as 
\begin{align*}
&\mathcal{F}^{-1}\Big(\{1+|.|^2\}^{-\frac{\alpha}{2}}\mathcal{F}\big(\chi \big[\<Psi>_{n,m}(t,.)-\<Psi>_{n,m}(s,.)\big]\big)\Big)(x)\\
&=\frac{1}{(2\pi)^d}\int_{\mathbb{R}^d}d\lambda \, e^{\imath \langle x, \lambda \rangle} \{1+|\lambda|^2\}^{-\frac{\alpha}{2}}\mathcal{F}\big(\chi \big[\<Psi>_{n,m}(t,.)-\<Psi>_{n,m}(s,.)\big]\big)(\lambda)\\
&=\frac{1}{(2\pi)^d}\int_{\mathbb{R}^d}d\la \, \{1+|\lambda|^2\}^{-\frac{\alpha}{2}}e^{\imath \langle x, \lambda \rangle}\left(\int_{\mathbb{R}^d}d\be\, \widehat{\chi}(\la-\be) \cf\big(\big[\<Psi>_{n,m}(t,.)-\<Psi>_{n,m}(s,.)\big]\big)(\be)\right)\, ,
%&=\int_{\mathbb{R}^d}\int_{\mathbb{R}^d}e^{i \langle x, \lambda \rangle}\{1+|\lambda|^2\}^{-\frac{\alpha}{2}}e^{-i \langle y, \lambda \rangle}\chi(y) [\<Psi>_{n,m}(t,y)-\<Psi>_{n,m}(s,y)]\, dyd\lambda \, .
\end{align*}
so that
\begin{align}
&\mathbb{E}\bigg[\Big|\mathcal{F}^{-1}\Big(\{1+|.|^2\}^{-\frac{\alpha}{2}}\mathcal{F}\big(\chi \big[\<Psi>_{n,m}(t,.)-\<Psi>_{n,m}(s,.)\big]\big)\Big)(x)\Big|^{2}\bigg]\nonumber\\
&=\frac{1}{(2\pi)^{2d}}\iint_{(\R^d)^2} \frac{d\la d\lati}{\{1+|\la|^2\}^{\frac{\al}{2}} \{1+|\lati|^2\}^{\frac{\al}{2}}}e^{\imath \langle x,\la-\lati\rangle} \iint_{(\R^d)^2} d\be d\betati \, \widehat{\chi}(\la-\be) \overline{\widehat{\chi}(\lati-\betati)} \mathcal{Q}_{n,m;s,t}(\be,\betati) \, ,\label{bound-correc-q}
\end{align}
with
$$\mathcal{Q}_{n,m;s,t}(\be,\betati):=\mathbb{E}\Big[\cf\big(\big[\<Psi>_{n,m}(t,.)-\<Psi>_{n,m}(s,.)\big]\big)(\be) \overline{\cf\big(\big[\<Psi>_{n,m}(t,.)-\<Psi>_{n,m}(s,.)\big]\big)(\betati)}\Big] \, .$$
To expand the latter quantity, observe first that, using the covariance formula~\eqref{cova-luxo}, we get 
\begin{align*}
&\mathbb{E}\Big[\big\{\<Psi>_{n,m}(t,y)-\<Psi>_{n,m}(s,y)\big\}\overline{\big\{\<Psi>_{n,m}(t,\tilde{y})-\<Psi>_{n,m}(s,\tilde{y})\big\}}\Big]\nonumber\\
&= c\int_{\left\{(\xi,\eta)\in D_{n,m}\right\}}\frac{1}{|\xi|^{2H_0-1}}\prod_{i=1}^{d}\frac{1}{|\eta_i|^{2H_i-1}}|\gamma_{s,t}(\xi,|\eta|)|^2e^{\imath \langle \eta, y \rangle}e^{-\imath\langle \eta, \tilde{y} \rangle}d\xi d\eta \, ,
\end{align*}
where $D_{n,m}:=(B_{2m}^1 \times B_m^d)\setminus (B_{2n}^1 \times B_n^d)$, and hence
\begin{equation}\label{expr-q}
\mathcal{Q}_{n,m;s,t}(\beta,\betati)=c\,  \int_{\left\{(\xi,\eta)\in D_{n,m}\right\}}d\xi d\eta\frac{1}{|\xi|^{2H_0-1}}\prod_{i=1}^{d}\frac{1}{|\eta_i|^{2H_i-1}}|\gamma_{s,t}(\xi,|\eta|)|^2 \delta_{\beta=\eta} \delta_{\betati=\eta} \, .
\end{equation}
Combining~\eqref{appl-hyper}-\eqref{bound-correc-q}-\eqref{expr-q} and using a standard Fubini argument, we end up with the estimate
\begin{align*}
& \int_{\mathbb{R}^d}dx\, \mathbb{E}\bigg[\Big|\mathcal{F}^{-1}\Big(\{1+|.|^2\}^{-\frac{\alpha}{2}}\mathcal{F}\big(\chi \big[\<Psi>_{n,m}(t,.)-\<Psi>_{n,m}(s,.)\big]\big)\Big)(x)\Big|^{2p}\bigg]\\
&\lesssim \int_{\R^d} dx \, \bigg(\int_{\left\{(\xi,\eta)\in D_{n,m}\right\}}d\xi d\eta\frac{1}{|\xi|^{2H_0-1}}\prod_{i=1}^{d}\frac{1}{|\eta_i|^{2H_i-1}}|\gamma_{s,t}(\xi,|\eta|)|^2\\
&\hspace{3cm} \iint_{(\R^d)^2} \frac{d\la d\lati}{\{1+|\la|^2\}^{\frac{\al}{2}} \{1+|\lati|^2\}^{\frac{\al}{2}}}e^{\imath \langle x,\la-\lati\rangle}  \, \widehat{\chi}(\la-\eta) \overline{\widehat{\chi}(\lati-\eta)} \bigg)^p\\
&\lesssim  \left(\int_{(\xi,\eta)\in D_{n,m}}\frac{1}{|\xi|^{2H_0-1}}\prod_{i=1}^{d}\frac{1}{|\eta_i|^{2H_i-1}}\{1+|\eta|^2\}^{-\alpha}|\gamma_{s,t}(\xi,|\eta|)|^2\, d\xi d\eta\right)^p \, ,
\end{align*}
where we have used Lemma~\ref{lem:handle-chi-correc} to get the last inequality.

\smallskip

Now we can obviously write
\begin{align}
& \left(\int_{(\xi,\eta)\in D_{n,m}}\frac{1}{|\xi|^{2H_0-1}}\prod_{i=1}^{d}\frac{1}{|\eta_i|^{2H_i-1}}\{1+|\eta|^2\}^{-\alpha}|\gamma_{s,t}(\xi,|\eta|)|^2\, d\xi d\eta\right)^p \nonumber\\
&\lesssim  \left(\int_{2^{2n}\leq |\xi|\leq 2^{2m}}\int_{|\eta|\leq 2^m}\frac{1}{|\xi|^{2H_0-1}}\prod_{i=1}^{d}\frac{1}{|\eta_i|^{2H_i-1}}\{1+|\eta|^2\}^{-\alpha}|\gamma_{s,t}(\xi,|\eta|)|^2 \, d\xi d\eta\right)^p \nonumber\\
&\hspace{1cm}+ \left(\int_{|\xi|\leq 2^{2m}}\int_{2^n\leq |\eta|\leq 2^m}\frac{1}{|\xi|^{2H_0-1}}\prod_{i=1}^{d}\frac{1}{|\eta_i|^{2H_i-1}}\{1+|\eta|^2\}^{-\alpha}|\gamma_{s,t}(\xi,|\eta|)|^2\, d\xi d\eta\right)^p \nonumber\\
&=:(\mathbb{I}_{n,m}(s,t))^p+(\mathbb{II}_{n,m}(s,t))^p \, . \label{decompos-i-ii}
\end{align}
Let us focus on the estimation of $\mathbb{I}_{n,m}(s,t)$. To this end, we fix $0<\varepsilon<\min\big(H_0,\frac12\big)$, so that
%$$\frac{1}{|\xi|^{2H_0-1}}=\frac{1}{|\xi|^{2H_0-2\varepsilon-1}} \frac{1}{|\xi|^{2\varepsilon}}\leq 2^{-4n\varepsilon}\frac{1}{|\xi|^{2H_0-2\varepsilon-1}}\, .$$
%Hence
\begin{equation}\label{element-bounding}
\mathbb{I}_{n,m}(s,t)\leq 2^{-4n\varepsilon}  \int_{\mathbb{R}}d\xi\int_{\R^d}d\eta \, \frac{1}{|\xi|^{2H_0-2\varepsilon-1}}\prod_{i=1}^{d}\frac{1}{|\eta_i|^{2H_i-1}}\{1+|\eta|^2\}^{-\alpha}|\gamma_{s,t}(\xi,|\eta|)|^2\, ,
\end{equation}
which, through an elementary hyperspherical change of variables, leads us to
\begin{equation}\label{chg-of-var}
\mathbb{I}_{n,m}(s,t)\lesssim 2^{-4n\varepsilon} \int_{0}^{\infty}dr \, \frac{\left\{1+r^2\right\}^{-\alpha}}{r^{2(H_1+\cdots+H_d)-2d+1}}\left(\int_{\mathbb{R}}d\xi\, \frac{|\gamma_{s,t}(\xi,r)|^2}{|\xi|^{2H_0-2\varepsilon-1}}\right) \, .
\end{equation}
We can now apply Corollary~\ref{tec} with $H:=H_0-\varepsilon$, which gives, for all $0<\kappa< \min(H_0-\varepsilon,\frac{1}{2}-\varepsilon)$,
\begin{align}
&\mathbb{I}_{n,m}(s,t)\nonumber\\
&\lesssim 2^{-4n\varepsilon} |t-s|^{2\kappa}\int_{0}^{\infty}dr \, \frac{\left\{1+r^2\right\}^{-\alpha}}{r^{2(H_1+\cdots+H_d)-2d+1}}\frac{1}{1+r^{4H_0-4\kappa-2-8\varepsilon}}\nonumber\\
&\lesssim 2^{-4n\varepsilon}|t-s|^{2\kappa}\left(\int_0^1 \frac{1}{r^{2(H_1+\cdots+H_d)-2d+1}}dr + \int_{1}^{\infty}\frac{1}{r^{2\alpha+2(2H_0+H_1+\cdots+H_d)-2d-1-4\kappa -8\varepsilon}}dr \right)\, .\label{boun-i-pr}
\end{align}
Due to our assumption~\eqref{assump-gene-al}, we can in fact pick $\varepsilon$ and $\ka$ sufficiently small so that
$$4\varepsilon + 2\kappa< \alpha-\bigg[d+1-\bigg(2H_0+\sum_{i=1}^{d}H_i\bigg)\bigg] \, ,$$
and for such a choice, the integrals involved in~\eqref{boun-i-pr} are obviously finite, implying that
$$\mathbb{I}_{n,m}(s,t)\lesssim 2^{-4n\varepsilon}|t-s|^{2\kappa}\, .$$
It is easy to see that the above estimates could also be used to bound $\mathbb{II}_{n,m}(s,t)$, leading to the very same estimate
$$\mathbb{II}_{n,m}(s,t)\lesssim 2^{-4n\varepsilon}|t-s|^{2\kappa}\, .$$
Going back to~\eqref{decompos-i-ii}, we deduce the desired bound~\eqref{bou-psi-step-1}.

\

\noindent
\textbf{Step 2:}
The estimate obtained in the previous step can naturally be rephrased as
\begin{equation}\label{recap-step-1}
\mathbb{E}\Big[\big\|\chi \<Psi>_{n,m}(t,.)-\chi\<Psi>_{n,m}(s,.)\big\|_{\mathcal{W}^{-\alpha,2p}}^{2p}\Big]\lesssim 2^{-4n \varepsilon p}|t-s|^{2 \ka p}\, ,
\end{equation}
for all $p\geq 1$, $1\leq n \leq m$, $0\leq s \leq t \leq 1$ and $\varepsilon,\ka >0$ small enough.

\smallskip

By choosing $p\geq 1$ large enough so that $2\ka p>1$, Kolmogorov continuity criterion allows us to assert that $\chi \<Psi>_{n,m} \in \mathcal{C}([0,T]; \mathcal{W}^{-\alpha,2p}(\mathbb{R}^d)) $ almost surely. In turn, this puts us in a position to use the classical Garsia-Rodemich-Rumsey estimate (see~\cite{GRR}) and deduce that almost surely, for all $p\geq 1$, $0<\ka_0<\ka$, $0\leq s \leq t \leq 1$, one has
$$ \|\chi \<Psi>_{n,m}(t,.)-\chi\<Psi>_{n,m}(s,.)\|_{\mathcal{W}^{-\alpha,2p}}^{2p} \lesssim |t-s|^{2\ka_0 p}\int_{[0,1]^2} \frac{\|\chi \<Psi>_{n,m}(u,.)-\chi\<Psi>_{n,m}(v,.)\|_{\mathcal{W}^{-\alpha,2p}}^{2p}}{|u-v|^{2\ka_0 p+2}}\, dudv\, ,$$
for some proportional constant that only depends on $\ka_0$ and $p$.

\smallskip

Picking $s=0$ and taking the supremum over $t\in [0,1]$, we derive
$$ \|\chi\<Psi>_{n,m}\|_{\mathcal{C}_T\mathcal{W}^{-\alpha,2p}}^{2p} \lesssim\int_{[0,1]^2}\frac{\|\chi \<Psi>_{n,m}(u,.)-\chi\<Psi>_{n,m}(v,.)\|_{\mathcal{W}^{-\alpha,2p}}^{2p}}{|u-v|^{2\ka_0 p+2}}dudv\, ,$$
and therefore, using~\eqref{recap-step-1} again, we obtain that
\begin{eqnarray*}
\mathbb{E}\Big[\|\chi\<Psi>_{n,m}\|_{\mathcal{C}_T\mathcal{W}^{-\alpha,2p}}^{2p}\Big]
%&\lesssim&\int_{[0,1]^2}\frac{\mathbb{E}\big[[\|\chi \<Psi>_{n,m}(u,.)-\chi\<Psi>_{n,m}(v,.)\|_{\mathcal{W}^{-\alpha,p}(\mathbb{R}^d)}^{p}\big]}{|u-v|^{\varepsilon_0 p+2}}\, dudv \\
&\lesssim& 2^{-4n \varepsilon p}\int_{[0,1]^2}\frac{dudv}{|u-v|^{-2(\ka-\ka_0) p+2}}\ \lesssim \ 2^{-4n \varepsilon p} \, ,
\end{eqnarray*}
for any $p\geq 1$ large enough so that $-2(\ka-\ka_0) p+2<1$.

\

We can conclude that for any $p\geq 1$ large enough,
\begin{equation}\label{bou-in-l-p}
\|\chi\<Psi>_{n,m}\|_{ L^{2p}(\Omega; \mathcal{C}_T\mathcal{W}^{-\alpha,2p})} \lesssim 2^{-2n \varepsilon }\, .
\end{equation}
In particular, $(\chi\<Psi>_{n})_{n\geq 1}$ is a Cauchy sequence in $L^{2p}(\Omega; \mathcal{C}([0,T]; \mathcal{W}^{-\alpha,2p}(\mathbb{R}^d)))$ (for any $p\geq 1$ large enough), which entails convergence in this space to a limit $\chi\<Psi>$. Going back to~\eqref{bou-in-l-p}, we also have 
$$
\|\chi\<Psi>-\chi\<Psi>_{n}\|_{ L^{2p}(\Omega; \mathcal{C}_T\mathcal{W}^{-\alpha,2p})}  \lesssim 2^{-2n \varepsilon }\, ,
$$
and from there, a standard Borell-Cantelli argument provides us with the desired almost sure convergence of $(\chi\<Psi>_{n})_{n\geq 1}$ to $\chi\<Psi>$ in $\mathcal{C}([0,T]; \mathcal{W}^{-\alpha,p}(\mathbb{R}^d))$, for every $2\leq p<\infty$. Finally, the convergence in $\mathcal{C}([0,T]; \mathcal{W}^{-\alpha,\infty}(\mathbb{R}^d))$ follows from the Sobolev embedding $\cw^{-\al+\frac{d}{p}+\eta,p}(\R^d) \subset \cw^{-\al,\infty}(\R^d)$, for any $\eta>0$.

\

\subsection{Proof of Proposition~\ref{sto1}}\label{subsec:cerise}
We will follow the same general scheme as in the proof of Proposition~\ref{sto}. Just as in Section~\ref{subsec:proof-sto-1}, let us assume that $T= 1$, and set, for all $m, n\geq 0$, $\<Psi2>_{n,m}:=\<Psi2>_m-\<Psi2>_n$.

\

\noindent
\textbf{Step 1:}
Our main objective here is to show that for all $p\geq 1$, $0\leq n \leq m$, $0\leq s \leq t \leq 1$, $\varepsilon,\ka>0$ small enough, and for every $\al$ satisfying
\begin{equation}\label{assump-al-bis}
d+1-\bigg(2H_0+\sum_{i=1}^{d}H_i\bigg)<\al <\frac14 \, ,
\end{equation}
one has
\begin{equation}\label{mom-estim-wick-sq}
 \int_{\mathbb{R}^d}\mathbb{E}\bigg[\Big|\mathcal{F}^{-1}\Big(\{1+|.|^2\}^{-\alpha}\mathcal{F}\big(\chi^2 \big[\<Psi2>_{n,m}(t,.)-\<Psi2>_{n,m}(s,.)\big]\big)\Big)(x)\Big|^{2p}\bigg]dx \lesssim 2^{-4n \varepsilon p}|t-s|^{\ka p}\, ,
\end{equation}
where the proportional constant only depends on $p, \alpha,$ and $\chi$. 

\smallskip

For the sake of conciseness, we shall only prove estimate~\eqref{mom-estim-wick-sq} for $n=0$, that is we will only focus on the estimate for the time-variation $\<Psi2>_{m}(t,.)-\<Psi2>_{m}(s,.)$, with $m\geq 1$. The extension of the result to all $m\geq n\geq 0$ could in fact be easily deduced from the combination of the subsequent estimates with the elementary bounding argument used in~\eqref{element-bounding}. %We leave the details of this procedure as an exercise to the reader.

\

A first fundamental observation is that the contractivity argument used in~\eqref{appl-hyper} can be extended to the present setting. Indeed, the random variable under consideration clearly belongs to the first two chaoses generated by $W$ (with representation~\eqref{defi:approx-b-n} of the noise in mind), and therefore, due to the hypercontractivity property holding in such a space (see e.g.~\cite{wick}), we can assert that
\begin{align*}
&\mathbb{E}\bigg[\Big|\mathcal{F}^{-1}\Big(\{1+|.|^2\}^{-\alpha}\mathcal{F}\big(\chi^2 \big[\<Psi2>_{m}(t,.)-\<Psi2>_{m}(s,.)\big]\big)\Big)(x)\Big|^{2p}\bigg]\nonumber\\
&\leq c_p\, \mathbb{E}\bigg[\Big|\mathcal{F}^{-1}\Big(\{1+|.|^2\}^{-\alpha}\mathcal{F}\big(\chi^2 \big[\<Psi2>_{m}(t,.)-\<Psi2>_{m}(s,.)\big]\big)\Big)(x)\Big|^{2}\bigg]^p \, , 
\end{align*}
where the constant $c_p$ only depends on $p$.

\smallskip

Let us then write
\begin{align*}
&\mathcal{F}^{-1}\Big(\{1+|.|^2\}^{-\alpha}\mathcal{F}\big(\chi^2 \big[\<Psi2>_{m}(t,.)-\<Psi2>_{m}(s,.)\big]\big)\Big)(x)\\
&=\frac{1}{(2\pi)^d}\int_{\mathbb{R}^d}d\la \, \{1+|\lambda|^2\}^{-\alpha}e^{\imath \langle x, \lambda \rangle}\left(\int_{\mathbb{R}^d}d\be\, \widehat{\chi^2}(\la-\be) \cf\big(\big[\<Psi2>_{m}(t,.)-\<Psi2>_{m}(s,.)\big]\big)(\be)\right)\, ,
\end{align*}
which immediately yields
\begin{align}
&\mathbb{E}\bigg[\Big|\mathcal{F}^{-1}\Big(\{1+|.|^2\}^{-\al}\mathcal{F}\big(\chi^2 \big[\<Psi2>_{m}(t,.)-\<Psi2>_{m}(s,.)\big]\big)\Big)(x)\Big|^{2}\bigg]\nonumber\\
&=\frac{1}{(2\pi)^{2d}}\iint_{(\R^d)^2} \frac{d\la d\lati}{\{1+|\la|^2\}^\al \{1+|\lati|^2\}^\al}e^{\imath \langle x,\la-\lati\rangle} \iint_{(\R^d)^2} d\be d\betati \, \widehat{\chi^2}(\la-\be) \overline{\widehat{\chi^2}(\lati-\betati)} \mathcal{Q}^{(2)}_{m;s,t}(\be,\betati) \, ,\label{correcti}
\end{align}
with
$$\mathcal{Q}^{(2)}_{m;s,t}(\be,\betati):=\mathbb{E}\Big[\cf\big(\big[\<Psi2>_{m}(t,.)-\<Psi2>_{m}(s,.)\big]\big)(\be) \overline{\cf\big(\big[\<Psi2>_{m}(t,.)-\<Psi2>_{m}(s,.)\big]\big)(\betati)}\Big] \, .$$

\

In order to expand $\mathcal{Q}^{(2)}_{m;s,t}(\be,\betati)$, let us first recall that the expansion of the (\enquote{renormalized}) quantities
$$\mathbb{E}\Big[\<Psi2>_{m}(t,y)\overline{\<Psi2>_{m}(s,\tilde{y})}\Big]$$
is governed by the following standard application of  Wick formula (see e.g.~\cite{wick}) :

\begin{lemma}\label{wick}
For all $m, n \geq 1, s, t \geq 0$ and $y, \tilde{y} \in \mathbb{R}^d,$ it holds that
$$\mathbb{E} \Big[\<Psi2>_{m}(t,y)\overline{\<Psi2>_{n}(s,\tilde{y})}\Big]=\Big|\mathbb{E}\Big[\<Psi>_{m}(t,y)\overline{\<Psi>_{n}(s,\tilde{y})}\Big]\Big|^2+\Big|\mathbb{E}\Big[\<Psi>_{m}(t,y)\<Psi>_{n}(s,\tilde{y})\Big]\Big|^2 \, .$$
\end{lemma}

\

Using Lemma~\ref{wick} and the shortcut notation $\<Psi>_{m}(u,v;z):=\<Psi>_{m}(v,z)-\<Psi>_{m}(u,z)$, we easily verify that
\small
\begin{align*}
&\mathbb{E}\Big[[\<Psi2>_{m}(t,y)-\<Psi2>_{m}(s,y)]\overline{[\<Psi2>_{m}(t,\tilde{y})-\<Psi2>_{m}(s,\tilde{y})]}\Big]\\
&=\bigg[\mathbb{E}\Big[\<Psi>_{m}(s,t;y)\overline{\<Psi>_{m}(t,\tilde{y})}\Big]\cdot \mathbb{E}\Big[\overline{\<Psi>_{m}(t,y)}\<Psi>_{m}(t,\tilde{y})\Big]+\mathbb{E}\Big[\<Psi>_{m}(s,y)\overline{\<Psi>_{m}(t,\tilde{y})}\Big]\cdot \mathbb{E}\Big[\overline{\<Psi>_{m}(s,t;y)}\<Psi>_{m}(t,\tilde{y})\Big]\\
&\hspace{0.5cm}+ \mathbb{E}\Big[\<Psi>_{m}(t,s;y)\overline{\<Psi>_{m}(s,\tilde{y})}\Big]\cdot \mathbb{E}\Big[\overline{\<Psi>_{m}(t,y)}\<Psi>_{m}(s,\tilde{y})\Big]+\mathbb{E}\Big[\<Psi>_{m}(s,y)\overline{\<Psi>_{m}(s,\tilde{y})}\Big]\cdot \mathbb{E}\Big[\overline{\<Psi>_{m}(t,s;y)}\<Psi>_{m}(s,\tilde{y})\Big]\bigg]\\
&+ \bigg[\mathbb{E}\Big[\<Psi>_{m}(s,t;y)\<Psi>_{m}(t,\tilde{y})\Big]\cdot \mathbb{E}\Big[\overline{\<Psi>_{m}(t,y)}\overline{\<Psi>_{m}(t,\tilde{y})}\Big]+\mathbb{E}\Big[\<Psi>_{m}(s,y)\<Psi>_{m}(t,\tilde{y})\Big]\cdot \mathbb{E}\Big[\overline{\<Psi>_{m}(s,t;y)}\overline{\<Psi>_{m}(t,\tilde{y})}\Big]\\
&\hspace{0.5cm}+ \mathbb{E}\Big[\<Psi>_{m}(t,s;y)\<Psi>_{m}(s,\tilde{y})\Big]\cdot \mathbb{E}\Big[\overline{\<Psi>_{m}(t,y)}\overline{\<Psi>_{m}(s,\tilde{y})}\Big]+\mathbb{E}\Big[\<Psi>_{m}(s,y)\<Psi>_{m}(s,\tilde{y})\Big]\cdot \mathbb{E}\Big[\overline{\<Psi>_{m}(t,s;y)}\overline{\<Psi>_{m}(s,\tilde{y})}\Big]\bigg]\, ,
\end{align*}
\normalsize
which, combined with the covariance formulas~\eqref{cova-luxo}-\eqref{cova-luxo-2}, allows us to expand~\eqref{correcti} as 
\small
\begin{align*}
&\mathbb{E}\bigg[\Big|\mathcal{F}^{-1}\Big(\{1+|.|^2\}^{-\al}\mathcal{F}\big(\chi^2 \big[\<Psi2>_{m}(t,.)-\<Psi2>_{m}(s,.)\big]\big)\Big)(x)\Big|^{2}\bigg]\\
&=c\int_{(\xi,\eta)\in D_{m}}d\xi d\eta \int_{(\xiti,\etati)\in D_{m}}d\tilde{\xi} d\tilde{\eta} \, \frac{1}{|\xi|^{2H_0-1}}\prod_{i=1}^{d}\frac{1}{|\eta_i|^{2H_i-1}}\frac{1}{|\tilde{\xi}|^{2H_0-1}}\prod_{i=1}^{d}\frac{1}{|\tilde{\eta_i}|^{2H_i-1}} \\
&\hspace{3cm} \bigg(\sum_{\ell=1}^8\gga^\ell_{m;s,t}(\xi,|\eta|;\xiti,|\etati|)\bigg) \\
&\hspace{1cm}\iint_{(\R^d)^2} \frac{d\la d\lati}{\{1+|\la|^2\}^\al \{1+|\lati|^2\}^\al}e^{\imath \langle x,\la-\lati\rangle} \widehat{\chi^2}(\la-(\eta-\etati)) \overline{\widehat{\chi^2}(\lati-(\eta-\etati))} \, ,
\end{align*}
\normalsize
with $\gga^\ell_{m;s,t}=\gga^\ell_{m;s,t}(\xi,|\eta|;\xiti,|\etati|)$ given by
\small
\begin{align*}
\gga^1_{m;s,t}:=\gamma_{s,t}(\xi,|\eta|)\, \overline{\gamma_{t}(\xi,|\eta|)}\, |\gamma_t(\xiti,|\etati|)|^2\, ,& \quad \gga^2_{m;s,t}:=\gamma_s(\xi,|\eta|)\, \overline{\gamma_{t}(\xi,|\eta|)}\, \overline{\gamma_{s,t}(\xiti,|\etati|)}\, \gamma_t(\xiti,|\etati|)\, ,\\
\gga^3_{m;s,t}:=\gamma_{t,s}(\xi,|\eta|)\, \overline{\gamma_{s}(\xi,|\eta|)}\, \overline{\gamma_t(\xiti,|\etati|)} \, \gamma_s(\xiti,|\etati|)\, , &\quad \gga^4_{m;s,t}:=|\gamma_s(\xi,|\eta|)|^2\, \, \overline{\gamma_{t,s}(\xiti,|\etati|)}\, \gamma_s(\xiti,|\etati|)  \, ,\\
\gga^5_{m;s,t}:=\gamma_{s,t}(\xi,|\eta|)\, \gamma_{t}(-\xi,|\eta|)\, \overline{\gamma_t(\xiti,|\etati|)}\overline{\ga_t(-\xiti,|\etati|)}\, ,& \quad \gga^6_{m;s,t}:=\gamma_{s}(\xi,|\eta|)\, \gamma_{t}(-\xi,|\eta|)\, \overline{\gamma_{s,t}(\xiti,|\etati|)}\overline{\ga_t(-\xiti,|\etati|)}\, ,\\
\gga^7_{m;s,t}:=\gamma_{t,s}(\xi,|\eta|)\, \gamma_{s}(-\xi,|\eta|)\, \overline{\gamma_t(\xiti,|\etati|)}\overline{\ga_s(-\xiti,|\etati|)}\, , &\quad \gga^8_{m;s,t}:=\gamma_{s}(\xi,|\eta|)\, \gamma_{s}(-\xi,|\eta|)\, \overline{\gamma_{t,s}(\xiti,|\etati|)}\overline{\ga_s(-\xiti,|\etati|)}  \, .
\end{align*}
\normalsize
At this point, we can rely on the technical Lemma~\ref{lem:handle-chi-correc} to assert that
\begin{equation}\label{bou-sum-cj-i-p}
\int_{\R^d} dx \, \mathbb{E}\bigg[\Big|\mathcal{F}^{-1}\Big(\{1+|.|^2\}^{-\al}\mathcal{F}\big(\chi^2 \big[\<Psi2>_{m}(t,.)-\<Psi2>_{m}(s,.)\big]\big)\Big)(x)\Big|^{2}\bigg]^p\lesssim \bigg( \sum_{\ell=1}^8 \cj^\ell_{m;s,t}\bigg)^p
\end{equation}
where
\begin{align}
&\cj^\ell_{m;s,t}:=\int_{(\xi,\eta)\in D_{m}}d\xi d\eta \int_{(\xiti,\etati)\in D_{m}}d\tilde{\xi} d\tilde{\eta} \, \frac{1}{|\xi|^{2H_0-1}}\prod_{i=1}^{d}\frac{1}{|\eta_i|^{2H_i-1}}\frac{1}{|\tilde{\xi}|^{2H_0-1}}\prod_{i=1}^{d}\frac{1}{|\tilde{\eta_i}|^{2H_i-1}}\nonumber\\
&\hspace{7cm} \big|\gga^\ell_{m;s,t}(\xi,|\eta|;\xiti,|\etati|)\big|\big\{1+|\eta-\tilde{\eta}|^2\big\}^{-2\alpha} \, .\label{defi-cj-i}
\end{align}
Let us focus on the estimate for $\cj^1_{m;s,t}$. Using the trivial bound $|\eta-\etati|\geq | |\eta|-|\etati| |$, one has
\begin{align}
&\cj^1_{m;s,t}= \int_{(\xi,\eta)\in D_{m}}d\xi d\eta \int_{(\xiti,\etati)\in D_{m}}d\tilde{\xi} d\tilde{\eta} \, \big\{1+|\eta-\tilde{\eta}|^2\big\}^{-2\alpha}  \nonumber\\
&\hspace{0.5cm}\bigg(\frac{1}{|\xi|^{2H_0-1}}\prod_{i=1}^{d}\frac{1}{|\eta_i|^{2H_i-1}}|\gamma_t(\xi,|\eta|)||\gamma_{s,t}(\xi,|\eta|)|\bigg)\cdot \bigg(\frac{1}{|\tilde{\xi}|^{2H_0-1}}\prod_{i=1}^{d}\frac{1}{|\tilde{\eta_i}|^{2H_i-1}}|\gamma_t(\tilde{\xi},|\tilde{\eta}|)|^2\bigg)  \nonumber\\
&\lesssim  \int_{(\mathbb{R}\times\mathbb{R}^d)^2}d\xi d\eta d\tilde{\xi} d\tilde{\eta} \, \big\{1+||\eta|-|\tilde{\eta}||^2\big\}^{-2\alpha}  \nonumber\\
&\hspace{0.5cm}\bigg(\frac{1}{|\xi|^{2H_0-1}}\prod_{i=1}^{d}\frac{1}{|\eta_i|^{2H_i-1}}|\gamma_t(\xi,|\eta|)||\gamma_{s,t}(\xi,|\eta|)|\bigg)\cdot\bigg(\frac{1}{|\tilde{\xi}|^{2H_0-1}}\prod_{i=1}^{d}\frac{1}{|\tilde{\eta_i}|^{2H_i-1}}|\gamma_t(\tilde{\xi},|\tilde{\eta}|)|^2\bigg) \, . \label{boun-cj-1-m}
\end{align}

\

Now, let us decompose the latter integration domain into $(\R\times \mathbb{R}^d )^2:=D_1\cup D_2$, where
$$D_1:=\Big\{(\xi,\eta,\tilde{\xi},\tilde{\eta}):\,  0\leq |\tilde{\eta}|\leq \frac{|\eta|}{2} \ \text{or} \ |\tilde{\eta}|\geq \frac{3|\eta|}{2}\Big\}\ \text{and} \ D_2:=\Big\{(\xi,\eta,\tilde{\xi},\tilde{\eta}):\,  \frac{|\eta|}{2}<|\tilde{\eta}|<\frac{3|\eta|}{2}\Big\}\, .$$
For the integral over $D_1$, we can rely on the inequality $||\eta|-|\tilde{\eta}||\geq\max(\frac{|\eta|}{2},\frac{|\tilde{\eta}|}{3})$ (valid for all $(\xi,\eta,\tilde{\xi},\tilde{\eta}) \in D_1$) to write
\begin{align*}
\ca_1&:=\int_{D_1}\frac{d\xi d\eta d\tilde{\xi} d\tilde{\eta}}{\{1+||\eta|-|\tilde{\eta}||^2\}^{2\alpha}}\bigg(\frac{1}{|\xi|^{2H_0-1}}\prod_{i=1}^{d}\frac{1}{|\eta_i|^{2H_i-1}}|\gamma_t(\xi,|\eta|)||\gamma_{s,t}(\xi,|\eta|)|\bigg)\cdot\\
&\hspace{7cm}\bigg(\frac{1}{|\tilde{\xi}|^{2H_0-1}}\prod_{i=1}^{d}\frac{1}{|\tilde{\eta_i}|^{2H_i-1}}|\gamma_t(\tilde{\xi},|\tilde{\eta}|)|^2\bigg) \nonumber \\
& \lesssim \Bigg(\int_{\mathbb{R}\times\mathbb{R}^d}\frac{d\xi d\eta }{\{1+|\eta|^2\}^{\alpha}}\frac{1}{|\xi|^{2H_0-1}}\prod_{i=1}^{d}\frac{1}{|\eta_i|^{2H_i-1}}|\gamma_t(\xi,|\eta|)||\gamma_{s,t}(\xi,|\eta|)|  \Bigg)\cdot\nonumber \\
&\hspace{5cm} \Bigg(\int_{\mathbb{R}\times\mathbb{R}^d}\frac{d\tilde{\xi} d\tilde{\eta}}{\{1+|\tilde{\eta}|^2\}^{\alpha}}\frac{1}{|\tilde{\xi}|^{2H_0-1}}\prod_{i=1}^{d}\frac{1}{|\tilde{\eta_i}|^{2H_i-1}}|\gamma_t(\tilde{\xi},|\tilde{\eta}|)|
^2\Bigg)\nonumber \\
&\lesssim \Bigg(\int_{\mathbb{R}\times\mathbb{R}^d}\frac{d\xi d\eta }{\{1+|\eta|^2\}^{\alpha}}\frac{1}{|\xi|^{2H_0-1}}\prod_{i=1}^{d}\frac{1}{|\eta_i|^{2H_i-1}}|\gamma_t(\xi,|\eta|)|^2  \Bigg)^{\frac{1}{2}}\cdot\nonumber \\ 
&\hspace{2cm} \Bigg(\int_{\mathbb{R}\times\mathbb{R}^d}\frac{d\xi d\eta }{\{1+|\eta|^2\}^{\alpha}}\frac{1}{|\xi|^{2H_0-1}}\prod_{i=1}^{d}\frac{1}{|\eta_i|^{2H_i-1}}|\gamma_{s,t}(\xi,|\eta|)|^2  \Bigg)^{\frac{1}{2}}\cdot\nonumber \\ 
&\hspace{4cm}\Bigg(\int_{\mathbb{R}\times\mathbb{R}^d}\frac{d\tilde{\xi} d\tilde{\eta}}{\{1+|\tilde{\eta}|^2\}^{\alpha}}\frac{1}{|\tilde{\xi}|^{2H_0-1}}\prod_{i=1}^{d}\frac{1}{|\tilde{\eta_i}|^{2H_i-1}}|\gamma_t(\tilde{\xi},|\tilde{\eta}|)|^2\Bigg)\, ,
\end{align*}
where we have merely used Cauchy-Schwarz inequality to derive the last estimate. 

\smallskip

Observe that we are here dealing with the exact same integral pattern as in the proof of Proposition~\ref{sto} (see in particular~\eqref{element-bounding}) and therefore we can mimic the arguments in~\eqref{chg-of-var}-\eqref{boun-i-pr} to deduce the desired estimate, namely: for any $\ka >0$ small enough,
$$\ca_1\lesssim |t-s|^{\ka} \, .$$

\smallskip

Let us turn to the integral over $D_2$ and lean on a hyperspherical change of variable to write
\begin{align*}
&\int_{\frac{|\eta|}{2}<|\tilde{\eta}|<\frac{3|\eta|}{2}}\frac{d\tilde{\eta}}{\{1+||\eta|-|\tilde{\eta}||^2\}^{2\alpha}}|\gamma_t(\tilde{\xi},|\tilde{\eta}|)|^2 \prod\limits_{i=1}^{d}\frac{1}{|\tilde{\eta_i}|^{2H_i-1}} \nonumber \\
&=|\eta|^{-2(H_1+\cdots+H_d)+2d}\int_{\frac{1}{2}<|\tilde{\eta}|<\frac{3}{2}}\frac{d\tilde{\eta}}{\{1+|\eta|^2(1-|\tilde{\eta}|)^2\}^{2\alpha}}|\gamma_t(\tilde{\xi},|\eta||\tilde{\eta}|)|^2 \prod\limits_{i=1}^{d}\frac{1}{|\tilde{\eta_i}|^{2H_i-1}} \nonumber \\
&\lesssim |\eta|^{-2(H_1+\cdots+H_d)+2d}\int_{\frac{1}{2}<r<\frac{3}{2}}\frac{dr}{\{1+|\eta|^2(1-r)^2\}^{2\alpha}}|\gamma_t(\tilde{\xi},|\eta|r)|^2\, .
\end{align*}
As a consequence,
\begin{align*}
\ca_2&:=\int_{D_2}\frac{d\xi d\eta d\tilde{\xi} d\tilde{\eta}}{\{1+||\eta|-|\tilde{\eta}||^2\}^{2\alpha}}\bigg(\frac{1}{|\xi|^{2H_0-1}}\prod_{i=1}^{d}\frac{1}{|\eta_i|^{2H_i-1}}|\gamma_t(\xi,|\eta|)||\gamma_{s,t}(\xi,|\eta|)|\bigg)\cdot\\
&\hspace{7cm}\bigg(\frac{1}{|\tilde{\xi}|^{2H_0-1}}\prod_{i=1}^{d}\frac{1}{|\tilde{\eta_i}|^{2H_i-1}}|\gamma_t(\tilde{\xi},|\tilde{\eta}|)|^2\bigg) \nonumber \\
&\lesssim \int_{\mathbb{R}^d}\frac{d\eta}{|\eta|^{2(H_1+\cdots+H_d)-2d}}\prod_{i=1}^{d}\frac{1}{|\eta_i|^{2H_i-1}}\int_{\frac{1}{2}<r<\frac{3}{2}}\frac{dr}{\{1+|\eta|^2(1-r)^2\}^{2\alpha}}\cdot \nonumber \\
&\hspace{5cm}\bigg(\int_{\mathbb{R}}d\xi\, \frac{|\gamma_t(\xi,|\eta|)||\gamma_{s,t}(\xi,|\eta|)|}{|\xi|^{2H_0-1}}\bigg)\cdot \bigg(\int_{\mathbb{R}}d\tilde{\xi}\, \frac{|\gamma_t(\tilde{\xi},|\eta|r)|^2}{|\tilde{\xi}|^{2H_0-1}}\bigg) \nonumber \\
&\lesssim \int_{\mathbb{R}^d}\frac{d\eta}{|\eta|^{2(H_1+\cdots+H_d)-2d}}\prod\limits_{i=1}^{d}\frac{1}{|\eta_i|^{2H_i-1}}\int_{\frac{1}{2}<r<\frac{3}{2}}\frac{dr}{\{1+|\eta|^2(1-r)^2\}^{2\alpha}} \cdot\nonumber \\
&\hspace{2cm}\bigg(\int_{\mathbb{R}}d\xi\,  \frac{|\gamma_t(\xi,|\eta|)|^2}{|\xi|^{2H_0-1}}\bigg)^{\frac{1}{2}}\cdot \bigg(\int_{\mathbb{R}}d\xi\,  \frac{|\gamma_{s,t}(\xi,|\eta|)|^2}{|\xi|^{2H_0-1}}\bigg)^{\frac{1}{2}}\cdot \bigg(\int_{\mathbb{R}}d\tilde{\xi}\, \frac{|\gamma_t(\tilde{\xi},|\eta|r)|^2}{|\tilde{\xi}|^{2H_0-1}}\bigg)\, . 
\end{align*}
Using again a hyperspherical change of variable (with respect to $\eta$), we obtain
\begin{align*}
\ca_2&\lesssim \int_{0}^{\infty}\frac{d\rho}{\rho^{4(H_1+\cdots+H_d)-4d+1}}\int_{\frac{1}{2}<r<\frac{3}{2}}\frac{dr}{\{1+\rho^2(1-r)^2\}^{2\alpha}} \nonumber \\
&\hspace{2cm}\bigg(\int_{\mathbb{R}}d\xi\, \frac{|\gamma_t(\xi,\rho)|^2}{|\xi|^{2H_0-1}}\bigg)^{\frac{1}{2}}\cdot \bigg(\int_{\mathbb{R}}d\xi\,  \frac{|\gamma_{s,t}(\xi,\rho)|^2}{|\xi|^{2H_0-1}}\bigg)^{\frac{1}{2}}\cdot \bigg(\int_{\mathbb{R}}d\tilde{\xi}\, \frac{|\gamma_t(\tilde{\xi},\rho r)|^2}{|\tilde{\xi}|^{2H_0-1}}\bigg)\, , \nonumber 
\end{align*}
and we can now use Corollary~\ref{tec} to assert that 
\begin{align*}
\ca_2&\lesssim |t-s|^{\kappa}\Bigg[\int_{0}^{1}\frac{d\rho}{\rho^{4(H_1+\cdots+H_d)-4d+1}}\\
&\hspace{2cm}+\int_{1}^{\infty}\frac{d\rho}{\rho^{4(2H_0+H_1+\cdots+H_d)-4d-3-8\varepsilon-8\kappa}}\int_{\frac{1}{2}<r<\frac{3}{2}}\frac{dr}{\{1+\rho^2(1-r)^2\}^{2\alpha}}\Bigg]
\end{align*}
for all $0<\varepsilon<\frac{1}{2}$ and $0<\kappa<\min(H_0,\frac{1}{2}-\varepsilon)$. 

\smallskip

At this point, observe that
\begin{align}
&\int_{1}^{\infty}\frac{d\rho}{\rho^{4(2H_0+H_1+\cdots+H_d)-4d-3-8\varepsilon-8\kappa}}\int_{\frac{1}{2}<r<\frac{3}{2}}\frac{dr}{\{1+\rho^2(1-r)^2\}^{2\alpha}} \nonumber \\
&\leq \int_{1}^{\infty}\frac{d\rho}{\rho^{4\alpha+4(2H_0+H_1+\cdots+H_d)-4d-3-8\varepsilon-8\kappa}}\int_{\frac{1}{2}<r<\frac{3}{2}}\frac{dr}{(1-r)^{4\alpha}} \, .\label{cond-al-imp}
\end{align}
Thanks to our assumption~\eqref{assump-al-bis} on $\al$, we know on the one hand that $4\al <1$, and on the other hand we can pick $\varepsilon, \ka >0$ such that 
$$4\alpha+4(2H_0+H_1+\cdots+H_d)-4d-3-8\varepsilon-8\kappa>1\, . $$
For such a choice of $\varepsilon,\ka$, the integrals in the right-hand side of~\eqref{cond-al-imp} are clearly finite, and therefore we have shown that
$$\ca_2\lesssim |t-s|^{\kappa} \, .$$

\smallskip

Going back to~\eqref{boun-cj-1-m}, we have thus proved that, uniformly over $m$,
$$\cj^1_{m;s,t} \lesssim |t-s|^{\kappa} \, .$$ 
It is now easy to realize that the other seven integrals $\{\cj^\ell_{m;s,t},\, \ell=2,\ldots,8\}$ (as defined in~\eqref{defi-cj-i}) could be handled with the very same arguments (yielding the very same final bound). 

\smallskip

Injecting the above estimates into~\eqref{bou-sum-cj-i-p} provides us with the desired bound~\eqref{mom-estim-wick-sq}.

\

\noindent
\textbf{Step 2: Conclusion.} Let $\al$ satisfying the condition in the statement of Proposition~\ref{sto1}, that is $\alpha > d+1-\big(2H_0+\sum_{i=1}^{d}H_i\big)$.  

\smallskip

If in addition one has $\al < \frac14$, then condition~\eqref{assump-al-bis} is satisfied, and so the moment estimate~\eqref{mom-estim-wick-sq} holds true. Starting from this estimate, we can use the same arguments as in Step 2 of Section~\ref{subsec:proof-sto-1} to obtain that $(\chi^2\<Psi2>_{n})_{n\geq 1}$ converges almost surely in the space $\mathcal{C}([0,T]; \mathcal{W}^{-2\alpha,p}(\mathbb{R}^d))$, for every $2\leq p <\infty$.

\smallskip

If $\al \geq \frac14$, observe that due to assumption~\eqref{cond-hurst-psi-2}, we can pick $\al'$ satisfying $\al'<\al$ and $d+1-\big(2H_0+\sum_{i=1}^{d}H_i\big) <\al'<\frac14$ (that is, $\al'$ satisfies~\eqref{assump-al-bis}). By repeating the above arguments, we deduce that the sequence $(\chi^2\<Psi2>_{n})_{n\geq 1}$ converges almost surely in $\mathcal{C}([0,T]; \mathcal{W}^{-2\alpha',p}(\mathbb{R}^d))$, and therefore it converges almost surely in $\mathcal{C}([0,T]; \mathcal{W}^{-2\alpha,p}(\mathbb{R}^d))$ as well, for $2\leq p<\infty$.

\smallskip

Finally, the (a.s.) convergence in $\mathcal{C}([0,T]; \mathcal{W}^{-2\alpha,\infty}(\mathbb{R}^d))$ can be easily derived from the Sobolev embedding $\cw^{-2\al+\frac{d}{p}+\eta,p}(\R^d)\subset \cw^{-2\al,\infty}(\R^d)$, for any $\eta>0$, which completes the proof of Proposition~\ref{sto1}.

\

\subsection{Proof of Proposition~\ref{prop:renorm-cstt}}\label{subsec:renorm-cstt}
Fix $d\geq 1$ and $(H_0,\dots,H_d) \in (0,1)^{d+1}$ such that 
$$d+\frac{3}{4}<2H_0+\sum_{i=1}^{d}H_i\leq d+1\, .$$ 
Using~\eqref{cova-luxo}, the quantity under consideration can be written as
\begin{eqnarray*}
\sigma_n(t,x)\ =\ \mathbb{E}\big[|\<Psi>_n(t,x)|^2\big]&=&c_H^2\int_{|\xi|\leq 4^n} \frac{d\xi}{|\xi|^{2H_0-1}}\int_{|\eta|\leq 2^n}\prod\limits_{i=1}^{d}\frac{d\eta_i}{|\eta_i|^{2H_i-1}}|\gamma_t(\xi,|\eta|)|^2 \\
&=&C_H\int_{0}^{2^n} \frac{dr}{r^{2(H_1+\cdots+H_d)-2d+1}}\int_{|\xi|\leq 4^n}\frac{d\xi}{|\xi|^{2H_0-1}}|\gamma_t(\xi,r)|^2 \, ,
\end{eqnarray*}
for some constant $C_H$, and where the last identity is derived from a hyperspherical change of variables. The above formula shows in particular that $\sigma_n$ does not depend on $x$, as stated in Proposition \ref{prop:renorm-cstt}. Regarding the desired estimate~\eqref{estim-sigma-n}, it is now a consequence of the following technical result (applied with $\alpha:=2H_0 \in (0,2)$ and $\kappa:=d+1-[2H_0+\sum_{i=1}^{d}H_i] \geq 0$):

\begin{proposition}
For all $\alpha \in (0,2)$ and $\kappa \geq 0$ verifying $\alpha+\kappa>1$, one has
$$\int_{0}^{2^n}\frac{dr}{r^{-2\alpha-2\kappa+3}}\int_{|\xi|\leq 4^n}\frac{d\xi}{|\xi|^{\alpha-1}}|\gamma_t(\xi,r)|^2\underset{n \rightarrow \infty}\sim  \left\{
    \begin{array}{ll}
        \frac{\pi}{\kappa}4^{n \kappa}t & \mbox{if}\hspace{0,3cm} \kappa>0 \,   \\
        \pi \ln{(4)}\cdot n t & \mbox{if}\hspace{0,3cm} \kappa=0\, 
    \end{array}
\right. .$$
\end{proposition}

\begin{proof}
It is easy to check that
$$|\gamma_t(\xi,r)|^2=2\frac{1-\cos(t(\xi-r^2))}{|r^2-\xi|^2}\, ,$$ 
which allows us to write the integral under consideration as
\begin{align}
\int_{0}^{2^n}\frac{dr}{r^{-2\alpha-2\kappa+3}}\int_{|\xi|\leq 4^n}\frac{d\xi}{|\xi|^{\alpha-1}}|\gamma_t(\xi,r)|^2 &=
\int_{0}^{4^n}\frac{dr}{r^{-\alpha-\kappa+2}}\int_{|\xi|\leq 4^n}\frac{d\xi}{|\xi|^{\alpha-1}}\frac{1-\cos(t(\xi-r))}{(\xi-r)^2}\nonumber \\
&= 4^{n\kappa}\int_{0}^{1}\frac{dr}{r^{-\alpha-\kappa+2}}\int_{|\xi|\leq 1}\frac{d\xi}{|\xi|^{\alpha-1}}\frac{1-\cos(4^{n}t(\xi-r))}{4^{n}(\xi-r)^2}\nonumber \\
%&= 2\cdot4^{n\kappa}\int_{0}^{1}\frac{dr}{r^{-\alpha-\kappa+2}}\int_{|\xi|\leq 1}\frac{d\xi}{|\xi|^{\alpha-1}}\frac{\sin^{2}(\frac{4^{n}t}{2}(\xi-r))}{4^{n}(\xi-r)^2} \\
&= t\cdot4^{n\kappa}\int_{0}^{1}\frac{dr}{r^{-\alpha-\kappa+2}}\int_{|\xi|\leq 1}\frac{d\xi}{|\xi|^{\alpha-1}}\frac{\sin^{2}(\frac{4^{n}t}{2}(\xi-r))}{\frac{4^{n}t}{2}(\xi-r)^2}\, .\label{starti-id}
\end{align}

\

\

\noindent
\textit{First case: $\ka>0$}.

\smallskip

Let us set $\Phi(x):=\frac{\sin^{2}(x)}{\pi x^2}$ for every $x\in \R$. This function is positive, even and satisfies $\int_{\mathbb{R}}\Phi(x)dx=1$, so that the sequence
$$\Phi_n(x):=\frac{4^{n}t}{2}\Phi\Big(\frac{4^{n}t}{2}x\Big)$$
defines an approximation of the identity. Let us now rewrite the quantity in~\eqref{starti-id} as
\begin{align*}
&t4^{n\kappa}\int_{0}^{1}\frac{dr}{r^{-\alpha-\kappa+2}}\int_{|\xi|\leq 1}\frac{d\xi}{|\xi|^{\alpha-1}}\frac{\sin^{2}(\frac{4^{n}t}{2}(\xi-r))}{\frac{4^{n}t}{2}(\xi-r)^2}\\
&= \pi t 4^{n\kappa}\int_{0}^{1}\frac{dr}{r^{-\alpha-\kappa+2}}\int_{|\xi|\leq 1}\frac{d\xi}{|\xi|^{\alpha-1}}\Phi_n(\xi-r)= \pi t 4^{n\kappa}\int_{\mathbb{R}}h(r)(\Phi_n*g)(r)dr
\end{align*}
where we have set
$$h(r):=\mathbbm{1}_{[0,1]}(r)\frac{1}{r^{-\alpha-\kappa+2}} \quad \text{and} \quad g(\xi):=\mathbbm{1}_{[-1,1]}(\xi)\frac{1}{|\xi|^{\al-1}}\, .$$
Due to our assumptions on $\ka,\al$ (i.e., $\ka>0$, $\al\in (0,2)$ and $\alpha+\kappa>1$), we can exhibit a pair of conjugate exponents $p, p'>1$ such that 
$$g \in L^p(\mathbb{R}) \quad \text{and}\quad h \in L^{p'}(\mathbb{R}) \, .$$
Indeed, it is easy to check that any $p$ such that
$$\max(0,\al-1) < \frac{1}{p} < \min(1,\al+\ka-1)$$
meets the above requirements. Let us fix such a pair $(p,p')$, and now write
\begin{align*}
\Big|\int_{\mathbb{R}}h(r)(\Phi_n*g)(r)dr-\int_{\mathbb{R}}h(r)g(r)dr\Big|&=\Big|\int_{\mathbb{R}}h(r)[(\Phi_n*g)(r)-g(r)]dr\Big| \\
&\leq\Big(\int_{\mathbb{R}}|h(r)|^{p'}dr\Big)^{\frac{1}{p'}}\Big(\int_{\mathbb{R}}|(\Phi_n*g)(r)-g(r)|^{p}dr\Big)^{\frac{1}{p}}\, 
\end{align*}
which tends to $0$ as $n$ tends to infinity. Since $\int_{\mathbb{R}}h(r)g(r)dr=\frac{1}{\kappa}$, we obtain the desired conclusion
$$\int_{0}^{2^n}\frac{dr}{r^{-2\alpha-2\kappa+3}}\int_{|\xi|\leq 4^n}\frac{d\xi}{|\xi|^{\alpha-1}}|\gamma_t(\xi,r)|^2 \underset{n \rightarrow \infty}\sim \frac{\pi t}{\kappa} 4^{n\kappa} \, .$$

\

\

\noindent
\textit{Second case: $\ka=0$.}

\smallskip

Note first that, due to our assumptions, one has here $1<\al <2$. Besides, let us recall (see~\eqref{starti-id}) that
\begin{equation}\label{estim-ka-zero}
\int_{0}^{2^n}\frac{dr}{r^{-2\alpha+3}}\int_{|\xi|\leq 4^n}\frac{d\xi}{|\xi|^{\alpha-1}}|\gamma_t(\xi,r)|^2=t\int_{0}^{1}\frac{dr}{r^{-\alpha+2}}\int_{|\xi|\leq 1}\frac{d\xi}{|\xi|^{\alpha-1}}\frac{\sin^{2}(\frac{4^{n}t}{2}(\xi-r))}{\frac{4^{n}t}{2}(\xi-r)^2}\, .
\end{equation}
Consider the function $f$ defined for all $T>0$ by
$$f(T):=\int_{0}^{1}\frac{dr}{r^{-\alpha+2}}\int_{|\xi|\leq 1}\frac{d\xi}{|\xi|^{\alpha-1}}\frac{\sin^{2}(T(\xi-r))}{(\xi-r)^2}.$$
One has of course
\begin{eqnarray*}
f'(T)&=&\int_{0}^{1}\frac{dr}{r^{-\alpha+2}}\int_{|\xi|\leq 1}\frac{d\xi}{|\xi|^{\alpha-1}}\frac{2\sin(T(\xi-r))\cos(T(\xi-r))}{(\xi-r)}\\
&=&\int_{0}^{1}\frac{dr}{r^{-\alpha+2}}\int_{|\xi|\leq 1}\frac{d\xi}{|\xi|^{\alpha-1}}\frac{\sin(2T(\xi-r))}{(\xi-r)}
\end{eqnarray*}
and
\begin{eqnarray*}
f''(T)&=&2\int_{0}^{1}\frac{dr}{r^{-\alpha+2}}\int_{|\xi|\leq 1}\frac{d\xi}{|\xi|^{\alpha-1}}\cos(2T(\xi-r))\\
&=&2\int_{0}^{1}\frac{dr}{r^{-\alpha+2}}\int_{0}^1\frac{d\xi}{|\xi|^{\alpha-1}}\{\cos(2T(\xi+r))+\cos(2T(\xi-r))\}\\
&=&4\bigg(\int_{0}^{1}dr\, \frac{\cos(2Tr)}{r^{-\alpha+2}}\bigg) \bigg(\int_{0}^1 d\xi\,\frac{\cos(2T\xi)}{|\xi|^{\alpha-1}}\bigg) \, = \, \frac{2}{T}\bigg(\int_{0}^{2T}dr\, \frac{\cos(r)}{r^{-\alpha+2}}\bigg) \bigg(\int_{0}^{2T} d\xi\,\frac{\cos(\xi)}{|\xi|^{\alpha-1}}\bigg) \, .
\end{eqnarray*}
At this point, let us recall the following standard Fourier transform: for $\beta \in (0,1)$ and $\xi \neq 0$,
$$\int_{\R}dx \, \frac{e^{-ix\xi}}{|x|^\beta}=\frac{2\sin\Big( \frac{\pi \be}{2}\Big)\Gamma(1-\beta)}{|\xi|^{1-\beta}} \, ,$$
which immediately yields
$$\int_{0}^\infty dx \, \frac{\cos(x)}{|x|^\beta}=\sin\Big( \frac{\pi \be}{2}\Big)\Gamma(1-\beta) \, ,$$
and therefore
$$f''(T) \sim \frac{2}{T} \sin\Big( \frac{\pi }{2}(2-\al)\Big)\sin\Big( \frac{\pi }{2}(\al-1)\Big)\Gamma(\al-1)\Gamma(2-\al) \, .$$
Using Euler's reflection formula 
$$\Gamma(\beta)\Gamma(1-\beta)=\frac{\pi}{\sin(\pi \beta)}\quad \quad (\be\in (0,1)) \, ,$$
the above quantity reduces to
\begin{align*}
&\sin\Big( \frac{\pi }{2}(2-\al)\Big)\sin\Big( \frac{\pi }{2}(\al-1)\Big)\Gamma(\al-1)\Gamma(2-\al) \\
&=\pi \frac{\sin\Big( \frac{\pi }{2}(2-\al)\Big)\sin\Big( \frac{\pi }{2}(\al-1)\Big)}{\sin(\pi(\al-1))}=\frac{\pi}{2}\, \frac{\sin\Big( \frac{\pi }{2}(2-\al)\Big)\sin\Big( \frac{\pi }{2}(\al-1)\Big)}{\sin\Big( \frac{\pi }{2}(\al-1)\Big)\cos\Big( \frac{\pi }{2}(\al-1)\Big)}=\frac{\pi}{2} \, ,
\end{align*}
and so 
$$f''(T)\underset{T \rightarrow \infty}\sim \frac{\pi}{T}\, .$$
We can finally use a standard comparison argument (see Lemma~\ref{int} below) to successively assert that
$f'(T) \underset{T \rightarrow \infty}\sim \pi \ln(T)$ and
$$f(T)\underset{T \rightarrow \infty}\sim \pi (T\ln(T)-T)\underset{T \rightarrow \infty}\sim \pi T\ln(T)\, .$$
Going back to~\eqref{estim-ka-zero}, we get the desired conclusion, namely
$$\int_{0}^{2^n}\frac{dr}{r^{-2\alpha+3}}\int_{|\xi|\leq 4^n}\frac{d\xi}{|\xi|^{\alpha-1}}|\gamma_t(\xi,r)|^2 \underset{n \rightarrow \infty}\sim \pi t n \ln(4)\, .$$
\end{proof}

\begin{lemma}\label{int}
Fix $a \in \mathbb{R}$ and let $g:[a,+\infty[\rightarrow \mathbb{R}$, $h:[a,+\infty[\rightarrow (0,\infty)$, be two continuous functions. If $g(t)\underset{t \rightarrow \infty}\sim h(t)$ and $\int_a^{+\infty}h(t)dt=\infty$, then
$$\int_a^{T}g(t)dt\underset{T \rightarrow \infty}\sim \int_a^{T}h(t)dt \, .$$
\end{lemma}

\

\subsection{Global definition of the linear solution}\label{subsec:glob-def}

\

\smallskip

At first reading, the result of Proposition~\ref{sto} only provides us with a \emph{local} definition of the process $\<Psi>$. In other words, what is actually given by the statement is the set of the limit elements $\{\chi\<Psi>, \, \chi\in \cac_c^\infty(\R^d)\}$. For the sake of rigor, let us say a few words about how those elements can be glued together into a single process $\<Psi>$, which we can then inject into the transformation $v:=u-\<Psi>$ of Definition~\ref{defi:sol-regu} or Definition~\ref{defi:sol}.

\smallskip

To this end, fix $p\geq 2$ and $\al$ satisfying~\eqref{assump-gene-al}. Let us denote by $\cp$ the set of sequences $\si=(\si_k)_{k\geq 1}$ such that for each $k\geq 1$, $\si_k:\R^d \to \R$ is a smooth function satisfying
\begin{equation*}
    \si_k(x)=  \left\{
    \begin{array}{ll}
		1 & \mbox{if}\  \|x\|\leq k \, ,\\
        0 & \mbox{if} \ \|x\|\geq k+1\, .  
    \end{array}
\right.
\end{equation*}

Given such a sequence $\si$, and for each fixed $k\geq 1$, let us denote by $\<Psi>^{(\si_k)}$ the limit of the sequence $(\si_k\<Psi>_n)_{n\geq 1}$ in the space $\cac([0,T];\cw^{-\al,p}(\R^d))$, as provided by Proposition~\ref{sto}. We know in particular that $\<Psi>^{(\si_k)}$ is defined on a probability space $\Omega^{(\si_k)}$ of full measure $1$. Let us set $\Omega^{(\si)}:=\cap_{k\geq 1}\Omega^{(\si_k)}$, and note that this space is still of measure $1$.

\smallskip

For every fixed time $t \in [0,T]$, we now define the random distribution 
$$\<Psi>^{(\si)}(t):\Omega^{(\si)}\to \cd'(\R^d)$$
as follows: for every test function $\vp:\R^d \to \R$ such that $\textup{supp}(\vp) \subset B(0,k)$ (for some $k\geq 1$),
\begin{equation*}
\langle \<Psi>^{(\si)}(t), \vp \rangle := \langle \<Psi>^{(\si_k)}(t), \vp \rangle \, .
\end{equation*}

\begin{proposition}\label{prop:defi-psi-glo}
\begin{enumerate}[$(i)$]
\item The above distribution $\<Psi>^{(\si)}$ is well defined, i.e. for all $1\leq k\leq \ell$ and for every test function $\vp:\R^d \to \R$ with $\textup{supp}(\vp) \subset B(0,k)\subset B(0,\ell)$, one has
$$\langle \<Psi>^{(\si_k)}(t), \vp \rangle =\langle \<Psi>^{(\si_\ell)}(t), \vp \rangle  \quad \text{on} \ \ \Omega^{(\si)} \, .$$
\item For any test function $\chi:\R^d \to \R$, one has, on $\Omega^{(\si)}$,
$$\chi \cdot\<Psi>_n \underset{n \rightarrow \infty}\rightarrow \chi\cdot \<Psi>^{(\si)} \quad \text{in}\ \ \mathcal{C}([0,T]; \mathcal{W}^{-\al,p}(\mathbb{R}^d)) \, .$$
\item If $\si,\ga \in \cp$, it holds that
$$\<Psi>^{(\si)}=\<Psi>^{(\ga)} \quad \text{on} \ \ \Omega^{(\si)}\cap \Omega^{(\ga)} \, .$$
\end{enumerate}
Due to the latter identification property, we simply set $\<Psi>:=\<Psi>^{(\si)}$, for some fixed element $\si\in \cp$.
\end{proposition}

\begin{proof}  $(i)$ We first show that for   $ 1 \leq k \leq \ell$
\begin{equation}\label{uni}
\<Psi>^{(\sigma_k)}=\sigma_k \<Psi>^{(\sigma_{\ell})}\quad \text{on } \;  \Omega^{(\sigma_k)}\cap \Omega^{(\sigma_{\ell})}.
\end{equation}
By Proposition~\ref{sto},  the sequence $\sigma_k\sigma_{\ell} \<Psi>_n=\sigma_k \<Psi>_n$ converges almost surely to $\<Psi>^{(\sigma_k)}$  in the space $\mathcal{C}([0,T]; \mathcal{W}^{-\al,p}(\mathbb{R}^d))$, on $\Omega^{(\sigma_k)}$.   But by continuity of the product in $\mathcal{C}([0,T]; \mathcal{W}^{-\al,p}(\mathbb{R}^d))$ by the  test function $\sigma_k$, we also have that $\sigma_k\sigma_{\ell} \<Psi>_n$ converges almost surely to $\sigma_k\<Psi>^{(\sigma_{\ell})}$ in $\mathcal{C}([0,T]; \mathcal{W}^{-\al,p}(\mathbb{R}^d))$ on $\Omega^{(\sigma_{\ell})}$, and we deduce \eqref{uni} by uniqueness of the  limit. Therefore by \eqref{uni},  for all $t \in [0,T]$ and  $\vp \in \mathcal{C}^{\infty}_c(\mathbb{R}^d)$ such that $\textup{supp}(\vp) \subset {B}(0,k)$ 
$$
\langle \<Psi>^{(\sigma_k)}(t), \vp \rangle =\langle \sigma_k \<Psi>^{(\sigma_{\ell})}(t), \vp \rangle =\langle  \<Psi>^{(\sigma_{\ell})}(t), \sigma_k\vp \rangle = \langle  \<Psi>^{(\sigma_{\ell})}(t), \vp \rangle\quad \text{on } \;  \Omega^{(\sigma)}.
$$

$(ii)$ Let  $\chi \in \mathcal{C}^{\infty}_c(\mathbb{R}^d)$, then  there exists $k\geq 1$ such that $\textup{supp}(\chi) \subset  {B}(0,k)$. According to Proposition~\ref{sto}, $\chi \<Psi>_n=\chi \sigma_k \<Psi>_n$ converges almost surely to $\chi \<Psi>^{(\sigma_k)}$ in $\mathcal{C}([0,T]; \mathcal{W}^{-\al,p}(\mathbb{R}^d))$ on $\Omega^{(\sigma_k)}$.     But $\chi \<Psi>^{(\sigma_k)}=\chi\<Psi>^{(\sigma)}$ on $\Omega^{(\sigma_k)}$. Indeed, for all $t \in [0,T],$ if $\vp \in \mathcal{C}^{\infty}_c(\mathbb{R}^d),$
$$
\langle \chi\<Psi>^{(\sigma)}(t), \vp \rangle = \langle \<Psi>^{(\sigma)}(t), \chi \vp \rangle =\langle  \<Psi>^{(\sigma_k)}(t), \chi \vp \rangle =\langle  \chi \<Psi>^{(\sigma_k)}(t), \vp \rangle,
$$
where we have used that $\textup{supp}(\chi \vp) \subset {B}(0,k)$ to derive the second equality.

$(iii)$ For all $t \in [0,T]$ and $\vp \in \mathcal{C}^{\infty}_c(\mathbb{R}^d)$ such that $\textup{supp}(\vp) \subset  {B}(0,k),$ we have
$$
\langle \sigma_k \<Psi>_n(t), \vp \rangle = \langle \sigma_k \gamma_k  \<Psi>_n(t),  \vp \rangle = \langle   \gamma_k\<Psi>_n(t), \sigma_k \vp \rangle  = \langle   \gamma_k \<Psi>_n(t),   \vp \rangle,
$$
then by taking the limit $n \longrightarrow +\infty$, we get   $ \langle \<Psi>^{(\sigma_k)}(t), \vp \rangle = \langle \<Psi>^{(\gamma_k)}(t), \vp \rangle $  on  $\Omega^{(\sigma_k)}\cap \Omega^{(\gamma_{k})}$, hence the result.     
\end{proof}

\begin{remark}
The above patching procedure could of course be applied to the second-order process $\chi^2 \<Psi2>$ as well, leading to a well-defined distribution-valued function $\<Psi2>$.
\end{remark}

\section{Deterministic analysis of the equation under \texorpdfstring{condition~\eqref{cond-regu-case}}{H1}}\label{sec:regular}

In this section, we propose to analyze the equation in the \emph{regular} situation, that is when assumption~\eqref{cond-regu-case} on the Hurst index is satisfied, and the linear solution $\rho \<Psi>$ (defined by Proposition~\ref{sto}) is known to be a function of time \emph{and} space. Remember that in this case, the model is interpreted through Definition~\ref{defi:sol-regu}. Thus, what we wish to solve in this section is the equation 
\begin{multline}\label{det-0}
v_t=S_t(\phi)-\imath\int_0^t S_{t-\tau}(\rho^2 |v_{\tau}|^2)\, d\tau-\imath\int_0^t S_{t-\tau}((\rho \overline{v}_{\tau})\cdot(\rho \<Psi>_{\tau}))\, d\tau\\
-\imath\int_0^t S_{t-\tau}((\rho v_{\tau})\cdot(\overline{\rho \<Psi>_{\tau}}))\, d\tau-\imath\int_0^t S_{t-\tau}(|\rho\<Psi>_{\tau}|^2)\, d\tau \, , \quad t\in [0,T]\, .
\end{multline}
%\begin{equation}
%v_t=S_t(\phi)-\imath\int_0^t S_{t-\tau}(\rho^2v_{\tau}^2)\, d\tau-2\imath \int_0^t S_{t-\tau}((\rho v_{\tau})\cdot(\rho \<Psi>_{\tau}))\, d\tau-\imath\int_0^t S_{t-\tau}((\rho \<Psi>_{\tau})^2)\, d\tau \, .
%\end{equation}

\smallskip

As opposed to the stochastic arguments used in the previous section, our strategy towards a (local) solution $v$ for~\eqref{det-0} will rely on deterministic estimates only. In other words, we henceforth consider $\rho\<Psi>$ as a fixed (i.e., non-random) element in the space
$$\mathcal{E}_{\beta}:=\bigcap_{2\leq p\leq \infty}\mathcal{C}([0,T];\mathcal{W}^{\beta,p}(\mathbb{R}^d))= \mathcal{C}([0,T];H^{\beta}(\mathbb{R}^d)) \cap \mathcal{C}([0,T];\mathcal{W}^{\beta,\infty}(\mathbb{R}^d))\, ,$$
for some appropriate $0<\beta<1$ (where $\beta=-\alpha$ is given by Proposition~\ref{sto}), and try to solve the following deterministic equation: for $\luxor \in \ce_\be$,
\begin{multline}\label{det}
v_t=S_t(\phi)-\imath\int_0^t S_{t-\tau}(\rho^2|v_{\tau}|^2)\, d\tau-\imath \int_0^t S_{t-\tau}((\rho \overline{v}_{\tau})\cdot \luxor_{\tau})\, d\tau\\
-\imath \int_0^t S_{t-\tau}((\rho v_{\tau})\cdot \overline{\luxor}_{\tau})\, d\tau-\imath\int_0^t S_{t-\tau}(|\luxor_{\tau}|^2)\, d\tau \, .
\end{multline}

%The solving procedure will be based on a standard fixed-point argument. 
Let us set the stage for this solving procedure by reporting on fundamental estimates related to the two main operations in~\eqref{det}.

\subsection{Pointwise multiplication}

%The following fractional Leibniz rule ( ) will be one of the main keys in our study:
\begin{lemma}[Fractional Leibniz rule, see~{\cite[Proposition 1.1, p. 105]{Taylor}}]\label{lem:frac-leibniz}
Let $s\geq 0$, $1<r<\infty$ and $1<p_1,p_2,q_1,q_2< \infty$ satisfying 
$$\frac{1}{r}=\frac{1}{p_1}+\frac{1}{p_2}=\frac{1}{q_1}+\frac{1}{q_2} \, .$$ 
Then one has
$$\|u\cdot v\|_{\mathcal{W}^{s,r}(\mathbb{R}^d)}\lesssim \|u\|_{\mathcal{W}^{s,p_1}(\mathbb{R}^d)}\|v\|_{L^{p_2}(\mathbb{R}^d)}+\|u\|_{L^{q_1}(\mathbb{R}^d)}\|v\|_{\mathcal{W}^{s,q_2}(\mathbb{R}^d)} \, .$$
%where the proportional constant only depends on $d, s, r, p_1, p_2, q_1, q_2.$
\end{lemma}

\smallskip

\subsection{Convolution with the Schr{\"o}dinger group}

\

\smallskip

Naturally, we also need some control on the operation $(\phi,F,t)\mapsto S_t\phi-\imath\int_0^t S_{t-s}(F_s) \, ds$, or otherwise stated some estimate on the solution $u$ to the general Schr\"{o}dinger equation
\begin{equation}\label{dim-d-schrodinger}
\left\{
\begin{array}{l}
\imath\partial_t u(t,x)-\Delta u(t,x)= F(t,x) \, , \quad  t\in [0,T] \, , \, x\in \R^d \, ,\\
u(0,x)=\phi(x)\, .
\end{array}
\right.
\end{equation}
Such a control is classically provided by the so-called Strichartz inequalities, which will prove to be sufficient for our purpose in this functional setting.

\begin{definition}\label{pair}
A pair $(p,q)$ is said to be Schr\"{o}dinger admissible if $$(p,q) \in [2,\infty]^2,\hspace{0,1cm} (p,q,d)\neq (2,\infty,2),\hspace{0,1cm} \frac{2}{p}+\frac{d}{q}=\frac{d}{2}.$$ 
\end{definition}

\begin{lemma}[Strichartz inequalities, see~{\cite[Paragraph 2.3]{caz}}]\label{lem:strichartz}
Fix $d\geq 1$, $s \in \mathbb{R}$, and let $u$ stand for the mild solution of equation~\eqref{dim-d-schrodinger}.

\smallskip

Then for all Schr\"{o}dinger admissible pairs $(p,q)$ and $(a,b)$,  it holds that
\begin{equation}\label{strichartz}
\|u\|_{L^p([0,T]; \mathcal{W}^{s,q}(\mathbb{R}^{d}))}\lesssim \|\phi\|_{H^s(\mathbb{R}^d)}+\|F\|_{L^{a'}([0,T];\mathcal{W}^{s,b'}(\mathbb{R}^d))}\, ,
\end{equation}
where the notations $a',b'$ refer to the H{\"o}lder conjugates of $a,b$.
\end{lemma}

\

\subsection{Solving the equation}\label{subsec:determ-regu}

\

\smallskip

Our main result regarding equation~\eqref{det} can be stated as follows:
\begin{theorem}\label{thm:regular}
Assume that $1\leq d\leq 4$ and fix $\beta \in (0,1)$. Consider the Schr{\"o}dinger admissible pair $(p,q)$ given by the formulas
 $$p=\frac{12}{d-\beta} \ , \ q=\frac{6d}{2d+\beta} \, ,$$
and for every $T>0$, define the space $X^{\beta}(T)$ as
$$X^{\beta}(T):=\mathcal{C}([0,T]; H^\beta(\mathbb{R}^d))\cap L^p([0,T];\mathcal{W}^{\beta,q}(\mathbb{R}^d)).$$
Then for all $\phi \in H^\beta(\mathbb{R}^d)$ and $\luxor \in \mathcal{E}_{\beta}$, there exists a time $T_0>0$ such that equation~\eqref{det} admits a unique solution in $X^{\beta}(T_0)$.
\end{theorem}

\smallskip

This local well-posedness result will be derived from a standard fixed-point argument. To this end, we introduce the map $\Gamma$ defined by the right-hand side of~\eqref{det}, that is: for all $\luxor \in \mathcal{E}_{\beta}$, $\phi \in H^\beta(\mathbb{R}^d)$, $v\in X^{\beta}(T)$, $T\geq 0$ and $t\in [0,T]$, set
%\begin{equation*} 
%\Gamma_{T, \luxo}(v)_{t}:=S_t(\phi)-\imath\int_0^t S_{t-\tau}(\rho^2v_{\tau}^2)\, d\tau-2\imath\int_0^t S_{t-\tau}((\rho v_{\tau})\cdot \luxo_{\tau})\, d\tau-\imath\int_0^t S_{t-\tau}((\luxo_\tau)^2)\, d\tau \, .
%\end{equation*}
\begin{multline*}
\Gamma_{T, \luxor}(v)_{t}:=S_t(\phi)-\imath\int_0^t S_{t-\tau}(\rho^2|v_{\tau}|^2)\, d\tau-\imath \int_0^t S_{t-\tau}((\rho \overline{v}_{\tau})\cdot \luxor_{\tau})\, d\tau\\
-\imath \int_0^t S_{t-\tau}((\rho v_{\tau})\cdot \overline{\luxor}_{\tau})\, d\tau-\imath\int_0^t S_{t-\tau}(|\luxor_{\tau}|^2)\, d\tau \, .
\end{multline*}

\begin{proposition}\label{control-regular}
In the setting of Theorem~\ref{thm:regular}, the following bounds hold true: there exists $\varepsilon>0$ such that for all $0\leq T\leq 1$, $\phi \in H^\beta(\mathbb{R}^d), (\luxor_1,\luxor_2) \in \mathcal{E}_{\beta}\times \mathcal{E}_{\beta}$ and $v, v_1, v_2 \in X^\beta(T)$, 
\begin{equation}\label{bound1}
\|\Gamma_{T, \luxor_1}(v)\|_{X^\beta(T)}\lesssim \|\phi\|_{H^\beta(\mathbb{R}^d)} +T^{\varepsilon}\Big[\|v\|_{X^\beta(T)}^2+\| \luxor_1\|_{L^\infty_T \mathcal{W}^{\beta,q}}\|v\|_{X^\beta(T)}+\|\luxor_1\|_{L^\infty_T\mathcal{W}^{\beta,q}}^2\Big] \, ,
\end{equation}
and
\begin{align}
&\|\Gamma_{T, \luxor_1}(v_1)-\Gamma_{T, \luxor_2}(v_2)\|_{X^\beta(T)}\nonumber \\
&\lesssim T^{\varepsilon}\Big[\|v_1-v_2\|_{X^\beta(T)}\big\{\|v_1\|_{X^\beta(T)}+\| v_2\|_{X^\beta(T)}\big\}+\|\luxor_1-\luxor_2\|_{L^\infty_T \mathcal{W}^{\be,q}}\| v_1\|_{X^\beta(T)}\nonumber\\
&\hspace{1cm}+\|\luxor_2\|_{L^\infty_T\mathcal{W}^{\be,q}}\|v_1-v_2\|_{X^\beta(T)}+\|\luxor_1-\luxor_2\|_{L^\infty_T \mathcal{W}^{\be,q}}\big\{\|\luxor_1\|_{L^\infty_T\mathcal{W}^{\be,q}}+\|\luxor_2\|_{L^\infty_T\mathcal{W}^{\be,q}}\big\}\Big]\, ,\label{bound2}   
\end{align}
where the proportional constants only depend on $s$ and $\rho$.
\end{proposition}

Before we turn to the proof of this proposition, let us briefly recall that, once endowed~\eqref{bound1}-\eqref{bound2}, the statement of Theorem~\ref{thm:regular} follows from a standard two-step procedure. Namely, using~\eqref{bound1}, we can first establish that for any $T=T(\phi,\luxor)>0$ small enough, there exists a ball in $X^{\beta}(T)$ that is stable through the application of $\Gamma_{T, \luxor}$. Then, thanks to~\eqref{bound2} (applied with $\luxor_1=\luxor_2=\luxor$), we can show that $\Gamma_{T, \luxor}$ is actually a contraction on this ball (for $T>0$ possibly even smaller), which completes the proof of the assertion.

\smallskip

Note also that the continuity of $\Gamma_{T, \luxor}$ with respect to $\luxor$ (an immediate consequence of~\eqref{bound2}) will be the key ingredient toward item $(ii)$ of Theorem~\ref{resu}.

\smallskip

\begin{proof}[Proof of Proposition~\ref{control-regular}]
Let us set, for any suitable distribution $u$ on $\mathbb{R}^{d+1}$,
$$\cg(u)_t:=-\imath \int_0 ^t S_{t-\tau}(u_{\tau})\, d\tau\, ,$$
which allows to recast $\Gamma_{T,\luxor}$ as
$$\Gamma_{T,\luxor}(v)=S(\phi)+\cg(\rho^2 |v|^2)+\cg(\rho \overline{v} \cdot\luxor)+\cg(\rho v \cdot \overline{\luxor})+\cg(|\luxor|^2) \, .$$
Let us now bound each of the four above terms separately.

\

\noindent
\textbf{Bound on $S(\phi)$:}
Since $(\infty,2)$ and $(p,q)$ are both Schr\"{o}dinger admissible pairs, we can apply Lemma~\ref{lem:strichartz} to assert that
\begin{equation}\label{estima-1}
\|S(\phi)\|_{X^\beta(T)}\lesssim \|\phi\|_{H^\beta} \, .
\end{equation}

\

\noindent
\textbf{Bound on $\cg(\rho^2 |v|^2)$:}
By Lemma~\ref{lem:strichartz}, we can first assert that
$$\|\cg(\rho^2|v|^2)\|_{X^\beta(T)}\lesssim \|\rho^2|v|^2\|_{L^{p'}_T\mathcal{W}^{\beta,q'}}\, .$$
Let us now introduce the additional parameter $n:=\frac{3d}{d-\beta}> 1$, in such a way that
$$\frac{1}{q'}=\frac{1}{q}+\frac{1}{n} \, .$$
Using the fractional Leibniz rule given by Lemma~\ref{lem:frac-leibniz}, we get that for all $t\geq 0$,
\begin{eqnarray*}
\|\rho^2|v|^2(t,.)\|_{\mathcal{W}^{\beta,q'}}\lesssim \|\rho v(t,.)\|_{\mathcal{W}^{\beta,q}}\|\rho v(t,.)\|_{L^n} \, .
\end{eqnarray*}
It is easy to check that $\beta \geq d(\frac{1}{q}-\frac{1}{n})$, and accordingly we can rely on the Sobolev embedding 
\begin{equation}\label{sobol-emb}
\mathcal{W}^{\beta,q}(\mathbb{R}^d)\hookrightarrow L^n(\mathbb{R}^d)
\end{equation}
to derive that
$$\|\rho^2|v|^2(t,.)\|_{\mathcal{W}^{\beta,q'}}\lesssim \|\rho v(t,.)\|_{\mathcal{W}^{\beta,q}}^2\, .$$
Consider the parameter $m$ defined through the relation
$$\frac{1}{p'}=\frac{2}{p}+\frac{1}{m} \, .$$ 
Since $1\leq d\leq 4$ and $\beta \in (0,1)$, it can actually be verified that $\frac{1}{m}=1-\frac{d-\beta}{4} >0$. Then, by H\"{o}lder inequality, we have
\begin{equation*}
\|\rho^2|v|^2\|_{L^{p'}_T\mathcal{W}^{\beta,q'}}\lesssim T^{\frac{1}{m}}\|v\|_{L^p_T\mathcal{W}^{\beta,q}}^2 \lesssim T^{1-\frac{d-\beta}{4}}\|v\|_{X^\beta(T)}^2 \, ,
\end{equation*}
and we have thus shown that 
\begin{equation}\label{estima-2}
\|\cg(\rho^2|v|^2)\|_{X^\beta(T)}\lesssim T^{1-\frac{d-\beta}{4}}\|v\|_{X^\beta(T)}^2\, .
\end{equation}

\

\noindent
\textbf{Bound on $\cg(\rho \overline{v} \cdot\luxor)$, $\cg(\rho v \cdot \overline{\luxor})$:}
Just as above, we can first apply Lemma~\ref{lem:strichartz} to get that
\begin{equation*}
\|\cg(\rho \overline{v} \cdot\luxor)\|_{X^\beta(T)}+\|\cg(\rho v \cdot \overline{\luxor})\|_{X^\beta(T)}\lesssim   \|\rho \overline{v}\cdot \luxor\|_{ L^{p'}_T \mathcal{W}^{\beta,q'}}.
\end{equation*}
Thanks to Lemma~\ref{lem:frac-leibniz}, it holds, for all $t\geq 0,$
\begin{eqnarray*}
 \|\rho \overline{v}\cdot \luxor(t,.)\|_{ \mathcal{W}^{\beta,q'}}  &\lesssim &\|\rho v(t,.)\|_{\mathcal{W}^{\beta,q}}\|\luxor(t,.)\|_{L^n}+\|\luxor(t,.)\|_{\mathcal{W}^{\beta,q}}\|\rho v(t,.)\|_{L^n}\\
&\lesssim&\|\luxor(t,.)\|_{\mathcal{W}^{\beta,q}}\|\rho v(t,.)\|_{\mathcal{W}^{\beta,q}}\, ,
\end{eqnarray*}
where we have again used the Sobolev embedding~\eqref{sobol-emb}.

\smallskip

Then, by H\"{o}lder inequality, we deduce
\begin{align*}
\|\rho \overline{v}\cdot \luxor\|_{ L^{p'}_T \mathcal{W}^{\beta,q'}} &\lesssim T^{\frac{1}{p}+\frac{1}{m}}\|\luxor\|_{L^\infty_T\mathcal{W}^{\beta,q}}\|v\|_{L^p_T\mathcal{W}^{\beta,q}} \\
& \lesssim T^{\frac{1}{p}+\frac{1}{m}}\|\luxor\|_{L^\infty_T \mathcal{W}^{\beta,q}}\|v\|_{X^\beta(T)}\, ,
\end{align*}
and we have thus established that
\begin{equation}\label{estima-3}
\|\cg(\rho \overline{v} \cdot\luxor)\|_{X^\beta(T)}+\|\cg(\rho v \cdot \overline{\luxor})\|_{X^\beta(T)}\lesssim T^{\frac{1}{p}+\frac{1}{m}}\|\luxor\|_{L^\infty_T\mathcal{W}^{\beta,q}}\|v\|_{X^\beta(T)} \, .
\end{equation}

\

\noindent
\textbf{Bound on $\cg(|\luxor|^2)$:}
By Lemma~\ref{lem:strichartz}, 
$$\|\cg(|\luxor|^2)\|_{X^\beta(T)}\lesssim \||\luxor|^2\|_{L^{p'}_T\mathcal{W}^{\be,q'}}\, .$$
Using Lemma~\ref{lem:frac-leibniz} and the Sobolev embedding~\eqref{sobol-emb}, we get that for every $t\geq 0,$
$$
\||\luxor|^2(t,.)\|_{\mathcal{W}^{\be,q'}}\lesssim \|\luxor(t,.)\|_{\mathcal{W}^{\be,q}}\|\luxor(t,.)\|_{L^n}\lesssim\|\luxor(t,.)\|_{\mathcal{W}^{\be,q}}^2\, .
$$
Then
$$
\||\luxor|^2\|_{L^{p'}_T\mathcal{W}^{\be,q'}}\lesssim T^{\frac{1}{p'}}\|\luxor\|_{L^\infty_T\mathcal{W}^{\be,q}}^2 \, ,
$$
and finally
\begin{equation}\label{estima-4}
\|\cg(|\luxor|^2)\|_{X^\beta(T)}\lesssim T^{\frac{1}{p'}}\|\luxor\|_{L^\infty_T\mathcal{W}^{\be,q}}^2 \, .
\end{equation}

\

The combination of estimates~\eqref{estima-1},~\eqref{estima-2},~\eqref{estima-3} and~\eqref{estima-4} entails the desired bound~\eqref{bound1}. 

\

It is then easy to see that~\eqref{bound2} can be derived from similar arguments: for instance,
\begin{equation*}
\|\cg(\rho^2(|v_1|^2-|v_2|^2))\|_{X^\beta(T)}\lesssim \|\rho^2(|v_1|^2-|v_2|^2)\|_{L^{p'}_T\mathcal{W}^{\be,q'}}\lesssim \||v_1|^2-|v_2|^2\|_{L^{p'}_T\mathcal{W}^{\be,q'}}\, .
\end{equation*}
Combining again Lemma~\ref{lem:frac-leibniz} and embedding~\eqref{sobol-emb}, we obtain, for every $t\geq 0$,
\begin{eqnarray*}
\|(|v_1|^2-|v_2|^2)(t,.)\|_{\mathcal{W}^{\be,q'}}&\lesssim &\|(v_1-v_2)(t,.)\|_{\mathcal{W}^{\be,q}}\{\|v_1(t,.)\|_{L^n}+\|v_2(t,.)\|_{L^n}\}\nonumber\\
& &+\|(v_1-v_2)(t,.)\|_{L^n}\{\|v_1(t,.)\|_{\mathcal{W}^{\be,q}}+\|v_2(t,.)\|_{\mathcal{W}^{\be,q}}\}\nonumber\\
& \lesssim&\|(v_1-v_2)(t,.)\|_{\mathcal{W}^{\be,q}}\{\|v_1(t,.)\|_{\mathcal{W}^{\be,q}}+\|v_2(t,.)\|_{\mathcal{W}^{\be,q}}\}\nonumber\, ,
\end{eqnarray*}
and as a result
$$\||v_1|^2-|v_2|^2\|_{L^{p'}_T\mathcal{W}^{\be,q'}}\lesssim T^{\frac{1}{m}}\|v_1-v_2\|_{L^p_T\mathcal{W}^{\be,q}}\{\|v_1\|_{L^p_T\mathcal{W}^{\be,q}}+\|v_2\|_{L^p_T\mathcal{W}^{\be,q}} \} \, .$$
\end{proof}

\subsection{Proof of Theorem~\ref{resu}}\label{subsec:proof-main-theo-regu}

\

\smallskip

At this point, the statement of Theorem~\ref{resu} (item $(i)$) is of course a mere combination of the construction of $\rho\<Psi>$ as an element in $\ce_\be$ (Proposition~\ref{sto}) with the well-posedness result of Theorem~\ref{thm:regular}. In brief, it suffices to apply Theorem~\ref{thm:regular} (in an almost sure way) to $\luxor:=\rho\<Psi>$.

\smallskip

As for the convergence property in item $(ii)$, it can easily be deduced from the continuity of $\Gamma_{T,\luxor}$ with respect to $\luxor$ (along~\eqref{bound2}) and the almost sure convergence of $\chi \<Psi>_n$ to $\chi \<Psi>$. Additional details about this elementary procedure can be found in the proof of~\cite[Theorem 1.7]{deya-wave}.

\

\section{Deterministic analysis of the equation under condition~\texorpdfstring{\eqref{cond-irreg-case}}{H2'}}\label{sec:irreg-case-det}

It remains us to deal with the wellposedness issue in the rough case, that is to present the proof of Theorem~\ref{resu1}. Therefore, we assume in this section that condition~\eqref{cond-irreg-case} on the Hurst indexes is satisfied. We recall that in this rough situation, the equation is understood in the sense of Definition~\ref{defi:sol}, that is as
\begin{multline}\label{ccl-sto}
v_t=S_t(\phi)-\imath\int_0^t S_{t-\tau}(\rho^2 |v_{\tau}|^2)\, d\tau-\imath\int_0^t S_{t-\tau}((\rho \overline{v}_{\tau})\cdot(\rho \<Psi>_{\tau}))\, d\tau\\
-\imath\int_0^t S_{t-\tau}((\rho v_{\tau})\cdot(\overline{\rho \<Psi>_{\tau}}))\, d\tau-\imath\int_0^t S_{t-\tau}(\rho^2 \<Psi2>_{\tau})\, d\tau\, , 
\end{multline}
%\begin{equation}\label{ccl-sto}
%v_t=S_t(\phi)-\imath\int_0^t S_{t-\tau}(\rho^2v_{\tau}^2)\, d\tau-2\imath\int_0^t S_{t-\tau}((\rho v_{\tau})\cdot(\rho \<Psi>_{\tau}))\, d\tau-\imath\int_0^t S_{t-\tau}(\rho^2 \<Psi2>_{\tau})\, d\tau
%\end{equation}
where the processes $\rho \<Psi>$ and $\rho^2 \<Psi2>$ are defined through Proposition~\ref{sto} and Proposition~\ref{sto1}.

\smallskip

In order to handle~\eqref{ccl-sto}, we intend to follow the same deterministic approach as in Section~\ref{sec:regular}. In other words, we will henceforth consider the pair $(\rho \<Psi>,\rho^2 \<Psi2>)$ as a given element in the space
\begin{equation}\label{defi-space-e-al}
\mathcal{R}_\al:=L^\infty\big([0,T];\mathcal{W}^{-\alpha,\infty}(\mathbb{R}^d)\big)\times L^\infty\big([0,T];H^{-2\alpha}(\mathbb{R}^d)\big) \, ,
\end{equation}
for some $0<\al<1$ (provided by Propositions~\ref{sto} and~\ref{sto1}), and then try to solve the more general deterministic equation: for $(\luxo,\cherry) \in \mathcal{R}_{\alpha}$,
%\begin{equation}
%v_t=S_t(\phi)-\imath\int_0^t S_{t-\tau}(\rho^2v_{\tau}^2)\, d\tau-2\imath\int_0^t S_{t-\tau}((\rho v_{\tau})\cdot(\luxo_{\tau}))\, d\tau-\imath\int_0^t S_{t-\tau}(\cherry_{\tau})\, d\tau \, .
%\end{equation}
\begin{multline}\label{ccl}
v_t=S_t(\phi)-\imath\int_0^t S_{t-\tau}(\rho^2 |v_{\tau}|^2)\, d\tau-\imath\int_0^t S_{t-\tau}((\rho \overline{v}_{\tau})\cdot \luxo_{\tau})\, d\tau\\
-\imath\int_0^t S_{t-\tau}((\rho v_{\tau})\cdot  \overline{\luxo_\tau})\, d\tau-\imath\int_0^t S_{t-\tau}(\cherry_{\tau})\, d\tau\, .
\end{multline}

As we have already highlighted it in Section~\ref{subsec:wellposed-rough-intro}, the whole specificity of the situation (in comparison to Section~\ref{sec:regular}) lies in the irregularity of $\luxo_\tau$ and $\cherry_\tau$, which can only be treated as negative-order distributions (note indeed that $\al >0$ in~\eqref{defi-space-e-al}). The technical ingredients towards a fixed-point argument need to be revised accordingly: this will be the purpose of the subsequent Sections~\ref{subsec:point-mult-inter}-\ref{subsec:commutator}, which lay the ground for our main wellposedness result, namely Theorem~\ref{thm:noregular}.

\subsection{Pointwise multiplication and interpolation}\label{subsec:point-mult-inter}

\

\smallskip

In view of the above considerations, our only possibility to handle the product $(\rho v_{\tau})\cdot(\luxo_{\tau})$ in~\eqref{ccl} will be to rely on the following general multiplication property in Sobolev spaces (see e.g.~\cite[Section 4.4.3]{prod} for a proof of this result):
\begin{lemma}\label{lem:product}
Fix $d\geq 1$. Let $\al,\be >0$ and $1 \leq p,p_1,p_2\leq \infty$ be such that
$$\frac{1}{p}= \frac{1}{p_1}+\frac{1}{p_2} \quad \text{and} \quad 0<\al<\be<\frac{d}{p_2}  \, .$$
If $f\in \cw^{-\al,p_1}(\R^d)$ and $g\in \cw^{\be,p_2}(\R^d)$, then $f\cdot g\in \cw^{-\al,p}(\R^d)$ and
$$
\| f\cdot g\|_{\cw^{-\al,p}} \lesssim \|f\|_{\cw^{-\al,p_1}} \| g\|_{\cw^{\be,p_2}} \ .$$

\end{lemma}

Let us also label the following classical interpolation result for further reference:

\begin{lemma}\label{lem:interpol}
Fix $d\geq 1$. Let $s,s_1,s_2\in \R$ and $1\leq p,p_1,p_2<\infty$ be such that, for some $\theta \in (0,1)$,
$$s=\theta s_1+(1-\theta)s_2 \quad \text{and} \quad \frac{1}{p}=\frac{\theta}{p_1}+\frac{1-\theta}{p_2} \, .$$
Then for every $v\in \cw^{s_1,p_1}(\R^d)\cap \cw^{s_2,p_2}(\R^d)$, it holds that $v\in \cw^{s,p}(\R^d)$ and
$$\|v\|_{\cw^{s,p}}\leq \|v\|_{\cw^{s_1,p_1}}^\theta \|v\|_{\cw^{s_2,p_2}}^{1-\theta}\, .$$
\end{lemma}

\smallskip

\subsection{A local regularization property of the Schr\"{o}dinger group $S$}\label{subsec:local-smoothing}

\

\smallskip

It is a well-known fact that the classical Strichartz inequalities for the Schr{\"o}dinger group (summed up in Lemma~\ref{lem:strichartz}) do not offer any regularization effect, as can be seen from the constant derivative parameter $s$ in~\eqref{strichartz}. This phenomenon naturally becomes a fundamental obstacle in our rough setting, where, for stability reasons, the distribution $(\rho v_{\tau})\cdot(\luxo_{\tau})$ in~\eqref{ccl} is expected to turn into a function through the action of $S$. 

\smallskip

A possible way to reach such a regularization property is to let \emph{local} Sobolev topologies come into picture, through the consideration of the spaces $H^s_\rho(\R^d)$ defined by~\eqref{defi:h-rho}. Our main technical result in this direction can be stated as follows: 

\begin{lemma}\label{lem:loc-regu}
Fix $d\geq 1$. Let $\rho:\R^d \to \R$ be of the form~\eqref{form-rho}, $0\leq \al,\kappa \leq\frac{1}{2}$ and $0\leq T\leq 1$. Assume that $\phi \in H^{-\alpha}(\mathbb{R}^d)$, $F \in L^1([0,T];H^{-\alpha}(\mathbb{R}^d))$, and consider the solution $u$ of the following inhomogeneous Schr\"{o}dinger equation on $\mathbb{R}^d$
\begin{equation*}
\left\{
\begin{array}{l}
\imath \partial_t u(t,x)-\Delta u(t,x)= F(t,x) \, , \quad  t\in [0,T] \, , \, x\in \R^d \, ,\\
u_0=\phi \, .
\end{array}
\right.
\end{equation*}
Then it holds that 
\begin{equation}\label{ineq-loc-reg}
\|u\|_{L^{\frac{1}{\kappa}}_T H^{-\al+\ka}_\rho}\lesssim \|\phi\|_{H^{-\alpha}(\mathbb{R}^d)}+\|F\|_{L^1_T H^{-\alpha}}\, ,
\end{equation}
where the proportional constant only depends on $\rho$, $\al$ and $\ka$.
\end{lemma}

\smallskip

The above property can in fact be seen as a slight extension of the result of~\cite[Theorem 3.1]{constantin-saut}. For the sake of clarity, we have postponed the proof of the lemma to Section~\ref{subsec:proof-loc-reg}.

\smallskip

\subsection{A commutator estimate}\label{subsec:commutator}

\

\smallskip

Keeping our objective in mind (that is, to settle a fixed-point argument for~\eqref{ccl}), the previous estimate~\eqref{ineq-loc-reg} clearly lacks some stability: the left-hand side is indeed based on the consideration of a local Sobolev norm (in $H^{-\al+\ka}_\rho$), while the right-hand side appeals to a standard Sobolev space ($H^{-\alpha}$).

\smallskip

Our strategy to overcome this problem will consist in using the presence of the cut-off function~$\rho$ \emph{within the model~\eqref{ccl}} (through $\rho v$), which somehow allows us to turn global Sobolev norms into local ones. To implement this idea, an additional commutator-type estimate will be required:

\begin{lemma}\label{lem:commut}
For every $s >0$ and for all test functions $\rho, g:\R^d \to \R$, it holds that
\begin{equation}\label{improv-boun}
 \|(\emph{\id}-\Delta)^{\frac{s}{2}}(\rho\cdot g)-\rho \cdot (\emph{\id}-\Delta)^{\frac{s}{2}}(g)\|_{L^2(\mathbb{R}^d)}\lesssim \|g\|_{H^{s-1}(\mathbb{R}^d)} \, ,
\end{equation}
where the proportional constant only depends on $\rho$ and $s$.

\smallskip

As a consequence, for every test function $\rho:\R^d \to \R$ and for every $g\in H^s_\rho(\R^d) \cap H^{s-1}(\R^d)$, it holds that
\begin{equation}\label{commut-1}
\|\rho \cdot g\|_{H^s} \lesssim \|g\|_{H^s_\rho}+\|g\|_{H^{s-1}} 
\end{equation}
and 
\begin{equation}\label{commut-2}
\|g\|_{H^s_\rho} \lesssim \|\rho \cdot g\|_{H^s}+\|g\|_{H^{s-1}} \, ,
\end{equation}
for some proportional constant depending only on $\rho$ and $s$.
\end{lemma}

\smallskip

\begin{proof}
See Section~\ref{subsec:proof-commut}.
\end{proof}

\subsection{Solving the auxiliary deterministic equation}\label{subsec:solving-irreg}

\

\smallskip

Let us fix (once and for all) a cut-off function $\rho:\R^d\to\R$ of the form~\eqref{form-rho}, and for all $T\geq 0$, $\al,\ka>0$, $p,q\geq 2$, define the space 
\begin{equation}\label{scale-spaces-x}
X^{\alpha,\ka,(p,q)}_\rho(T):=\mathcal{C}([0,T]; H^{-2\alpha}(\mathbb{R}^d))\cap L^p([0,T]; \cw^{-2\alpha,q}(\mathbb{R}^d))\cap L^{\frac{1}{\ka}}_T H^{-2\al+\ka}_\rho \, .
\end{equation}
Besides, recall that the space $\mathcal{R}_\al$ has been introduced in~\eqref{defi-space-e-al}.

\smallskip

We are finally in a position to state (and prove) the main result of this section:

\begin{theorem}\label{thm:noregular}
Assume that $1\leq d\leq 3$ and that 
\begin{equation}\label{condi-al}
0<\al<
\left\{
\begin{array}{l}
\frac{3}{20} \quad \text{if} \ d=1\\
\frac{1}{10} \quad \text{if} \ d=2\\
\frac{1}{24} \quad \text{if} \ d=3 \ .
\end{array}
\right.
\end{equation}
Then one can find parameters $\ka>0$ and $p,q\geq 2$ such that for all $\phi \in H^{-2\alpha}(\mathbb{R}^d)$ and $(\luxo,\cherry) \in \mathcal{R}_{\alpha}$, there exists a time $T>0$ for which equation~\eqref{ccl} admits a unique solution in the above-defined set $X^{\alpha,\ka,(p,q)}_\rho(T)$.
\end{theorem}

\smallskip

\begin{remark}
As could be checked from a review of our arguments in the below proof, the condition~\eqref{condi-al} on $\al$ is essentially optimal \emph{with respect to the spaces and the tools that we have relied on}. To be more specific, condition~\eqref{condi-al} is derived from an optimal choice of the four parameters $\al,\ka,p,q$ in the scale of spaces~\eqref{scale-spaces-x}, when using Lemmas~\ref{lem:product}-\ref{lem:commut} to estimate the right-hand side of~\eqref{ccl}.\\
\indent We do not pretend that this restriction on $\al$ could not be alleviated by considering a different solution space, or using more sophisticated tools to control the equation.
\end{remark}

\smallskip

\begin{remark}\label{rk:bourgain-spaces}
As we mentionned it in the introduction, the well-posedness of similar (deterministic) quadratic NLS has already been studied in the literature. A recurrent ingredient consists of sharp bilinear estimates prevailing in the so-called Bourgain spaces (see \cite{Co-De-Ke-Sta,Beje-DeSilva}). However, it seems to us that those techniques could not be directly applied to our problem, for two reasons.\\
\indent Firstly, it is not clear how the term $\rho^2 |u|^2$ could be treated through the bilinear estimates of \cite{Co-De-Ke-Sta}, since the Bourgain spaces $X^{s,b}$ are not stable by multiplication with a $\mathcal{C}_c^{\infty}$ function\footnote{We thank Jean-Marc Delort for this remark.}. Secondly, even if we replace $\rho$ with $1$ in the initial problem \eqref{equa-abstract} (thus getting access to sharp bilinear estimates for $|u|^2$), it is unlikely that the stochastic terms $\<Psi>$ and $\<Psi2>$ can then be injected into Bourgain spaces, owing to the spatial asymptotic behavior of those processes.
%\indent This is essentially why we have instead adopted the above framework and strategy, combining the usual Strichartz estimates with the Kato smoothing effect. 
\end{remark}

\

Just as in Section~\ref{subsec:determ-regu}, the proof of Theorem~\ref{thm:noregular} is in fact a straightforward consequence of the following estimates for the map $\Gamma_{T, \luxo,\cherry}$ defined for all $T\geq 0$ and $(\luxo,\cherry)\in \mathcal{R}_\al$ by
\begin{equation*} 
\Gamma_{T, \luxo,\cherry}(v):=S(\phi)+\cg(\rho^2 |v|^2)+\cg(\rho \overline{v} \cdot\luxo)+\cg(\rho v \cdot\overline{\luxo})+\cg(\cherry)\, ,
\end{equation*}
where the shortcut notation $\cg$ refers to the operator
$$\cg(u)_t:=-\imath \int_0 ^t S_{t-\tau}(u_{\tau})\, d\tau\, .$$

\begin{proposition}\label{last}
Assume that $1\leq d\leq 3$ and that $\al$ satisfies condition~\eqref{condi-al}. Then one can find parameters $\ka>0$, $p,q\geq 2$ and $\varepsilon >0$ such that, setting $X(T):=X^{\alpha,\ka,(p,q)}_\rho(T)$, the following bounds hold true: for all $0\leq T\leq 1$, $\phi \in H^{-2\alpha}(\mathbb{R}^d), (\luxo_1,\cherry_1) \in \mathcal{R}_{\alpha},  (\luxo_2,\cherry_2) \in \mathcal{R}_{\alpha}$ and $v, v_1, v_2 \in X(T)$,

\smallskip

\begin{equation}
\|\Gamma_{T, \luxo_1,\cherry_1}(v)\|_{X(T)}\lesssim  \|\phi\|_{H^{-2\alpha}}+T^{\varepsilon}\Big[\|v\|_{X(T)}^2+\|\luxo_1\|_{L^\infty_T \mathcal{W}^{-\alpha,\infty}} \|v\|_{X(T)}+\|\cherry_1\|_{L^\infty_T H^{-2\alpha}}\Big]\, , \label{bound7}
\end{equation}
and
\begin{align}
&\|\Gamma_{T, \luxo_1, \cherry_1}(v_1)-\Gamma_{T, \luxo_2, \cherry_2}(v_2)\|_{X(T)}
\nonumber \\
&\lesssim \ T^{\varepsilon}\Big[\|v_1-v_2\|_{X(T)}\{\|v_1\|_{X(T)}+\|v_2\|_{X(T)}\}+\|\luxo_1-\luxo_2\|_{L^\infty_T \mathcal{W}^{-\alpha,\infty}}\|v_1\|_{X(T)}\nonumber\\
&\hspace{4cm}+\|\luxo_2\|_{L^\infty_T \mathcal{W}^{-\alpha,\infty}}\| v_1-v_2\|_{X(T)}+\|\cherry_1-\cherry_2\|_{L^\infty_T H^{-2\alpha}}\Big]\label{bound8} \, , 
\end{align}
where the proportional constants depend only on $\rho$ and $\al$.
\end{proposition}

The choice of the three parameters $\ka,p,q$ in the above proposition highly depends on the space dimension $d\in \{1,2,3\}$. For the sake of clarity, let us consider each value of $d$ in a distinct subsection.

\smallskip

\subsubsection{Proof of Proposition~\ref{last} when $d=1$}\label{sec:proof-dim-1}

\

\smallskip

\textit{In this situation, we pick $\ka$ such that $3\al<\ka <\inf(\frac12,\frac34-2\al)$ and $(p,q):=(\infty,2)$, so that the space under consideration reduces to
$$X(T):=\mathcal{C}([0,T]; H^{-2\alpha}(\mathbb{R}))\cap  L^{\frac{1}{\ka}}_T H^{-2\al+\ka}_\rho \, .$$
Also, we set $\theta:=\frac{2\al}{\ka}\in (0,\frac23)$.}

\

%\tbc{Recall that we set, for any suitable distribution $u$ on $\mathbb{R}^{d+1}$,
%$$\cg(u)_t:=\int_0 ^t S_{t-\tau}(u_{\tau})\, d\tau \, .$$}
We now bound each term in the expression of $\Gamma_{T, \luxo, \cherry}$ separately. In the sequel we assume that $0 \leq T \leq 1$.

\

\noindent
\textbf{Bound on $S(\phi)$:}
As $S$ is a unitary operator on $H^{-2\alpha}(\mathbb{R})$, one has
$$\|S(\phi)\|_{L^{\infty}_T H^{-2\al}}= \|\phi\|_{H^{-2\alpha}}\, .$$
Besides, since $\alpha\leq \frac14$ and $\ka\leq \frac12$, we can apply Lemma~\ref{lem:loc-regu} to assert that
$$\|S(\phi)\|_{L^{\frac{1}{\ka}}_T H^{-2\al+\ka}_\rho}\lesssim  \|\phi\|_{H^{-2\alpha}}\, ,$$
and we have thus shown that
$$\|S(\phi)\|_{X(T)}\lesssim \|\phi\|_{H^{-2\alpha}}.$$

\

\noindent
\textbf{Bound on $\cg(\rho^2 |v|^2)$:}
Since $\ka >0$, and since $\rho$ is smooth and compactly-supported, one has
\begin{align*}
\|\cg(\rho^2 |v|^2)\|_{X(T)}&=\|\cg(\rho^2 |v|^2)\|_{L^{\infty}_T H^{-2\al}}+\|\cg(\rho^2 |v|^2)\|_{L^{\frac{1}{\ka}}_T H^{-2\al+\ka}_\rho}\\
&\lesssim \|\cg(\rho^2 |v|^2)\|_{L^{\infty}_T H^{-2\al+\ka}}+\|\cg(\rho^2 |v|^2)\|_{L^{\frac{1}{\ka}}_T H^{-2\al+\ka}}\\
& \lesssim \| \cg(\rho^2 |v|^2)\|_{L^{\infty}_T H^{-2\al+\ka}}\, .\label{refer-proof-dim-1}
\end{align*}
From here we can apply Strichartz inequality (Lemma~\ref{lem:strichartz}) to assert that
\begin{equation}\label{appli-schro-dim-1}
\|\cg(\rho^2 |v|^2)\|_{X(T)}\lesssim \|\rho^2 |v|^2\|_{L^{\frac43}_T \mathcal{W}^{-2\al+\ka,1}} \, .
\end{equation}
By Lemma~\ref{lem:frac-leibniz}, one has, for every fixed $t\geq 0$,
$$
\|\rho^2 |v|^2(t,.)\|_{\mathcal{W}^{-2\al+\ka,1}}\lesssim \| \rho v(t,.)\|_{H^{-2\al+\ka}}\| \rho v(t,.)\|_{L^2}\, ,
$$
and then, by Lemma~\ref{lem:interpol},
$$\| \rho v(t,.)\|_{L^2}\leq  \| \rho v(t,.)\|_{H^{-2\al+\ka}}^\theta \| \rho v(t,.)\|_{H^{-2\al}}^{1-\theta} \, ,$$
which entails, for every fixed $t\geq 0$,
\begin{align*}
\|\rho^2 |v|^2(t,.)\|_{\mathcal{W}^{-2\al+\ka,1}}&\lesssim \| \rho v(t,.)\|_{H^{-2\al+\ka}}^{1+\theta}\| \rho v(t,.)\|_{H^{-2\al}}^{1-\theta} \\
&\lesssim \|  v(t,.)\|_{H^{-2\al+\ka}_\rho}^{1+\theta}\| v(t,.)\|_{H^{-2\al}}^{1-\theta}+\| v(t,.)\|_{H^{-2\al}}^{2}   \, ,
\end{align*}
where we have used Lemma~\ref{lem:commut} to derive the second inequality.

\smallskip

As a result,
\begin{align*}
&\int_0^T dt \, \|\rho^2 |v|^2(t,.)\|_{\mathcal{W}^{-2\al+\ka,1}}^{\frac43}\\
&\lesssim \|v\|_{X(T)}^{\frac43(1-\theta)}  \int_0^T dt \, \|  v(t,.)\|_{H^{-2\al+\ka}_\rho}^{\frac43(1+\theta)} +T \|v\|_{X(T)}^{\frac83}  \\
&\lesssim T^{1-\frac43(1+\theta)\ka}\|v\|_{X(T)}^{\frac43(1-\theta)}  \bigg(\int_0^T dt \, \|  v(t,.)\|_{H^{-2\al+\ka}_\rho}^{\frac{1}{\ka}}\bigg)^{\frac43(1+\theta)\ka} +T \|v\|_{X(T)}^{\frac83}\\
&\lesssim  T^{1-\frac43(1+\theta)\ka} \|v\|_{X(T)}^{\frac83} \, ,
\end{align*}
and thus, going back to~\eqref{appli-schro-dim-1}, we have shown the desired estimate, that is
\begin{equation*} 
\|\cg(\rho^2 |v|^2)\|_{X(T)} \lesssim  T^{\frac34-(\ka+2\al)} \|v\|_{X(T)}^{2} \, .
\end{equation*}

\

\noindent
\textbf{Bound on $\cg(\rho \overline{v} \cdot\luxo_1)$, $\cg(\rho v \cdot\overline{\luxo_1})$:}
Since $\alpha\leq \frac14$ and $\ka\leq \frac12$, we can appeal to Lemma~\ref{lem:loc-regu} to assert that
\begin{align}
\|\cg(\rho \overline{v} \cdot\luxo_1)\|_{L^{\frac{1}{\ka}}_T H^{-2\al+\ka}_\rho}+\|\cg(\rho v \cdot\overline{\luxo_1})\|_{L^{\frac{1}{\ka}}_T H^{-2\al+\ka}_\rho}& \lesssim \|\rho \overline{v} \cdot \luxo_1\|_{L^1_T H^{-2\alpha}} \nonumber\\
&\lesssim \|\rho \overline{v} \cdot \luxo_1\|_{L^1_T H^{-\alpha}}\, .\label{pro-v-psi}
\end{align}
Then, as $-2\al+\ka > \al$, we can use Lemma~\ref{lem:product} to derive that for every $t\geq 0$,
$$
\|(\rho \overline{v} \cdot \luxo_1)(t,.)\|_{H^{-\alpha}} \lesssim\| \luxo_1(t,.)\|_{\mathcal{W}^{-\alpha,\infty}}\| \rho v(t,.)\|_{H^{-2\al+\ka}} \, .
$$
By applying Lemma~\ref{lem:commut}, we obtain that for every $t\geq 0$,
\begin{equation*}
\| \rho v(t,.)\|_{H^{-2\al+\ka}}\lesssim \|v(t,.)\|_{H^{-2\al+\ka}_\rho}+\|v(t,.)\|_{H^{-2\alpha}}\, ,
\end{equation*}
and so we deduce
\begin{align*}
\|\rho \overline{v} \cdot \luxo_1\|_{L^1_T H^{-\alpha}} +\|\rho v \cdot \overline{\luxo_1}\|_{L^1_T H^{-\alpha}}&\lesssim \|\luxo_1\|_{L^\infty_T \mathcal{W}^{-\alpha,\infty}}\{T^{1-\ka}\|v\|_{L^{\frac{1}{\ka}}_T H^{-2\al+\ka}_\rho}+T\, \|v\|_{L^\infty_T H^{-2\alpha}}\}\\
&\lesssim T^{1-\ka}  \|\luxo_1\|_{L^\infty_T \mathcal{W}^{-\alpha,\infty}}\|v\|_{X(T)}\, ,
\end{align*}
which, going back to~\eqref{pro-v-psi}, leads us to
\begin{equation}\label{boun-pro-1}
\|\cg(\rho \overline{v} \cdot\luxo_1)\|_{L^{\frac{1}{\ka}}_T H^{-2\al+\ka}_\rho}+\|\cg(\rho v \cdot\overline{\luxo_1})\|_{L^{\frac{1}{\ka}}_T H^{-2\al+\ka}_\rho}\lesssim T^{1-\ka}  \|\luxo_1\|_{L^\infty_T \mathcal{W}^{-\alpha,\infty}}\|v\|_{X(T)} \, .
\end{equation}

\

On the other hand, by applying Lemma~\ref{lem:strichartz}, we get
\begin{equation*}
\|\cg(\rho \overline{v} \cdot\luxo_1)\|_{L^\infty_T H^{-2\alpha}}+\|\cg(\rho v \cdot\overline{\luxo_1})\|_{L^\infty_T H^{-2\alpha}} \lesssim \|\rho \overline{v} \cdot \luxo_1\|_{L^1_T H^{-2\alpha}} 
\lesssim \|\rho \overline{v} \cdot \luxo_1\|_{L^1_T H^{-\alpha}}\, .\label{pro-v-psi-2}
\end{equation*}
We are thus in the same position as in~\eqref{pro-v-psi}, and we can repeat the above arguments to obtain
$$
\|\cg(\rho \overline{v} \cdot\luxo_1)\|_{L^\infty_T H^{-2\alpha}}+\|\cg(\rho v \cdot\overline{\luxo_1})\|_{L^\infty_T H^{-2\alpha}}\lesssim T^{1-\ka} \|\luxo_1\|_{L^\infty_T \mathcal{W}^{-\alpha,\infty}} \|v\|_{X(T)} \, .
$$

\smallskip

Combining this bound with~\eqref{boun-pro-1}, we can conclude that
$$\|\cg(\rho \overline{v} \cdot\luxo_1)\|_{X(T)}+\|\cg(\rho v \cdot\overline{\luxo_1})\|_{X(T)}\lesssim T^{1-\ka}  \|\luxo_1\|_{L^\infty_T \mathcal{W}^{-\alpha,\infty}} \|v\|_{X(T)} \, .$$

\

\noindent
\textbf{Bound on $\cg(\cherry_1)$:}
First, according to Lemma~\ref{lem:loc-regu}, we know that
$$
\|\cg(\cherry_1)\|_{L^{\frac{1}{\ka}}_T H^{-2\al+\ka}_\rho}\lesssim  \|\cherry_1\|_{L^1_T H^{-2\alpha}}\lesssim T\, \|\cherry_1\|_{L^\infty_T H^{-2\alpha}}\, .
$$
Then, by applying Lemma~\ref{lem:strichartz}, we obtain that
$$
\|\cg(\cherry_1)\|_{L^\infty_T H^{-2\alpha}}\lesssim  \|\cherry_1\|_{L^1_T H^{-2\alpha}}\lesssim T\, \|\cherry_1\|_{L^\infty_T H^{-2\alpha}}\, ,
$$
and we have thus shown that
$$\|\cg(\cherry_1)\|_{X(T)}\lesssim T \, \|\cherry_1\|_{L^\infty_T H^{-2\alpha}}\, .$$

\

Combining the above estimates provides us with~\eqref{bound7}. It is then easy to see that~\eqref{bound8} can be derived from similar arguments.

\

\subsubsection{Proof of Proposition~\ref{last} when $d=2$}

\

\smallskip

\textit{In this situation, we pick $\ka$ such that $3\al<\ka <\frac12-2\al$ and $(p,q):=(4,4)$, so that the space under consideration becomes
$$X(T):=\mathcal{C}([0,T]; H^{-2\alpha}(\mathbb{R}^2))\cap L^4([0,T]; \cw^{-2\alpha,4}(\mathbb{R}^2))\cap L^{\frac{1}{\ka}}_T H^{-2\al+\ka}_\rho \, .$$
Note in particular that the so-defined pair $(p,q)$ is Schr{\"o}dinger admissible. Also, as in the previous section, we set $\theta:=\frac{2\al}{\ka}\in (0,\frac23)$, and we only focus on the derivation of~\eqref{bound7} (estimate~\eqref{bound8} could be obtained along the same arguments).}

\

\noindent
\textbf{Bound on $S(\phi)$:}
The arguments are exactly the same as for $d=1$ (see Section~\ref{sec:proof-dim-1}), and yield
$$\|S(\phi)\|_{X(T)}\lesssim \|\phi\|_{H^{-2\alpha}}.$$

\

\noindent
\textbf{Bound on $\cg(\rho^2 |v|^2)$:}
Since $\ka >0$, and since $\rho$ is smooth and compactly-supported, one has
\begin{align*}
\|\cg(\rho^2 |v|^2)\|_{X(T)}&=\|\cg(\rho^2 |v|^2)\|_{L^{\infty}_T H^{-2\al}}+\|\cg(\rho^2 |v|^2)\|_{L^{4}_T \cw^{-2\al,4}}+\|\cg(\rho^2 |v|^2)\|_{L^{\frac{1}{\ka}}_T H^{-2\al+\ka}_\rho}\\
&\lesssim \|\cg(\rho^2 |v|^2)\|_{L^{\infty}_T H^{-2\al+\ka}}+\|\cg(\rho^2 |v|^2)\|_{L^{4}_T \cw^{-2\al+\ka,4}}+\|\cg(\rho^2 |v|^2)\|_{L^{\frac{1}{\ka}}_T H^{-2\al+\ka}} \\
&\lesssim \| \cg(\rho^2 |v|^2)\|_{L^{\infty}_T H^{-2\al+\ka}}+\|\cg(\rho^2 |v|^2)\|_{L^{4}_T \cw^{-2\al+\ka,4}}\, ,%\label{refer-proof-dim-2}
\end{align*}
and from here we can apply Strichartz inequality (Lemma~\ref{lem:strichartz}) to assert that
\begin{equation}\label{appli-schro-dim-2}
\|\cg(\rho^2 |v|^2)\|_{X(T)}\lesssim \|\rho^2 |v|^2\|_{L^{r'}_T \mathcal{W}^{-2\al+\ka,s'}} \, 
\end{equation}
where we define the (Schr{\"o}dinger admissible) pair $(r,s)$ along the formula
$$(r,s):=(\frac{4}{1+\theta},\frac{4}{1-\theta})\, .$$
By Lemma~\ref{lem:frac-leibniz}, one has, for every fixed $t\geq 0$,
$$
\|\rho^2 |v|^2(t,.)\|_{\mathcal{W}^{-2\al+\ka,s'}}\lesssim \| \rho v(t,.)\|_{H^{-2\al+\ka}}\| \rho v(t,.)\|_{L^{\frac{4}{1+\theta}}}\, .
$$
Besides, by using Lemma~\ref{lem:interpol}, one can check that
$$\| \rho v(t,.)\|_{L^{\frac{4}{1+\theta}}}\leq  \| \rho v(t,.)\|_{H^{-2\al+\ka}}^\theta \| \rho v(t,.)\|_{\cw^{-2\al,4}}^{1-\theta} \, .$$
Therefore, for every fixed $t\geq 0$,
\begin{align*}
\|\rho^2 |v|^2(t,.)\|_{\mathcal{W}^{-2\al+\ka,s'}}&\lesssim \| \rho v(t,.)\|_{H^{-2\al+\ka}}^{1+\theta}\| \rho v(t,.)\|_{\cw^{-2\al,4}}^{1-\theta} \\
&\lesssim \|  v(t,.)\|_{H^{-2\al+\ka}_\rho}^{1+\theta}\| v(t,.)\|_{\cw^{-2\al,4}}^{1-\theta}+\| v(t,.)\|_{H^{-2\al}}^{1+\theta}\| v(t,.)\|_{\cw^{-2\al,4}}^{1-\theta}   \, ,
\end{align*}
where we have used Lemma~\ref{lem:commut} to derive the second inequality.

\smallskip

Then, taking $\la:=\frac{3-\theta}{2} > 1$, we get
\begin{align*}
&\int_0^T dt \, \|\rho^2 |v|^2(t,.)\|_{\mathcal{W}^{-2\al+\ka,s'}}^{r'}\\
&\lesssim \bigg( \int_0^T dt \, \|  v(t,.)\|_{H^{-2\al+\ka}_\rho}^{(1+\theta)r'\la}\bigg)^{\frac{1}{\la}} \bigg(\int_0^T dt \, \| v(t,.)\|_{\cw^{-2\al,4}}^{(1-\theta)r'\la'}\bigg)^{\frac{1}{\la'}}+\|v\|_{X(T)}^{r'(1+\theta)} \int_0^T dt \, \|v_t\|^{r'(1-\theta)}_{\cw^{-2\al,4}} \, .
\end{align*}
With our choices of parameters (remember that $\ka <\frac12-2\al$ and $\theta=\frac{2\al}{\ka}$), one has in fact
$$(1+\theta)r'\la=2(1+\theta)<\frac{1}{\ka} \quad \text{and} \quad (1-\theta)r'\la' =4 \, ,$$
so that the above inequality yields
\begin{align*}
&\int_0^T dt \, \|\rho^2 |v|^2(t,.)\|_{\mathcal{W}^{-2\al+\ka,s'}}^{r'}\\
&\lesssim T^{\frac{1-2\ka(1+\theta)}{\la}}\bigg( \int_0^T dt \, \|  v(t,.)\|_{H^{-2\al+\ka}_\rho}^{\frac{1}{\ka}}\bigg)^{(1+\theta)r' \ka} \|v\|_{X(T)}^{\frac{4}{\la'}}+T^{1-\frac{r'(1-\theta)}{4}}\|v\|_{X(T)}^{r'(1+\theta)} \|v\|_{X(T)}^{r'(1-\theta)} \\
&\lesssim T^{\frac{1-2\ka-4\al}{\la}}\|v\|_{X(T)}^{(1+\theta)r'+\frac{4}{\la'}}+T^{1-\frac{r'(1-\theta)}{4}}\|v\|_{X(T)}^{2r'}  \, .
\end{align*}
It is now easy to check that this estimate can be rephrased as
$$
\|\rho^2 |v|^2\|_{L^{r'}_T \mathcal{W}^{-2\al+\ka,s'}}  \lesssim \{T^{\varepsilon}+T^{\frac12}\} \|v\|_{X(T)}^2 \, ,
$$
with $\varepsilon:=\frac{1}{2} (1-2\ka-4\al)$.

\smallskip

Going back to~\eqref{appli-schro-dim-2}, we can conclude that
\begin{equation*}
\|\cg(\rho^2 |v|^2)\|_{X(T)} \lesssim \{T^{\varepsilon}+T^{\frac12}\} \|v\|_{X(T)}^2 \, .
\end{equation*}

\

\noindent
\textbf{Bound on $\cg(\rho \overline{v} \cdot\luxo_1)$, $\cg(\rho v \cdot\overline{\luxo_1})$:}
Using the very same arguments as for $d=1$, we can show first that
\begin{equation*}
\|\cg(\rho \overline{v} \cdot\luxo_1)\|_{L^{\frac{1}{\ka}}_T H^{-2\al+\ka}_\rho} +\|\cg(\rho v \cdot\overline{\luxo_1})\|_{L^{\frac{1}{\ka}}_T H^{-2\al+\ka}_\rho}\lesssim \|\rho \overline{v} \cdot \luxo_1\|_{L^1_T H^{-\alpha}} \, ,
\end{equation*}
and then
\begin{equation*}
\|\rho \overline{v} \cdot \luxo_1\|_{L^1_T H^{-\alpha}} +\|\rho v \cdot \overline{\luxo_1}\|_{L^1_T H^{-\alpha}}\lesssim T^{1-\ka}    \|\luxo_1\|_{L^\infty_T \mathcal{W}^{-\alpha,\infty}} \|v\|_{X(T)}\, .
\end{equation*}
On the other hand, we can use Lemma~\ref{lem:strichartz} to assert that
\begin{align*}
&\big(\|\cg(\rho \overline{v} \cdot\luxo_1)\|_{L^\infty_T H^{-2\alpha}}+\|\cg(\rho v \cdot\overline{\luxo_1})\|_{L^\infty_T H^{-2\alpha}}\big)+\big(\|\cg(\rho \overline{v} \cdot\luxo_1)\|_{L^4_T \cw^{-2\alpha,4}}+\|\cg(\rho v \cdot\overline{\luxo_1})\|_{L^4_T \cw^{-2\alpha,4}}\big)\\
&\lesssim \|\rho \overline{v} \cdot \luxo_1\|_{L^1_T H^{-2\alpha}} \lesssim \|\rho \overline{v} \cdot \luxo_1\|_{L^1_T H^{-\alpha}} \, .
\end{align*}
Combining the above estimates easily provides us with the desired bound
$$\|\cg(\rho \overline{v} \cdot\luxo_1)\|_{X(T)}+\|\cg(\rho v \cdot\overline{\luxo_1})\|_{X(T)}\lesssim T^{1-\ka}\|\luxo_1\|_{L^\infty_T \mathcal{W}^{-\alpha,\infty}} \|v\|_{X(T)} \, .$$

\

\

\noindent
\textbf{Bound on $\cg(\cherry_1)$:}
The arguments are exactly the same as for $d=1$ (see Section~\ref{sec:proof-dim-1}), and lead us to
$$\|\cg(\cherry_1)\|_{X(T)}\lesssim T \, \|\cherry_1\|_{L^\infty_T H^{-2\alpha}}\, .$$

\

\subsubsection{Proof of Proposition~\ref{last} when $d=3$}

\

\smallskip

\textit{In this situation, we pick $\ka:=4\al$ and $(p,q):=(2,6)$, so that one has 
$$X(T):=\mathcal{C}([0,T]; H^{-2\alpha}(\mathbb{R}^3))\cap L^2([0,T]; \cw^{-2\alpha,6}(\mathbb{R}^3))\cap L^{\frac{1}{\ka}}_T H^{-2\al+\ka}_\rho \, .$$
Observe here again that $(p,q)=(2,6)$ defines a Schr{\"o}dinger admissible pair.}

\

\noindent
\textbf{Bound on $S(\phi)$:}
We can repeat the arguments used for $d=1,2$ to assert that
$$\|S(\phi)\|_{X(T)}\lesssim \|\phi\|_{H^{-2\alpha}(\mathbb{R}^3)}.$$

\

\noindent
\textbf{Bound on $\cg(\rho^2 |v|^2)$:}
Let us first write, just as for $d=2$,
\begin{align*}
\|\cg(\rho^2 |v|^2)\|_{X(T)}&=\|\cg(\rho^2 |v|^2)\|_{L^{\infty}_T H^{-2\al}}+\|\cg(\rho^2 |v|^2)\|_{L^{2}_T \cw^{-2\al,6}}+\|\cg(\rho^2 |v|^2)\|_{L^{\frac{1}{\ka}}_T H^{-2\al+\ka}_\rho}\\
&\lesssim \| \cg(\rho^2 |v|^2)\|_{L^{\infty}_T H^{-2\al+\ka}}+\|\cg(\rho^2 |v|^2)\|_{L^{2}_T \cw^{-2\al+\ka,6}}\, ,%\label{refer-proof-dim-2}
\end{align*}
and then apply Strichartz inequality (Lemma~\ref{lem:strichartz}) to obtain
\begin{equation}\label{appli-schro-dim-3}
\|\cg(\rho^2 |v|^2)\|_{X(T)}\lesssim \|\rho^2 |v|^2\|_{L^{2}_T \mathcal{W}^{-2\al+\ka,\frac65}} \, .
\end{equation}
By Lemma~\ref{lem:frac-leibniz}, one has, for every fixed $t\geq 0$,
$$
\|\rho^2 |v|^2(t,.)\|_{\mathcal{W}^{-2\al+\ka,\frac65}}\lesssim \| \rho v(t,.)\|_{H^{-2\al+\ka}}\| \rho v(t,.)\|_{L^3}\, .
$$
Besides, by using Lemma~\ref{lem:interpol}, one can check that
$$\| \rho v(t,.)\|_{L^3}\leq  \| \rho v(t,.)\|_{H^{-2\al+\ka}}^{\frac12} \| \rho v(t,.)\|_{\cw^{-2\al,6}}^{\frac12} \, .$$
Therefore, for every fixed $t\geq 0$,
\begin{align*}
\|\rho^2 |v|^2(t,.)\|_{\mathcal{W}^{-2\al+\ka,\frac65}}&\lesssim \| \rho v(t,.)\|_{H^{-2\al+\ka}}^{\frac32}\| \rho v(t,.)\|_{\cw^{-2\al,6}}^{\frac12} \\
&\lesssim \|  v(t,.)\|_{H^{-2\al+\ka}_\rho}^{\frac32}\| v(t,.)\|_{\cw^{-2\al,6}}^{\frac12}+\| v(t,.)\|_{H^{-2\al}}^{\frac32}\| v(t,.)\|_{\cw^{-2\al,6}}^{\frac12}   \, ,
\end{align*}
where we have used Lemma~\ref{lem:commut} to derive the second inequality.

\smallskip

This entails
\begin{align*}
&\int_0^T dt \, \|\rho^2 |v|^2(t,.)\|_{\mathcal{W}^{-2\al+\ka,\frac65}}^{2}\\
&\lesssim \bigg( \int_0^T dt \, \|  v(t,.)\|_{H^{-2\al+\ka}_\rho}^{6}\bigg)^{\frac{1}{2}} \bigg(\int_0^T dt \, \| v(t,.)\|_{\cw^{-2\al,6}}^{2}\bigg)^{\frac{1}{2}}+\|v\|_{X(T)}^{3} \int_0^T dt \, \|v_t\|_{\cw^{-2\al,6}} \\
&\lesssim \{T^{\frac12(1-24\al)}+T^{\frac12}\} \|v\|_{X(T)}^{4} \, ,
\end{align*}
which can obviously be recast as
$$
\|\rho^2 |v|^2\|_{L^{2}_T \mathcal{W}^{-2\al+\ka,\frac65}}  \lesssim  T^{\frac14(1-24\al)} \|v\|_{X(T)}^2 \, .
$$
Going back to~\eqref{appli-schro-dim-3}, we can conclude that
$$
\|\cg(\rho^2 |v|^2)\|_{X(T)} \lesssim T^{\frac14(1-24\al)} \|v\|_{X(T)}^2 \, .
$$

\

\noindent
\textbf{Bound on $\cg(\rho \overline{v} \cdot\luxo_1)$, $\cg(\rho v \cdot\overline{\luxo_1})$:}
Using the very same arguments as for $d=2$ (note indeed that $-2\al+\ka=2\al>\al$), we get that
$$\|\cg(\rho \overline{v} \cdot\luxo_1)\|_{X(T)}+\|\cg(\rho v \cdot\overline{\luxo_1})\|_{X(T)}\lesssim T^{1-\ka}\|\luxo_1\|_{L^\infty_T \mathcal{W}^{-\alpha,\infty}} \|v\|_{X(T)} \, .$$

\

\noindent
\textbf{Bound on $\cg(\cherry_1)$:}
Here again, the arguments are exactly the same as for $d=1,2$ and entail
$$\|\cg(\cherry_1)\|_{X(T)}\lesssim T \, \|\cherry_1\|_{L^\infty_T H^{-2\alpha}}\, .$$

\

\subsection{Proof of Theorem~\ref{resu1}}

\

\smallskip

With the statements of Proposition~\ref{sto1} and Theorem~\ref{thm:noregular} in hand, we are of course in the very same position as in Section~\ref{subsec:proof-main-theo-regu}, and so the desired properties follow again from an elementary combination of these results.

\section{Appendix}\label{appendix}

We gather here the proofs of two technical lemmas that have been used in the analysis of the rough case, namely Lemma~\ref{lem:loc-regu} and Lemma~\ref{lem:commut}.

\subsection{Proof of Lemma~\ref{lem:loc-regu}}\label{subsec:proof-loc-reg}

\

\smallskip

The argument is based on an interpolation procedure, combined with the following result:

\begin{proposition}[Constantin-Saut~\cite{constantin-saut}] \label{prop:constantin-saut}
Fix $d\geq 1$. Let $\rho:\R^d \to \R$ be of the form~\eqref{form-rho}, $0\leq\alpha\leq \frac12$ and $0\leq T\leq 1$. Assume that $\phi \in H^{-\alpha}(\mathbb{R}^d)$, $F \in L^1([0,T];H^{-\alpha}(\mathbb{R}^d))$, and consider the solution $u$ of the following inhomogeneous Schr\"{o}dinger equation on $\mathbb{R}^d$
\begin{equation*} 
\left\{
\begin{array}{l}
\imath \partial_t u(t,x)-\Delta u(t,x)= F(t,x) \, , \quad  t\in [0,T] \, , \, x\in \R^d \, ,\\
u_0=\phi \, .
\end{array}
\right.
\end{equation*}
Then it holds that 
\begin{equation}\label{constantin-saut}
\|u\|_{L^{2}_T H^{-\al+\frac12}_\rho}\lesssim \|\phi\|_{H^{-\alpha}}+\|F\|_{L^1_T H^{-\alpha}}\, ,
\end{equation}
where the proportional constant only depends on $\rho$ and $\al$.
\end{proposition}

\smallskip

\begin{proof}[Proof of Lemma~\ref{lem:loc-regu}]
For every fixed $t\in [0,T]$, we can use  the commutator estimate~\eqref{commut-1} and the fact that $\ka\leq 1$ to write
\begin{equation}\label{proof-loc-reg}
\|u(t,.)\|_{ H^{-\al+\ka}_\rho} \lesssim \|\rho \, u(t,.)\|_{ H^{-\al+\ka}} +\|u(t,.)\|_{ H^{-\al}} \, .
\end{equation}
Then, by a standard interpolation argument (see Lemma~\ref{lem:interpol}), we get that
\begin{align*}
\|\rho \, u(t,.)\|_{ H^{-\al+\ka}}&\lesssim\|\rho \, u(t,.)\|_{ H^{-\al}}^{1-2\ka}\|\rho \, u(t,.)\|_{ H^{-\al+\frac12}}^{2\ka}\\
&\lesssim \|u(t,.)\|_{ H^{-\al}}^{1-2\ka} \|u(t,.)\|_{ H^{-\al+\frac12}_\rho}^{2\ka}+\|u(t,.)\|_{ H^{-\al}} \, ,
\end{align*}
where we have used~\eqref{commut-1} to derive the second inequality.

\smallskip

By injecting the latter bound into~\eqref{proof-loc-reg}, we deduce that
$$\|u\|_{L^{\frac{1}{\ka}}_T H^{-\al+\ka}_\rho} \lesssim \|u\|_{ L^\infty_T H^{-\al}}^{1-2\ka} \|u\|_{L^2_T H^{-\al+\frac12}_\rho}^{2\ka}+\|u\|_{ L^\infty_T H^{-\al}} \, .$$
Using estimates~\eqref{strichartz} to bound $\|u\|_{ L^\infty_T H^{-\al}}$, and then estimate~\eqref{constantin-saut} to bound $\|u\|_{L^2_T H^{-\al+\frac12}_\rho}$, we get the desired inequality~\eqref{ineq-loc-reg}.
\end{proof}

\subsection{Proof of Lemma~\ref{lem:commut}}\label{subsec:proof-commut}

\

\smallskip

In order to establish this commutator estimate, we can essentially follow the arguments of Kato and Ponce in their proof of~\cite[Lemma X1]{kato-ponce}. However, \emph{in the specific case where $\rho$ is a test function} (which is the situation we would like to handle here), the bound~\eqref{improv-boun} is clearly sharper than the general estimate in~\cite[Lemma X1]{kato-ponce}. For this reason, let us briefly review the main modifications leading to~\eqref{improv-boun}.  The bound \eqref{improv-boun} also follows from the theory of pseudo-differential operators (see e.g. \cite{pseudo}) but the proof below only relies on the  classical estimate by Coifman and Meyer (\cite{coifman-meyer}).

\begin{proposition}\label{prop:coifman-meyer}
Fix $d\geq 1$ and consider a function $\sigma \in \mathcal{C}^{\infty}((\mathbb{R}^d \times \mathbb{R}^d)\backslash(0,0);\R)$ satisfying
\begin{equation}\label{bound-der-si}
|\partial^{\alpha}_{\xi}\partial^{\beta}_{\eta}\sigma(\xi,\eta)|\leq C_{\alpha,\beta}(|\xi|+|\eta|)^{-|\alpha|-|\beta|}
\end{equation}
for all $(\xi,\eta) \neq (0,0)$ and $\alpha,\beta \in \mathbb{N}^d.$ Let us denote by $B(\si)$ the bilinear operator defined for all test functions $\vp,\psi:\R^d \to\R$ as
$$B(\si)(\vp,\psi)(x)=\int_{\mathbb{R}^d\times\mathbb{R}^d}d\xi d\eta\, e^{i \langle x, \xi + \eta \rangle}\sigma(\xi,\eta)\cf \vp(\xi)\cf \psi(\eta) \, .$$
Then it holds that
$$\|B(\si)(\vp,\psi)\|_{L^2(\mathbb{R}^d)}\lesssim \|\vp\|_{L^{\infty}(\mathbb{R}^d)} \|\psi\|_{L^2(\mathbb{R}^d)}\, ,$$
where the proportional constant only depends on the coefficients $(C_{\alpha,\beta})_{\alpha,\beta \in \mathbb{N}^d}$ in~\eqref{bound-der-si}.
\end{proposition}

\smallskip

\begin{proof}[Proof of Lemma~\ref{lem:commut}]

The quantity under consideration can be written as
\begin{align*}
&\|(\id-\Delta)^{\frac{s}{2}}(\rho\cdot g)-\rho \cdot (\id-\Delta)^{\frac{s}{2}}(g)\|_{L^2(\mathbb{R}^d)}^2\\
&=c \int_{\R^d}d\xi\, \bigg| \int_{\R^d}d\eta \, \big[ \{1+|\xi|^2\}^{\frac{s}{2}}-\{1+|\eta|^2\}^{\frac{s}{2}}\big] \cf \rho (\xi-\eta) \cf g (\eta)\bigg|^2 \, .
\end{align*}
Let us introduce a smooth function $\Phi:\R\to [0,1]$ with support in $\big[-\frac13,\frac13\big]$ such that $\Phi=1$ on $\big[-\frac14,\frac14\big]$. Then bound the above integral as
\begin{equation*}
\int_{\R^d}d\xi\, \bigg| \int_{\R^d}d\eta \, \big[ \{1+|\xi|^2\}^{\frac{s}{2}}-\{1+|\eta|^2\}^{\frac{s}{2}}\big] \cf \rho (\xi-\eta) \cf g (\eta)\bigg|^2\lesssim   \cj_1+\cj_2 \, ,
\end{equation*}
where 
$$\cj_1:=\int_{\R^d}d\xi\, \bigg| \int_{\R^d}d\eta \, (1-\Phi)\Big( \frac{|\xi-\eta|}{|\eta|} \Big)\big[ \{1+|\xi|^2\}^{\frac{s}{2}}-\{1+|\eta|^2\}^{\frac{s}{2}}\big] \cf \rho (\xi-\eta) \cf g (\eta)\bigg|^2$$
and 
$$
\cj_2:=\int_{\R^d}d\xi\, \bigg| \int_{\R^d}d\eta \, \Phi\Big( \frac{|\xi-\eta|}{|\eta|} \Big)\big[ \{1+|\xi|^2\}^{\frac{s}{2}}-\{1+|\eta|^2\}^{\frac{s}{2}}\big] \cf \rho (\xi-\eta) \cf g (\eta)\bigg|^2 \, .
$$

\

\noindent
\textbf{Bound on $\cj_1$.} Using Cauchy-Schwarz inequality, we get first
\small
\begin{equation}\label{pro-comm-cs}
\cj_1 \leq  \|g\|_{H^{s-1}(\R^d)}^2 \, \bigg(\int_{\R^d\times \R^d}d\xi d\eta \, \big|(1-\Phi)\Big( \frac{|\xi-\eta|}{|\eta|} \Big)\big|^2\big| \cf \rho (\xi-\eta)\big|^2 \frac{\big| \{1+|\xi|^2\}^{\frac{s}{2}}-\{1+|\eta|^2\}^{\frac{s}{2}}\big|^2}{\{1+|\eta|^2\}^{s-1}} \bigg) \, . 
\end{equation} 
\normalsize
In order to show that the latter integral is indeed finite, observe that if $(1-\Phi)\big( |\xi-\eta|/|\eta| \big)\neq 0$, then $|\xi-\eta|\geq \frac14 |\eta|$, and so $|\xi-\eta|\geq \frac15 |\xi|$. Therefore, as $\rho$ is smooth and compactly-supported, one has, for all $\la,\be \geq 0$,
$$\big|(1-\Phi)\Big( \frac{|\xi-\eta|}{|\eta|} \Big)\big|^2\big| \cf \rho (\xi-\eta)\big|^2 \leq c_{\rho,\la,\be}\{1+|\xi|^2\}^{-\la}\{1+|\eta|^2\}^{-\be} \, ,$$
and the finiteness of the integral in~\eqref{pro-comm-cs} immediately follows.

\

\noindent
\textbf{Bound on $\cj_2$.} By Fourier isometry, we can write this quantity as $\cj_2=c\, \big\| F\|_{L^2(\R^d)}^2$, with 
$$F(x):=\int_{\R^d \times \R^d}d\xi d\eta \, e^{\imath \langle x,\xi+\eta\rangle} \Phi\Big(\frac{|\xi|}{|\eta|}\Big) \big[ \{1+|\xi+\eta|^2\}^{\frac{s}{2}}-\{1+|\eta|^2\}^{\frac{s}{2}}\big] \cf \rho (\xi) \cf g (\eta) \, .$$
Then
\begin{align*}
&\Phi\Big(\frac{|\xi|}{|\eta|}\Big) \big[ \{1+|\xi+\eta|^2\}^{\frac{s}{2}}-\{1+|\eta|^2\}^{\frac{s}{2}}\big] \cf \rho (\xi) \cf g (\eta)\\
&=\Phi\Big(\frac{|\xi|}{|\eta|}\Big) \{1+|\eta|^2\}^{\frac{s}{2}}\bigg[ \bigg(1+\frac{\langle \xi,\xi+2\eta\rangle}{1+|\eta|^2}\bigg)^{\frac{s}{2}}-1\bigg] \cf \rho (\xi) \cf g (\eta)\\
&=\Phi\Big(\frac{|\xi|}{|\eta|}\Big) \{1+|\eta|^2\}^{\frac{1}{2}}\bigg[ \bigg(1+\frac{\langle \xi,\xi+2\eta\rangle}{1+|\eta|^2}\bigg)^{\frac{s}{2}}-1\bigg] \cf \rho (\xi) \cf\big((\id-\Delta)^{\frac{s-1}{2}} g\big) (\eta) \, .
\end{align*}
At this point, observe that if $\Phi\Big(\frac{|\xi|}{|\eta|}\Big)\neq 0$, then $|\xi|\leq \frac13 |\eta|$, and so $|\langle \xi,\xi+2\eta\rangle | \leq \frac79 |\eta|^2$. Therefore, we can rely on the pointwise expansion
$$\Phi\Big(\frac{|\xi|}{|\eta|}\Big) \{1+|\eta|^2\}^{\frac{1}{2}}\bigg[ \bigg(1+\frac{\langle \xi,\xi+2\eta\rangle}{1+|\eta|^2}\bigg)^{\frac{s}{2}}-1\bigg] =\sum_{k\geq 1}a_k(s) \Phi\Big(\frac{|\xi|}{|\eta|}\Big)\{1+|\eta|^2\}^{\frac{1}{2}-k}\langle \xi,\xi+2\eta\rangle^k \, ,$$
where $a_k(s):=\frac{s(s-1)\cdots (s-k+1)}{k!}$.

\smallskip

Since $s>0$ and $\rho,g$ are assumed to be test functions, one has
\begin{align*}
&\int_{\R^d \times \R^d}d\xi d\eta \, \sum_{k\geq 1}\Big| a_k(s) \Phi\Big(\frac{|\xi|}{|\eta|}\Big)\{1+|\eta|^2\}^{\frac{1}{2}-k}\langle \xi,\xi+2\eta\rangle^k \cf \rho (\xi) \cf\big((\id-\Delta)^{\frac{s-1}{2}} g\big) (\eta) \Big|\\
&\lesssim \bigg(\sum_{k\geq 1}\big| a_k(s)\big|\bigg) \int_{\R^d \times \R^d}d\xi d\eta \, \{1+|\eta|^2\}^{\frac{1}{2}}  \big|\cf \rho (\xi)\big| \big|\cf\big((\id-\Delta)^{\frac{s-1}{2}} g\big) (\eta)\big| \ < \ \infty \, ,
\end{align*}
and accordingly we can write
\small
\begin{align*}
&F(x)=\sum_{k\geq 1}a_k(s) \int_{\R^d \times \R^d}d\xi d\eta \, e^{\imath \langle x,\xi+\eta\rangle}\Phi\Big(\frac{|\xi|}{|\eta|}\Big)\{1+|\eta|^2\}^{\frac{1}{2}-k}\langle \xi,\xi+2\eta\rangle^k   \cf \rho (\xi) \cf\big((\id-\Delta)^{\frac{s-1}{2}} g\big) (\eta)\\
&=\sum_{k\geq 1}a_k(s)\sum_{i=1}^d  \int_{\R^d \times \R^d}d\xi d\eta \,e^{\imath \langle x,\xi+\eta\rangle} \Phi\Big(\frac{|\xi|}{|\eta|}\Big)\{1+|\eta|^2\}^{\frac{1}{2}-k}\langle \xi,\xi+2\eta\rangle^{k-1}(\xi_i+2\eta_i) \big[\xi_i\cf \rho (\xi)\big] \cf\big((\id-\Delta)^{\frac{s-1}{2}} g\big) (\eta)\, .
\end{align*}
\normalsize
Using the notation of Proposition~\ref{prop:coifman-meyer} and the fact that $\xi_i\cf \rho (\xi)=\imath \cf\big(\partial_{x_i}\rho\big) (\xi)$, the latter identity can be rephrased as
\begin{equation*}
F(x)=\imath \sum_{k\geq 1}a_k(s)\sum_{i=1}^d  B(\si_{k,i})\big(\partial_{x_i}\rho,(\id-\Delta)^{\frac{s-1}{2}} g\big) \, ,
\end{equation*}
with
$$\si_{k,i}(\xi,\eta):=\Phi\Big(\frac{|\xi|}{|\eta|}\Big)\{1+|\eta|^2\}^{\frac{1}{2}-k}\langle \xi,\xi+2\eta\rangle^{k-1}(\xi_i+2\eta_i) \, .$$

\smallskip

It is not hard to check that for all $k\geq 1$ and $1\leq i\leq d$, the function $\si_{k,i}$ satisfies condition~\eqref{bound-der-si} with coefficients $(C_{\alpha,\beta})_{\alpha,\beta \in \mathbb{N}^d}$ independent of $k$ and $i$. Consequently, we are in a position to apply Proposition~\ref{prop:coifman-meyer} and conclude that
$$\|F\|_{L^2(\R^d)} \lesssim \sum_{k\geq 1}\big|a_k(s)\big|  \sum_{i=1}^d \|\partial_{x_i}\rho\|_{L^\infty(\R^d)} \|g\|_{H^{s-1}(\R^d)} \lesssim \|g\|_{H^{s-1}(\R^d)} \, .$$

\end{proof}

%\begin{remark}\label{rk:altern-proof}
%The statement of Lemma~\ref{lem:commut} could also be regarded (and proved) in the light of the theory of pseudo-differential operators. For instance, the result can be seen as an application of~\cite[Corollary I.4.1]{pseudo}. \tbc{However, a full description of this approach would require us to introduce additional cumbersome definitions, and therefore we have preferred to stick to the above (essentially self-contained) presentation.}
%\end{remark}

%%%%%%%%%%%%%%%%%%%%%%%%%%%%%%%%%%%%%%%%%%%%%%%%%%%%%%%%%%%%%%%%%%%%%%%%%%

\end{document}